\documentclass[11pt,reqno,a4paper]{article}
\usepackage{lmodern} 
\usepackage[T1]{fontenc} 
\usepackage{float}
\usepackage{comment}
\usepackage{tabularx}
\usepackage{adjustbox}
\usepackage{indentfirst}
 
\usepackage{enumerate}

\usepackage{amsmath,amssymb,amsthm,amsfonts,amscd,authblk,dsfont,cancel,diagbox,
enumerate,epsfig,eucal,fancyhdr,latexsym,lmodern,mathrsfs,mathtools,multirow,
musicography,skull,simplewick,subcaption,tikz-cd}
\usetikzlibrary{matrix,arrows,decorations.pathmorphing}
 \usepackage{relsize}
\usepackage[normalem]{ulem} 

\usepackage{hyperref}
\hypersetup{
pdftitle={},%
pdfauthor={},%
pdfsubject={},%
pdfkeywords={},%
colorlinks=true,%
linkcolor=black,%
urlcolor=black,
linktoc={},
linktocpage={false},%
pageanchor={},
citecolor={black},
}

\usepackage[english]{babel}
\usepackage[utf8]{inputenc}
\usepackage[margin=2.35cm]{geometry}

\usepackage{url}

\usepackage{cite}

\usepackage{titlesec}
\titleformat{\section}{\large\bfseries\filcenter}{\thesection}{1em}{}
\titleformat{\subsection}{\bfseries}{\thesubsection}{1em}{}
\let\oldbibliography\thebibliography
\renewcommand{\thebibliography}[1]{\oldbibliography{#1}
\setlength{\itemsep}{0\baselineskip}}
\makeatletter
\renewcommand\@makefntext[1]{\leftskip=2em\hskip-0.5em\@makefnmark#1}
\makeatother
\newtheorem{theorem}{Theorem}[section]
\newtheorem{theorem-intro}{Theorem}[]

\newtheorem{corollary}[theorem]{Corollary}
\newtheorem{lemma}[theorem]{Lemma}
\newtheorem{proposition}[theorem]{Proposition}

\theoremstyle{definition}
\newtheorem{remark}[theorem]{Remark}

\newtheorem{definition}[theorem]{Definition}

\newtheorem{setup}[theorem]{Setup}

\numberwithin{equation}{section}
\allowdisplaybreaks[1]

\catcode`@=12

\renewcommand\thanks[1]{%
  \begingroup
  \renewcommand\thefootnote{}\footnote{#1}%
  \addtocounter{footnote}{-1}%
  \endgroup
}
\renewcommand{\:}{\colon} 
\newcommand{\defeq}{\vcentcolon=} 
\newcommand{\verteq}{\rotatebox[origin=c]{90}{$\mkern1mu=$}}
\renewcommand{\i}{\mathrm{i}} 
\newcommand{\mat}[4]{\begin{pmatrix}#1 &#2  \\ #3 &#4 \end{pmatrix}}

\newcommand{\col}[2]{\begin{pmatrix}#1 \\#2\end{pmatrix}}

\renewcommand{\epsilon}{\varepsilon}

\newcommand{\NN}{\mathbb{N}}

\newcommand{\RR}{\mathbb{R}}
\newcommand{\CC}{\mathbb{C}}

\newcommand{\A}{\mathsf{A}}
\newcommand{\B}{\mathsf{B}}
\newcommand{\C}{\mathsf{C}}
\newcommand{\E}{\mathsf{E}}
\renewcommand{\L}{\mathsf{L}}
\newcommand{\F}{\mathsf{F}}
\newcommand{\G}{\mathsf{G}}
\renewcommand{\H}{\mathsf{H}}
\renewcommand{\P}{\mathsf{P}}
\newcommand{\K}{\mathsf{K}}
\newcommand{\U}{\mathcal{U}}

\newcommand{\R}{\mathsf{R}}
\newcommand{\V}{\mathsf{V}}

\newcommand{\Id}{\mathds{1}}

\renewcommand{\d}{\mathrm{d}}
\newcommand{\Ric}{{\mathrm{Ric}}}
\newcommand{\D}{\square}

\newcommand{\DeltakS}{\Delta_{k}}
\newcommand{\DeltaZS}{\Delta_{0}}
\newcommand{\DeltaOS}{\Delta_{1}}

\def\p{\partial}

\newcommand{\PS}{\mathcal{V}_{\mathrm{P}}} 
\newcommand{\PSS}{\mathcal{V}_{\Sigma}} 
\newcommand{\PSR}{\mathcal{V}_{\mathrm{R}}} 

\newcommand{\ran}{\textnormal{ran}}
\renewcommand{\ker}{\textnormal{ker}}

\newcommand{\WF}{\textnormal{WF}}
\newcommand{\CCR}{\mathrm{CCR}(\PS,\q_1)} 
\newcommand{\BS}{\mathrm{B}\mathcal{S}_{\mathrm{cl}}}
\renewcommand{\sc}{\mathrm{sc}}
\newcommand{\tc}{\mathrm{tc}}
\newcommand{\fc}{\mathrm{fc}}
\newcommand{\pc}{\mathrm{pc}}
\renewcommand{\c}{\mathrm{c}}
\renewcommand{\b}{\mathrm{b}}
\newcommand{\q}{\mathrm{q}}

\newcommand{\f}{\mathfrak{f}}
\newcommand{\g}{\mathfrak{g}}
\newcommand{\M}{\mathsf{M}}	
\newcommand{\T}{\mathsf{T}}	
\newcommand{\J}{\mathcal{J}} 
\newcommand{\vol}{{\textnormal{vol}_{g}\,}} 
\newcommand{\volS}{{\textnormal{vol}_{\Sigma}\,}} 
\usepackage{tikz-cd}

\tikzset{%
  symbol/.style={
    draw=none,
    every to/.append style={
      edge node={node [sloped, allow upside down, auto=false]{$#1$}}
    },
  },
}
\begin{document}
%
%
%
%
\begin{center}
\vspace*{-1cm}

{\Large\bf  The Quantization of Maxwell Theory\\[2mm] in the Cauchy Radiation Gauge:\\[3mm] Hodge Decomposition and Hadamard States} 

\vspace{5mm}

{\bf by}

\vspace{3mm}
\noindent
{  \bf  Simone Murro and Gabriel Schmid}\\[2mm]
\noindent   {\it Dipartimento di Matematica,  Universit\`a di Genova \& INdAM \& INFN, Italy}\\[2mm]

Emails: \ {\tt  murro@dima.unige.it, gabriel.schmid@dima.unige.it}
\\[10mm]
\end{center}
\begin{abstract}
The aim of this paper is to prove the existence of Hadamard states for the Maxwell equations on any globally hyperbolic spacetime. This will be achieved by introducing a new gauge fixing condition, the {\it Cauchy radiation gauge}, that will allow to suppress all the unphysical degrees of freedom. The key ingredient for achieving this gauge is a new Hodge decomposition for differential $k$-forms in Sobolev spaces on complete (possibly non-compact) Riemannian manifolds.
\end{abstract}
\paragraph*{Keywords:} Hodge decomposition on complete Riemannian manifolds, Hadamard states for Maxwell fields, Cauchy radiation gauge, quantum field theory on curved spacetimes.
\paragraph*{MSC 2020:} Primary: 35Q61, 47F05, 81T20; Secondary: 35J05, 58J45.
\tableofcontents
\section{Introduction}
The so-called {\it algebraic approach to
quantum field theory} is a very successful quantization scheme for describing quantum fields propagating on globally hyperbolic spacetimes~\cite{AQFT1,Gerard2019} and it stands at the forefront of scientific research. In this framework, the quantization is interpreted as a two-step procedure:
\begin{itemize}
\item[1.] The first consists of the assignment of a $*$-algebra of observables, which encodes structural properties such as causality, dynamics and canonical commutation relations, to a physical system.
\item[2.] The second step calls for the identification of a physical state, which is a positive,
linear and normalized functional on the algebra of observables, satisfying the so-called {\it Hadamard condition}. 
\end{itemize}
The Hadamard condition ensures the correct short-distance
behaviour of the n-point functions of the field and plays a key role in the perturbative approach to quantum field theory~\cite{Kasia}. Indeed, it implies the finiteness of the quantum fluctuations of the expectation value of every observable~\cite{FV} and it allows to constructing Wick polynomials and other nonlinear observables~\cite{VI}, like the stress-energy tensor.
For a complex scalar field~\cite{Christian1, GerardWrochnaOulghazi2017}, whose dynamics is ruled by a normally hyperbolic operator~$\D$, finding a quasifree Hadamard state amounts to constructing a pair of linear operators $\lambda^\pm\:\,C^\infty_\c(\M)\to C^\infty(\M)$ that satisfies the following properties:
\begin{itemize}
\item[(i)] \quad $(\lambda^{\pm})^{\ast}=\lambda^{\pm}$  with respect to a positive Hermitian form $(\cdot,\cdot)_{\M}$ on $C^{\infty}_{\c}(\M)$\,;
\item[(ii)]\quad $\D\lambda^{\pm}=\lambda^{\pm}\D=0$   and $\i(\lambda^{+}-\lambda^{-})$ is the causal propagator of $\D$\,;
\item[(iii)]\quad $\lambda^{\pm}\geq 0$ w.r.t.~$(\cdot,\cdot)_{\M}$\,;
\item[(iv)]\quad  $\WF^{\prime}(\lambda^{\pm})\subset\mathcal{N}^{\pm}\times\mathcal{N}^{\pm}\, $.\end{itemize}
Here $\WF^{\prime}(\lambda^{\pm})$ denotes the primed wavefront set of the Schwartz kernel of $\lambda^\pm$ and 
$\mathcal{N}^{\pm}$ are the two connected components of the lightcone $\mathcal{N}\defeq \{(p,\xi)\in \T^*\M \,|\, g^{-1}(\xi,\xi)=0 \}$.  \\

In order to generalize the construction of Hadamard states to Maxwell theory, three main difficulties have to be tackled:
\begin{enumerate}
\item The dynamics is ruled by a differential operator $\P$ that is not hyperbolic but is invariant under a gauge transformation;
\item The operator $\P$ is formally self-adjoint w.r.t.~a Hermitian product that is typically non-positive on the fibers of the vector bundle;
\item In most globally hyperbolic spacetimes, the theory is affected by infrared problems.
\end{enumerate}
Let us explain these problems in more details for Maxwell theory. Let $(\M,g)$ be a globally hyperbolic spacetime and consider the two differential operators 
$$ \P=\delta \d \qquad\text{and}\qquad \K = \d$$
acting on 1-forms, where $\delta$ is the codifferential and $\d$ the exterior derivative on $(\M,g)$. The operator $\P$ is formally self-adjoint with respect to the Hodge inner product
\begin{align*}
	(\cdot,\cdot)\defeq \int_{\M}\,g^{-1}(\cdot,\cdot)\,\vol 
	\end{align*}
which is not positive definite since $g^{-1}$ is Lorentzian. The equation $\P A=0$ for a 1-form $A$ are the Maxwell equations. Furthermore, $\P\K=0$, which encodes the fact that $\P$ is invariant under linear gauge transformations $A \mapsto A + \K \omega$ and it is responsible for the fact that $\P$ is not hyperbolic. The solutions of the Maxwell equations $\P A=0$ are obtained by solving $\D A =0$  subject to the Lorenz gauge condition $\delta A=0$ and the remaining gauge freedom is parametrized by $\omega\in\ker(\square)$, 
$$ \frac{\ker(\P)}{\ran(\d)} \simeq \frac{\ker(\D)\cap \ker(\delta)}{\d(\ker(\square))} \,.$$
 A first naive idea would be to construct Hadamard states for the hyperbolic theory (without constraints) and then to restrict it to Maxwell theory. However, since the  fiber metric $g^{-1}$ is not positive definite, Hadamard states for $\D$ do not exist due to the failure of the positivity condition (iii) (see e.g.~\cite[Section 6.3]{GerardWrochnaYM2015}). 
This means, that positivity can be only achieved in the quotient space. 
Finally let us comment on the last difficulty: the infrared problems. On ultrastatic spacetime, for a massive theory, like for Proca fields~\cite{MorettiMurroVolpe2022}, the construction of a Hadamard state is reduced to the construction of projection operators onto the subspace of positive frequency solutions. This is in particular achieved once that the operator $\sqrt{\Delta+m^2}$ and its inverse are well-defined. For the Maxwell theory, the Hodge-Laplacian $\Delta$ is in general not positive definite, therefore it is not possible to define an inverse. 
 \medskip

In this paper we tackle these difficulties by fully gauge fixing Maxwell theory and using pseudodifferential calculus to deal with infrared problems. This is the main result of this paper.
\begin{theorem}\label{thm:main}
Let $\P$ be the Maxwell operator on a globally hyperbolic manifold $(\M, g)$ and denote by $\CCR$ the CCR-algebra  associated to gauge-invariant and compactly supported observables for the Maxwell fields.
Then there exists a Hadamard state on $\CCR$.
\end{theorem}
\subsection{Structure of the Paper and Main Results }
The main idea of the paper is to fix completely  all the gauge degrees of freedom. This will be achieved by working in the so-called {\em Cauchy radiation gauge}, namely we shall only consider those solutions of the Maxwell equation that satisfy the conditions
$$ \D A =0 \qquad \delta A =0 \qquad A_0|_{\Sigma}=(\nabla_{\partial_t} A_0)|_\Sigma =0$$ 
where $A_0 \defeq (\partial_t \lrcorner A)$ and $\nabla_{\partial_{t}}A_0 \defeq (\partial_t \lrcorner \nabla_{\partial_{t}}A)$.
As explained in Section~\ref{Sec:Maxwell}, on globally hyperbolic manifolds $\M = \RR \times \Sigma$ that are ultrastatic, this gauge is equivalent to the {\it radiation gauge}, i.e.~$\delta A=0$ and $A_0=0$ on $\M$, and, if $\Sigma$ is not compact, it is equivalent to the {\it Coulomb gauge} $\delta_\Sigma A_\Sigma =0$ where $A_\Sigma\defeq A - A_0 \d t$. 
In order to show that the Cauchy radiation gauge can be always achieved, we will need to
solve the Poisson equation on complete Riemannian manifolds. To achieve our goal, we shall generalize the Hodge decomposition to non-compact, complete Riemannian manifolds. This will be arranged in  Section~\ref{sec:Hodge} by introducing suitable Sobolev spaces for $k$-forms. 
This is the second main result of the paper.

\begin{theorem}[Hodge decomposition on non-compact manifolds]\label{Thm:HodgeDecomSmooth}
Let $(\Sigma,h)$ be a complete Riemannian manifold, denote by $\d_{\Sigma}$, $\delta_{\Sigma}$ and $\Delta_{k}$ the exterior derivative, codifferential and Laplacian on $\Sigma$, and by $\H^{s}_{k}(\Sigma)$ with $s\in [0,\infty)$ the Sobolev space of $k$-forms as in Definition~\ref{def:sob space}. We define 
	\begin{align*}
		\Omega^{k}_{s}(\Sigma) &\defeq  \Omega^{k}(\Sigma;\CC)\cap \H^{s}_{k}(\Sigma)\,,\qquad\qquad\qquad\Omega_{\infty}^{k}(\Sigma)\defeq\bigcap_{s\in\RR}\Omega_{s}^{k}(\Sigma)\, ,\\
		\Omega^{k}_{s,\d}(\Sigma) &\defeq \Omega^{k}(\Sigma;\CC)\cap\overline{\d_{\Sigma}\Omega^{k-1}_{\infty}(\Sigma)}^{\Vert\cdot\Vert_{\H^{s}}},\qquad\Omega_{\infty,\d}^{k}(\Sigma)\defeq\bigcap_{s\in\RR}\Omega_{s,\d}^{k}(\Sigma)\, ,\\
		\Omega^{k}_{s,\delta}(\Sigma) &\defeq \Omega^{k}(\Sigma;\CC)\cap\overline{\delta_{\Sigma}\Omega^{k+1}_{\infty}(\Sigma)}^{\Vert\cdot\Vert_{\H^{s}}},\qquad\Omega_{\infty,\delta}^{k}(\Sigma)\defeq\bigcap_{s\in\RR}\Omega_{s,\delta}^{k}(\Sigma)\, .
	\end{align*}
	The space $\Omega^{k}_{s}(\Sigma)$ with $s\in [0,\infty]$ admits the following $\H^{s}$-orthogonal decomposition:
	\begin{align*}
		\Omega^{k}_{s}(\Sigma)\cong\Omega^{k}_{s,\d}(\Sigma)\oplus\Omega^{k}_{s,\delta}(\Sigma)\oplus \mathrm{ker}(\Delta_{k}\vert_{\Omega^{k}_{s}})\,.
	\end{align*}
	Moreover, for any $\omega\in\Omega^{k}_{s}(\Sigma)$, which we uniquely decompose as
$\omega=\alpha+\beta+\gamma$	
	with $\alpha\in\Omega^{k}_{s,\d}(\Sigma)$, $\beta\in\Omega^{k}_{s,\delta}(\Sigma)$ and $\gamma\in\mathrm{ker}(\Delta_{k}\vert_{\Omega^{k}_{s}})$, the following is true:
\begin{itemize}
	\item[(i)]$\alpha\in\Omega^{k}_{s,\d}(\Sigma)$ is exact, i.e.~$\alpha=\d_{\Sigma}\psi$ for some $\psi\in\Omega^{k-1}(\Sigma;\CC)$.
	\item[(ii)]$\beta\in\Omega^{k}_{s,\delta}(\Sigma)$ is coexact, i.e.~$\alpha=\delta_{\Sigma}\eta$ for some $\eta\in\Omega^{k+1}(\Sigma;\CC)$.
	\item[(iii)]$\d_{\Sigma}\omega=0$ if and only if $\beta=0$ and hence $\mathrm{ker}(\d_{\Sigma}\vert_{\Omega^{k}_{s}})=\Omega^{k}_{s,\d}(\Sigma)\oplus\mathrm{ker}(\Delta_{k}\vert_{\Omega^{k}_{s}})$.
	\item[(iv)]$\delta_{\Sigma}\omega=0$ if and only if $\alpha=0$ and hence $\mathrm{ker}(\delta_{\Sigma}\vert_{\Omega^{k}_{s}})=\Omega^{k}_{s,\delta}(\Sigma)\oplus\mathrm{ker}(\Delta_{k}\vert_{\Omega^{k}_{s}})$.
\end{itemize}	
\end{theorem}

The achievability of the Cauchy radiation gauge will be discussed in Section~\ref{Sec:PhaseSpaceQuant}, building on the results of Section~\ref{Sec:CauchyProblem} where the Cauchy problem for the wave operator on forms with smooth-Sobolev data will be studied. In Section~\ref{Sec:ConstructionHadamardFINAL} the construction of Hadamard states in the Cauchy radiation gauge in ultrastatic globally hyperbolic manifolds will be performed. The main benefit of working in the Cauchy radiation gauge is that the fiber metric becomes 
manifestly positive definite.
Using the notation introduced in Section~\ref{Sec:ConstructionHadamardFINAL}, this is our third main result of the paper.

\begin{theorem}[Hadamard projectors]\label{thm:hada}
 Let $(\M,g)$ be a globally hyperbolic, ultrastatic manifold with a Cauchy surface~$(\Sigma,h)$ of bounded geometry and set 
 \begin{align*}
	\Gamma_{\infty}(\V_{\rho_{1}})&\defeq \Omega^{0}_{\infty}(\Sigma_{t})\oplus\Omega^{0}_{\infty}(\Sigma_{t})\oplus \Omega^{1}_{\infty}(\Sigma_{t})\oplus\Omega^{1}_{\infty}(\Sigma_{t})\\
	\Gamma_{\infty,\d}(\V_{\rho_{1}})&\defeq \Omega^{0}_{\infty}(\Sigma_{t})\oplus\Omega^{0}_{\infty}(\Sigma_{t})\oplus \Omega^{1}_{\infty,\d}(\Sigma_{t})\oplus\Omega^{1}_{\infty,\d}(\Sigma_{t})
	\,.
\end{align*}
 for some $t\in\RR$. Consider the projection operator $\T_\Sigma$ defined in  Proposition~\ref{Prop:ProjOp} as 
\begin{align*}
		\T_{\Sigma}=\begin{pmatrix}
	0&0&0&0\\0&0&0&0\\0 & 0 & \mathds{1}-\d_{\Sigma}\Delta_{1}^{-1}\delta_{\Sigma} & 0\\ 0&0 &0& \mathds{1}-\d_{\Sigma}\Delta_{1}^{-1}\delta_{\Sigma}
	\end{pmatrix}\,.
	\end{align*}  
(see also Lemma~\ref{Lemma:ProjectionsHodgeDecomp} for the invertibility of $\Delta_1$).  Let $\epsilon_k$ be an approximate square root of the Hodge-Laplacian $\Delta_k$ as in Lemma~\ref{lem:commuting} and define  
 $\pi^{\pm}\:\Gamma_\infty(\V_{\rho_{1}})\to\Gamma_\infty(\V_{\rho_{1}})$  by
	\begin{align*}
		\pi^{\pm}\defeq\frac{1}{2}\begin{pmatrix}\mathds{1} & \pm\varepsilon_{0}^{-1} &0&0\\\pm\varepsilon_{0} &\mathds{1}&0&0\\ 0&0&\mathds{1} & \pm\varepsilon^{-1}_{1}\\0&0&\pm\varepsilon_{1} &\mathds{1}\end{pmatrix}\, .
	\end{align*}
 	The operators $c^{\pm}\:\Gamma_\infty(\V_{\rho_{1}})\to\Gamma_\infty(\V_{\rho_{1}})$ defined by 
$$c^\pm\defeq \T_{\Sigma}\pi^{\pm}\T_{\Sigma}$$ 
	 have the following properties: 
	\begin{itemize}
		\item[(i)]\quad $(c^{\pm})^{\dagger}=c^{\pm}$\, and \, $ c^\pm(\ran(\K_{\Sigma})\cap \Gamma_{\infty,\d})\subset\ran(\K_{\Sigma})$;
		\item[(ii)]\quad $(c^++c^-)\f=\f$ \, modulo\, $\ran(\K_{\Sigma})\cap \Gamma_{\infty,\d}$\, for any\, $\f\in\ker(\K_{\Sigma}^{\dagger}\vert_{\Gamma_\infty})$;
		\item[(iii)]\quad $\pm \q_{1,\Sigma}(\f,c^{\pm}\f)\geq 0$ \, for any \, $\f\in\ker(\K_{\Sigma}^{\dagger}\vert_{\Gamma_\infty})$;
	\item[(iv)]\quad$ \WF'(\mathcal{U}_{1} c_{1}^{\pm})\subset (\mathcal{N}^{\pm}\cup F)\times\T^{*}\Sigma$ \, for $F=\{k=0\}\subset\T^*\M$.
	\end{itemize}
Therefore,  the operators $\lambda^{\pm}\: \ker(\K^{\ast}\vert_{\Gamma_{\infty}})\to \Gamma_{\infty}(\V_1)$ defined by
$$\lambda^{\pm}\defeq \pm \, \i\, (\rho_{1}\G_{1})^{\ast}(\G_{1,\Sigma}c^{\pm})(\rho_{1}\G_{1}) $$ are the pseudo-covariances of a quasifree Hadamard state in the Cauchy radiation gauge.
\end{theorem}

We conclude our paper with Section~\ref{Sec:Finale}, where the proof of Theorem~\ref{thm:main} is presented. The proof is based on Theorem \ref{thm:hada}, namely by suitably restricting the covariances $c^{\pm}$ to $\Gamma_{\c}(\V_{\rho_{1}})$, and the existence of Hadamard states on generic globally hyperbolic manifolds will follow from a deformation argument.
\subsection*{Bibliographical Remarks}
The existence of Hadamard states for the Maxwell equations was already investigated under additional assumption on the Cauchy hypersurface $\Sigma$ of the spacetime: Fewster-Pfenning~\cite{FP} (generalizing results of Furlani~\cite{Fur}) assumed $\Sigma$ to be compact and with vanishing first cohomology group, Finster-Strohmaier~\cite{Felix}, extending the Gupta–Bleuler formalism, considered only $\Sigma$  subjected to an ‘absence of zero resonances’ condition for the Hodge Laplacian on 1-forms, while Dappiaggi-Siemssen~\cite{DappSiem2013} worked out the construction on asymptotically flat spacetimes. 

Later, the construction of Hadamard
states was generalized to the linearized Yang-Mills equations: Hollands~\cite{Hollands} considered the theory linearized around the zero solution on globally hyperbolic spacetimes with compact Cauchy surface and
vanishing first cohomology group, while Gérard-Wrochna~\cite{GerardWrochnaYM2015} linearized around non-zero solutions on globally hyperbolic spacetimes with Cauchy surface either compact and parallelizable or $\RR^n$ with a Riemannian metric $h$ satisfying 
$$c^{-1}  \leq [h_{ij}(x)] \leq c \, \quad \text{for } c > 0  \qquad \text{ and } \qquad |\partial_x^\alpha h_{ij} (x)| \leq C_\alpha \quad \forall \alpha \in \mathbb{N}^n \quad x \in \RR^n\,.$$

The case of the linearized Einstein equations is more problematic and a full proof of the existence of Hadamard state on globally hyperbolic manifolds is still missing. Benini-Dappiaggi-Murro~\cite{BDM} considered asymptotically flat spacetimes with methods drawing from earlier works of Ashtekar-Magnon-Ashtekar \cite{Ashtekar} and Dappiaggi-Moretti-Pinamonti~\cite{DMP}. The quantization turns out, however, to be limited to a subspace of classical degrees of freedom due to divergences at null
infinity. 
Gérard-Murro-Wrochna~\cite{GerardMurroWrochna2022} investigated the construction of Hadamard states for linearized gravity on analytic spacetimes using Wick rotation. However, the gauge invariance and positivity of the two-point functions are only obtained modulo the addition of some smooth corrections. Only very recently, Gérard~\cite{Gerard} managed to construct Hadamard states on Cauchy-compact globally hyperbolic spacetimes via a complete gauge fixing.
\subsection{General Notation and Conventions}
\begin{itemize}
\item[$\bullet$] $(\M=\RR\times\Sigma, g=-\beta^{2}\d t^2+h_t)$ denotes a globally hyperbolic manifold.
\item[$\bullet$] Given a vector bundle $\E$ over $\M$, we denote
by $\Gamma(\E)$ the linear space of smooth sections of $\E$ and by $\Gamma_\c(\E)$, $\Gamma_{\sc}(\E)$ and $\Gamma_{\tc}(\E)$ the subspace of compactly, resp.~spatially compactly, resp.~temporally compactly supported smooth sections of $\E$. 
\item[$\bullet$]A complex vector bundle $\E$ over $\M$ together with a non-degenerate Hermitian bundle metric $\langle\cdot,\cdot\rangle_{\E}\in\Gamma(\E^{\ast}\otimes\E^{\ast})$ will be referred to as a Hermitian bundle. We denote by 
\begin{align*}
	(\cdot,\cdot)_{\E}\defeq \int_{\M}\langle\cdot,\cdot\rangle_{\E}\,\mathrm{vol}_{g}
\end{align*}
the induced non-degenerate Hermitian form on $\Gamma(\E)$.
\item[$\bullet$] We denote by $g^{-1}$ the inverse metric of $g$ naturally defined on the cotangent bundle of $\M$.
\item[$\bullet$] For a vector-valued distribution $u\in\Gamma'_\c(\E)$, we adopt the standard convention that the wavefront set $\operatorname{WF}(u)$ is the union of the wavefront sets of its components in an arbitrary but fixed local frame. Furthermore, given a linear and continuous operator $A\:\Gamma_\c(\E)\to\Gamma(\E)$, the wavefront set $\WF(A)$ is defined as the wavefront set of its Schwartz kernel.
\item[$\bullet$] Given a bidistribution $u\in\Gamma'(\E\boxtimes \E)$,  the primed wavefront of $u$ is defined as
$$\WF'(u)\defeq\{(x,y,\xi,-\eta)\in \T^* (\M\times\M)\setminus\{0\} \,|\, (x,y,\xi,\eta)\in\WF(u)\}\,.$$
\item[$\bullet$] Given a $k$-form $\omega\in\Omega^{k}(\M)$ and a vector field $X\in\Gamma(\T\M)$ we denote by $X\lrcorner\,\omega\in\Omega^{k-1}(\M)$ the
contraction, i.e.~the insertion of $X$ into the first slot of $\omega$.
\item[$\bullet$]For a normed vector space $(\sf{X},\Vert\cdot\Vert_{\sf{X}})$ we will write $\sf{X}\text{-}\lim_{n\to\infty}x_{n}$ to indicate the limit of a converging sequence $(x_{n})_{n\in\mathbb{N}}\in\sf{X}^{\mathbb{N}}$ w.r.t.~the topology induced by $\Vert\cdot\Vert_{\sf{X}}$, in order to avoid confusions in case there are inequivalent norms of $\sf{X}$ in use.
\end{itemize} 
\subsection*{Acknowledgements}
We  are  very grateful to Christian Gérard and Nicola Pinamonti for inspiring discussions about this paper.
We would also like to thank  Tommaso Bruno, Matteo Capoferri, Claudio Dappiaggi, Simone Di Marino, Felix Finster and Stefano Galanda for useful suggestions related to the topic of this paper, as well as the referees for many valuable suggestions. This work was produced within the activities of the INdAM-GNFM and was partly carried out during the program ``Spectral Theory and Mathematical Relativity'' at the Erwin Schrödinger Institute. We are grateful to the ESI for support and the hospitality during our stay.
\subsection*{Fundings}
S.M.~was supported by the GNFM-INdAM Progetto Giovani \textit{Feynman propagator for Dirac fields:~a microlocal analytic approach.} The research of G.S.~is funded by a PhD scholarship of the University of Genoa. The research was supported in part by
the MIUR Excellence Department Project 2023-2027 awarded to the Department of Mathematics of the
University of Genoa.
\section{Maxwell Theory and the Cauchy Radiation Gauge}\label{Sec:Maxwell}
In the following, $(\M,g)$ will denote an {\it globally hyperbolic} ($n+1$)-dimensional Lorentzian manifold of metric signature
$(-, +, \dots, +)$, namely a connected, paracompact and smooth Hausdorff manifold for which there exists a smooth Cauchy temporal function $t\:\M\to\RR$ such that
\begin{align*}
	\M \simeq \RR\times\Sigma
	\qquad
	g=-\beta^2 \d t^2+h_t\,,
\end{align*}
where $\beta\:\RR \times \Sigma \to \RR$ is a smooth and positive function, $h_t$ is a Riemannian metric on each slice
$\Sigma_t\defeq\{t\} \times \Sigma$ varying smoothly with $t$, and each $\Sigma_{t}$ is a smooth spacelike Cauchy hypersurface. We remind the reader that a Cauchy hypersurface is an achronal set intersected exactly once by every inextensible timelike curve. This class of manifolds contains many important examples of spacetimes relevant to general relativity and cosmology. Given a globally hyperbolic manifold, we call {\it Maxwell bundles} the Hermitian bundles 
\begin{align*}
	\V_k \defeq (\M\times\CC)\otimes\bigwedge\nolimits^{k}\T^{\ast}\M,\qquad \langle\cdot,\cdot\rangle_{\V_{k}}\defeq g^{-1}_{(k)}(\cdot,\cdot)\defeq \frac{1}{k!}(g^{-1})^{\otimes k}(\overline{\cdot},\cdot)\, .
\end{align*}
whose sections are smooth (complex) differential $k$-forms $\Gamma(\V_k)=\Omega^k(\M;\CC)$, and we define the  \textit{(complex) Hodge inner products} by
\begin{align}\label{Hodgeprod}
	(\alpha,\beta)_{\V_{k}}\defeq \int_{\M}\,\langle\alpha,\beta\rangle_{\V_{k}}\,\vol 
	\end{align}
for all $\alpha,\beta\in\Gamma(\V_{k})$ with compactly overlapping supports.
As usual, we denote the exterior derivative by $\d\:\Omega^{k-1}(\M)\to\Omega^{k}(\M)$ and define the codifferential operator $\delta\:\Omega^{k}(\M)\to\Omega^{k-1}(\M)$ as the formal adjoint of $\d$ with respect to the Hodge inner product~\eqref{Hodgeprod}. Recall that while $\d$ is independent of the metric, the codifferential $\delta$ does depend on the Lorentzian metric $g$.

With this notation, we shall call \emph{Maxwell operator} the linear and formally self-adjoint differential operator 
\begin{align*}
	\P\:\Gamma(\V_1)\to\Gamma(\V_1)\,,\qquad\P\defeq \delta\d\, .
\end{align*}
The identity $\P\circ\d = 0$ encodes the fact that $\P$ is invariant under linear local gauge transformations $\Gamma(\V_{1})\ni A \mapsto A + \d f$ for $f\in\Gamma(\V_{0})$ and it implies that $\P$ is not hyperbolic. 
A straight-forward computations shows that the operators
\begin{align*}
\D_{0}\defeq \delta\d\:\Gamma(\V_0)\to\Gamma(\V_0) \,, \qquad	\D_{1}\defeq \P+\d\delta\:\Gamma(\V_1)\to\Gamma(\V_1)
\end{align*}
 are normally hyperbolic. Then, solutions of Maxwell's equations $\P A = 0$ can be obtained by
solving $\D_1 A = 0$ with the {\it Lorenz gauge condition} $\delta A = 0$. The Lorenz gauge, however, does not fix completely the gauge degrees of freedom and the remaining gauge freedom is parametrized by the transformations $A\mapsto A+\d f$ with $f\in\ker(\D_{0})$. Our goal is to remove this remaining gauge freedom. 

\begin{definition}\label{def:CauchyRadGa}
Let $(\M,g)$ be a globally hyperbolic manifold and $\Sigma$ a smooth spacelike Cauchy hypersurface. Furthermore, let $t\:\M\to\RR$ be a Cauchy temporal function such that $t^{-1}(0)\cong\Sigma$. A $1$-form $A\in\Omega^{1}(\M;\CC)$ satisfies the \emph{Cauchy radiation gauge} on
 $\Sigma$ if 
$$ \text{(i)} \;\; \delta A=0\qquad \text{(ii)}\;\; A_{n}\vert_{\Sigma}=(\nabla_{n}A)_{n}\vert_{\Sigma}=0 $$ 
where $n\propto\partial_{t}$ denotes the timelike unit normal of $\Sigma$, $A_{n}\defeq n\lrcorner A$ and $(\nabla_{n}A)_{n}\defeq n\lrcorner (\nabla_{n}A)$.
\end{definition}

With the next proposition, we show that our definition is (if evaluated on solutions) equivalent to the well-known radiation gauge on ultrastatic globally hyperbolic manifolds. In particular, if $\Sigma$ is non-compact, it is also equivalent to the Coulomb gauge.
To this end, we fix a globally hyperbolic manifold $\M\cong\RR\times\Sigma$ with $g=-\beta^{2}\d t^{2}+h_{t}$ and we use the following notation for differential operators on the (time-dependent) Riemannian manifold $(\Sigma,h_{t})$:
\begin{itemize}
\item[$\bullet$] $\overrightarrow{\nabla}$ is the Levi-Civita connection of $(\Sigma,h_{t})$;
\item[$\bullet$] $\d_\Sigma$ denotes the exterior derivative on $\Omega^k(\Sigma; \CC)$ and by $\delta_{\Sigma}$ we denote its formal adjoint with respect to the {\it (Riemannian-)Hodge inner product} defined as
	\begin{align}\label{HodgeL2Sigma}(\cdot,\cdot)_{k}\defeq \int_{\Sigma}h^{-1}_{(k)}(\cdot,\cdot)\,\text{vol}_{h}\quad \text{with}\quad h_{(k)}^{-1}(\cdot,\cdot)\defeq \frac{1}{k!}(h^{-1})^{\otimes k}(\overline{\cdot},\cdot)\,; \end{align}
\item[$\bullet$] $\DeltakS\defeq\d_{\Sigma}\delta_{\Sigma}+\delta_{\Sigma}\d_{\Sigma}$ is the \textit{de Rham-Hodge Laplacian} on $\Omega^k(\Sigma;\CC)$;
\item[$\bullet$] $k_{t}$ is the  \textit{second fundamental form} of $\Sigma$ in $\M$, i.e.~the time-dependent tensor field $k_{ij}=-\nabla_{j}n_{i}=-\frac{1}{2\beta}\partial_{t}h_{ij}$ on $\Sigma$, where $n=\beta^{-1}\partial_{t}$ denotes the future-directed timelike unit normal.
\end{itemize}

\begin{proposition}\label{Prop:GaugeCond}
Let $(\M=\RR\times\Sigma,g=-\d t^2+h)$ be a globally hyperbolic ultrastatic manifold and $A\in\ker(\P)$, which we decompose as $A=A_{0}\d t+A_{\Sigma}$ with $A_{0}\defeq \partial_{t}\lrcorner A$. Then, the following statements are equivalent:
\begin{enumerate}
\item[(i)] $A$ satisfies the \emph{Cauchy radiation gauge} on $\Sigma_{0}$;
\item[(ii)]$A$ satisfies both the \emph{temporal gauge} $A_0=0$ and the \emph{Coulomb gauge}  $\delta_{\Sigma}A_{\Sigma}=0$;
\item[(iii)] $A$ satisfies the \emph{radiation gauge}, i.e.~\emph{temporal gauge} $A_0=0$ and the \emph{Lorenz gauge} $\delta A=0$.
\end{enumerate}
Furthermore, if $\Sigma$ is non-compact and $A\in\ker(\P\vert_{\Gamma_{\sc}})$, then (i)-(iii) is equivalent to
\begin{itemize}
\item[(iv)]$A$ satisfies the \emph{Coulomb gauge} $\delta_{\Sigma}A_{\Sigma}=0$.
\end{itemize}
\end{proposition}

\begin{proof}
In the ultrastatic case, one easily derives the equations $\nabla_{0}A_{0}=\partial_{t}A_{0}$ with $\nabla_{0}A_{0}\defeq \partial_{t}\lrcorner\nabla_{\partial_{t}}A$ and $\delta A=\delta_{\Sigma}A_{\Sigma}+\partial_{t}A_{0}$, which shows (ii) $\Leftrightarrow$ (iii). To show (i) $\Leftrightarrow$ (iii) we notice that the following systems are equivalent:
$$ \begin{cases}
\P A =0\\
\delta A=0
\end{cases} \;\; \Longleftrightarrow \quad \begin{cases}
\D_1 A =0\\
\delta A=0
\end{cases} \;\; \Longleftrightarrow \quad \begin{cases}
(\partial_{t}^{2}+\Delta_{0})A_{0} =\square_{0}A_{0}=0\\
(\partial_{t}^{2}+\Delta_{1})A_{\Sigma} =0\\
\delta A=0
\end{cases} $$
We conclude that $A_0=0$ if and only if the initial data for $A_0$ are zero. Last but not least, we show (ii) $\Leftrightarrow$ (iv), which amounts to showing that for non-compact $\Sigma$ and $A\in\ker(\P\vert_{\Gamma_{\sc}})$, the Coulomb gauge implies the temporal gauge. Since $k=0$ and $\Ric(g)_{\mu 0}=0$ in the ultrastatic case (by the Gauss-Codazzi equations), a straight-forward computation in local coordinates shows that the Maxwell operator can be decomposed as $\P A=(\P A)_{0}\d t+(\P A)_{\Sigma}$ with $(\P A)_{0}\defeq \partial_{t}\lrcorner (\P A)$ and
\begin{align*}
       (\P A)_{0}&=\DeltaZS  A_{0}-\partial_{t}\delta_{\Sigma}A_{\Sigma}\\
      (\P A)_{\Sigma}&=\partial_{t}^{2}A_{\Sigma}+\DeltaOS  A_{\Sigma}-\d_{\Sigma}(\delta_{\Sigma}A_{\Sigma}+\partial_{t}A_{0})\, ,
\end{align*}
Therefore, for $A\in\ker(\P\vert_{\Gamma_{\sc}})$ satisfying the Coulomb gauge, we conclude that
$\DeltaZS  A_0=0$. For a fixed time $t$, $A_{0}\vert_{t}$ is compactly supported. But if $\Sigma$ is non-compact, the only harmonic function with compact support is the trivial one and hence $A_{0}$ is necessarily zero.
\end{proof}

\begin{remark}
\begin{itemize}
\item[(i)]Notice that the equivalence between (i)-(iii) and (iv) in Proposition~\ref{Prop:GaugeCond} is true more generally if $\Sigma$ has infinite volume and we choose $A_\Sigma \in C^\infty(\RR, \L_{1}^{2}(\Sigma))$, where $\L_{1}^{2}(\Sigma)$ denotes the $\L^{2}$-space of $1$-forms on $\Sigma$.
\item[(ii)]On a general globally hyperbolic spacetime, $\delta A$, $\delta_{\Sigma}A_{\Sigma}$ and $A_{0}$ are related via 
\begin{align*}
	\delta A=\delta_{\Sigma}A_{\Sigma}+\beta^{-2}\nabla_{0}A_{0}-\beta^{-1}\mathrm{tr}_{h}(k)A_{0}\hspace*{0.5cm}\text{with}\hspace*{0.5cm}\nabla_{0}A_{0}=\partial_{t}A_{0}-\beta^{-1}A_{0}\partial_{t}\beta-\beta h^{-1}(A_{\Sigma},\d_{\Sigma}\beta)\, .
\end{align*}
If $\d_{\Sigma}\beta=0$, it follows that $A_{0}=0$ implies $\delta A=\delta_{\Sigma}A_{\Sigma}$. Therefore, (ii) and (iii) in Proposition \ref{Prop:GaugeCond} are equivalent for every $A\in\Gamma(\V_{1})$ and every globally hyperbolic manifold with a spatially constant lapse function $\beta$. Examples include ultrastatic spacetimes, but also de Sitter and FLRW spacetimes which are relevant for cosmology.
\end{itemize}
\end{remark}

With the next proposition, we shall now see that the achievability of the Cauchy radiation gauge is closely related to the solvability of the Poisson equations on Riemannian manifolds. In particular, in the case of Cauchy-compact globally hyperbolic manifolds, the achievability of the Cauchy radiation gauge is a direct consequence of the Hodge decomposition theorem:

\begin{proposition}\label{Prop:CauchyRadGauge}
Let $(\M=\RR\times\Sigma,g=-\beta^{2}\d t^{2}+h_{t})$ be a globally hyperbolic manifold with compact Cauchy surface $\Sigma$. 
For any $A\in\Gamma(\V_{1})$, there exists a $f\in\Gamma(\V_{0})$, which is unique up to a constant, such that $A^{\prime}\defeq A-\d f$ satisfies the Cauchy radiation gauge on $\Sigma_{0}$. 
\end{proposition}

\begin{proof}
    Let us decompose $A=A_{0}\d t+A_{\Sigma}$ with $A_{0}=\partial_{t}\lrcorner A$. A straight-forward computation in coordinates shows that $\delta A=\delta_{\Sigma}A_{\Sigma}+\beta^{-2}\nabla_{0}A_{0}-\beta^{-1}\mathrm{tr}_{h}(k)A_{0}$. In particular, since $\beta>0$, we see that $A^{\prime}$ being in the Cauchy radiation gauge is equivalent to requiring 
$$ \text{(i)} \;\; \delta A^{\prime}=0\qquad \text{(ii)}\;\; A^{\prime}_{0}\vert_{\Sigma_{0}}=\delta_{\Sigma}A^{\prime}_{\Sigma}\vert_{\Sigma_{0}}=0 $$
i.e.~the condition $\nabla_{0}A^{\prime}_{0}\vert_{\Sigma_{0}}=0$ can be replaced with $\delta_{\Sigma}A^{\prime}_{\Sigma}\vert_{\Sigma_{0}}=0$. As a consequence, being in the Cauchy radiation gauge for $A^{\prime}$ is tantamount to solving the system 
    \begin{equation}\label{ast2}
    \begin{aligned}
        \begin{cases}
            \D_{0}f&=\delta A\\
            \pi &=A_{0}\vert_{\Sigma_{0}}\\
            \DeltaZS  a&=\delta_{\Sigma}A_{\Sigma}\vert_{\Sigma_{0}}
        \end{cases}
    \end{aligned}
    \end{equation}
    where $a\defeq f\vert_{\Sigma_{0}}$ and $\pi\defeq\nabla_{0}f\vert_{\Sigma_{0}}$ denote the initial data of $f$. Decomposing $A_{\Sigma}\vert_{\Sigma_{0}}\in\Omega^{1}(\Sigma)$ via the Hodge decomposition theorem, we find a unique $g\in \Omega^0(\Sigma)$ (up to constant) such that $\delta_\Sigma A_{\Sigma}\vert_{\Sigma_{0}}= \delta_\Sigma \d_ \Sigma g=\Delta_{0}g$. Hence, the Cauchy problem~\eqref{ast2} is equivalent to
     \begin{equation}\label{ast3}
    \begin{aligned}
        \begin{cases}
            \D_{0}f&=\delta A\\
             a&=g\\
            \pi &=A_{0}\vert_{\Sigma_{0}}            
        \end{cases}
    \end{aligned}
    \end{equation}
    where $g$ is specified by $\Delta_{0}g=\delta_{\Sigma}A_{\Sigma}\vert_{\Sigma_{0}}$ up to constant. Since $\D_{0}$ is a normally hyperbolic operator, the solution to the Cauchy problem~\eqref{ast3} exists and it is unique (up to a constant).
\end{proof}

\begin{remark}
Notice that, since $f$ is unique up to a constant, all the gauge degrees of freedom are fixed and we are left with no gauge freedom.
\end{remark}

In order to drop the assumption on the compactness of the Cauchy surface, it seems natural to generalize the Hodge decomposition to non-compact manifolds.  This will be the content of the next section. We also want to point out that the radiation gauge can be achieved on a class of manifolds with suitable isometries \cite{Tolksdorff}.
\section{The Hodge Decomposition for Sobolev Spaces}\label{sec:Hodge}
The aim of this section is to obtain a Hodge decomposition for Sobolev spaces of differential forms. To this end, we shall only consider \textit{complete} Riemannian manifolds in order to establish a suitable notion of Sobolev spaces. This will allow us to solve the Poisson equation with a divergence source and to achieve the Cauchy radiation gauge for Maxwell fields also on globally hyperbolic manifold, which are spatially non-compact.
\subsection{Sobolev Spaces on Complete Riemannian Manifolds}
Let $(\Sigma,h)$ be an oriented, connected and complete $n$-dimensional Riemannian manifold with empty boundary. For the sake of readability, in this section we shall use the following notation:
\begin{itemize}
\item  $\langle\cdot,\cdot\rangle_{\L^2}$ is the  (Riemannian) Hodge-inner product defined by~\eqref{HodgeL2Sigma} and $\L^{2}_{k}(\Sigma)\defeq\overline{\Omega^{k}_{\c}(\Sigma;\CC)}^{\Vert\cdot\Vert_{\L^{2}}}$,
\item $\d\defeq\d_\Sigma\:\Omega^{k}(\Sigma;\CC)\to\Omega^{k+1}(\Sigma;\CC)$ denotes the exterior derivative,
 \item $\delta\defeq\delta_{\Sigma}\:\Omega^{k+1}(\Sigma;\CC)\to\Omega^{k}(\Sigma;\CC)$ is the formal adjoint of $\d_\Sigma$ w.r.t.~the Hodge-inner product,
 \item $\overline{\Delta}\defeq\overline{\Delta_{k}}$ is the minimal self-adjoint extension of $\Delta_{k}=\delta\d +\d \delta\: \Omega^{k}_{\c}(\Sigma;\CC)\to \L^2_k(\Sigma)$, which exists by essential self-adjointness of $\Delta_{k}$ on complete manifolds, see e.g.~\cite{Gaffney1951,Chernoff1973,Strichartz1983}.
 \item Finally, by denoting $\E\defeq \E_{k}\defeq (\mathds{1}+\overline{\Delta}_{k})$, we define for all $s\in\mathbb{R}$ the operators
	\begin{align*}
		\E^{s}\:\mathcal{D}(\E^{s})\to\L^{2}_{k}(\Sigma)\, \,
	\end{align*}
	by means of the spectral theorem, i.e.~
	\begin{align*}
		\mathcal{D}(\E^{s})\defeq\{\omega\in\L^{2}_{k}(\Sigma)\mid (\lambda\mapsto\lambda^{s})\in \L^{2}(\sigma(\E),\d\mu^{\E}_{\omega})\},\quad \E^{s}\defeq \int_{\sigma(\E)}\,\lambda^{s}\,\mathrm{d}\mu^{\E}(\lambda)\, ,
	\end{align*}
	where $\mu^{\E}$ denotes the spectral measure of $\E$ and  $\mu^{\E}_{\omega}(\cdot)\defeq \langle \mu^{\E}(\cdot)\omega,\omega\rangle_{\L^{2}}$. Note that for $s<0$, the operators $\E^{s}$ are bounded.
\end{itemize}

Let us remark that $\sigma(\E)\subset [1,\infty)$, which implies $\sigma(\E^{s})\subset [1,\infty)$ for $s> 0$ by spectral calculus. It follows that $\E^{s}\:\mathcal{D}(\E^{s})\to\L^{2}_{k}(\Sigma)$ is invertible with bounded inverse $\E^{-s}$. This motivates the following definition:

\begin{definition}\label{def:sob space}
	We call the \textit{Sobolev space of degree $s\in [0,\infty)$} by
	\begin{align*}
		\H^{s}_{k}(\Sigma)\defeq \mathcal{D}\Big(\E_{k}^{s/2}\Big), \qquad \langle\cdot,\cdot\rangle_{\H^{s}}\defeq \Big\langle\E_{k}^{s/2}\cdot,\E_{k}^{s/2}\cdot\Big\rangle_{\L^{2}}\, .
	\end{align*}
	In addition, we define the following set equipped with its projective limit topology:
	\begin{align*}
		\H^{\infty}_{k}(\Sigma)\defeq \bigcap_{s\geq 0}\H^{s}_{k}(\Sigma)
	\end{align*}
\end{definition}

Note that $\E^{s}\:\H^{2s}_{k}(\Sigma)\to\L^{2}_{k}(\Sigma)$ is a unitary operator and hence $(\H^{s}_{k}(\Sigma),\langle\cdot,\cdot\rangle_{\H^{s}})$ a Hilbert space. Furthermore, $\H^{s}_{k}(\Sigma)$ is the completion of $\Omega^{k}_{\c}(\Sigma;\CC)$ w.r.t.~$\Vert\cdot\Vert_{\H^{s}}$, which in particular implies that $\Omega^{k}_{\c}(\Sigma;\CC)$ is dense in $\H^{s}_{k}(\Sigma)$, see e.g.~\cite{Dodziuk1981,Strichartz1983}.

\begin{remark}\label{Remark:SobolevEmbedding}
	Let $(\Sigma,h)$ be an $n$-dimensional Riemannian manifold of \textit{bounded geometry} \cite{Gromov1985}, i.e.~its injectivity radius is non-zero and the Riemann curvature tensor and all its covariant derivatives are (uniformly) bounded w.r.t.~the natural norm induced by $h$. Such a manifold is in particular complete and the following facts are well-known:
\begin{itemize}
	\item[(1)]The \textit{Sobolev embedding theorem} holds~\cite{Aubin1976,AubinBook,HebeyBook}, i.e.~$\H^{m}_{k}(\Sigma)\subset C^{l}(\Sigma,\bigwedge^{k}\T^{\ast}\Sigma)$ for $l,m\in\mathbb{N}$ with $m-l>n/2$. Furthermore, the inclusion map is continuous. In particular, note that this implies $\H^{\infty}_{\bullet}(\Sigma)\subset\Omega^{\bullet}(\Sigma;\CC)$. 
	\item[(2)] The definition of $\H^{m}_{k}(\Sigma)$ for $m\in\mathbb{N}$ is equivalent to the one of the standard Sobolev space $W^{m,2}_{k}(\Sigma)$, i.e.~any element in $\H^{m}_{k}(\Sigma)$ is $m$-times weakly differentiable with all weak (covariant) derivatives being in $\L^{2}_{k}(\Sigma)$, see e.g.~\cite{Aubin1976,Dodziuk1981}.
\end{itemize}
\end{remark}

With the next lemma, we shall show that there exists a continuous embedding between Sobolev spaces also on manifolds that are not of bounded geometry.

\begin{lemma} \label{Lemma:WeakCovergence}
	For any  $r \geq s\geq 0$, the Sobolev space $\H^{r}_{\bullet}(\Sigma)$ embeds continuously into $\H^{s}_{\bullet}(\Sigma)$. In particular, (strong) convergence in $\H^{r}_{\bullet}(\Sigma)$ implies (strong) convergence in $\H^{s}_{\bullet}(\Sigma)$ and $\L^{2}_{\bullet}(\Sigma)$.
\end{lemma}

\begin{proof} The spectral theorem implies $\E^{\alpha+\beta}=\overline{\E^{\alpha}\E^{\beta}}$ as well as $\mathcal{D}(\E^{\alpha}\E^{\beta})=\mathcal{D}(\E^{\beta})\cap\mathcal{D}(\E^{\alpha+\beta})$ and hence $\E^{\alpha+\beta}=\E^{\alpha}\E^{\beta}=\E^{\beta}\E^{\alpha}$ on $\mathcal{D}(\E^{\gamma})$ with $\gamma=\max\{\alpha,\beta,\alpha+\beta\}$ for all $\alpha,\beta\in\RR$. Now, let $\omega\in\H^{r}_{\bullet}(\Sigma)\subset\H^{s}_{\bullet}(\Sigma)$. It follows that
	\begin{align*}
		\Vert\omega\Vert_{\H^{s}}=\Vert\E^{\frac{s}{2}}\omega\Vert_{\L^{2}}=\Vert\E^{\frac{s-r}{2}}\E^{\frac{r}{2}}\omega\Vert_{\L^{2}}\leq\Vert\E^{\frac{r}{2}}\omega\Vert_{\L^{2}}=\Vert\omega\Vert_{\H^{r}}\, ,
	\end{align*}
	where we used that $\E^{\frac{s-r}{2}}$ is a bounded operator on $\L^{2}_{\bullet}(\Sigma)$ with operator norm $\Vert\E^{\frac{s-r}{2}}\Vert_{\L^{2}\to\L^{2}}=\mathrm{ess}\,\sup_{\lambda\in\sigma(\E)}\vert\lambda^{\frac{s-r}{2}}\vert\leq 1$. As a special case, we obtain $\Vert\omega\Vert_{\L^{2}}\leq\Vert\omega\Vert_{\H^{s}}$ for any $s\geq 0$.
\end{proof}
\subsection{Hodge Decomposition for Sobolev Spaces}
In order to extend the Hodge decomposition to Sobolev spaces, the first step is to show that the exterior derivative $\d$ and its formal $\L^{2}$-adjoint $\delta$ are well-defined as operators
\begin{align*}
	\d&\:\mathcal{D}(\d)\to\H^{s}_{k+1}(\Sigma)\\
	\delta&\:\mathcal{D}(\delta)\to\H^{s}_{k-1}(\Sigma)
\end{align*}
with dense domains $\mathcal{D}(\d)=\mathcal{D}(\delta)=\Omega^{k}(\Sigma;\CC)\cap\H^{\infty}_{k}(\Sigma)$. As a first step, we show that $\d$ and $\delta$ are formal adjoints of each other w.r.t.~$\langle\cdot,\cdot\rangle_{\H^{s}}$ in the following sense:

\begin{lemma}\label{Lemma:FormalAdjointSobolev}
	Let $s \in [0,\infty]$ and set
	$$		\Omega^{k}_{s}(\Sigma)\defeq \Omega^{k}(\Sigma;\CC)\cap\H^{s}_{k}(\Sigma)\,.$$
Then, for any $\alpha,\beta\in\Omega^{\bullet}_{s}(\Sigma)$ such that $\d\alpha\in\Omega^{\bullet}_{s}(\Sigma)$ and $\delta\beta\in\Omega^{\bullet}_{s}(\Sigma)$ it holds that
 \begin{align*}
		\langle\d\alpha,\beta\rangle_{\H^{s}}=\langle\alpha,\delta\beta\rangle_{\H^{s}}\, .
	\end{align*}
\end{lemma}

\begin{proof}
	First of all, we recall some general facts about $\d$ and $\delta$ as operators in $\L^{2}_{\bullet}(\Sigma)$:
	\begin{itemize}
		\item[(i)]Consider $\d$ and $\delta$ as densely-defined operators $\d\:\Omega^{k}_{\c}(\Sigma;\CC)\to\L^{2}_{k+1}(\Sigma)$ and $\delta\:\Omega^{k+1}_{\c}(\Sigma;\CC)\to\L^{2}_{k}(\Sigma)$. They are clearly formal adjoints w.r.t.~$\langle\cdot,\cdot\rangle_{\L^{2}}$. Furthermore, their $\L^{2}$-closures $\overline{\d}$ and $\overline{\delta}$ are adjoints in $\L^{2}_{\bullet}(\Sigma)$, see e.g.~\cite{Gaffney1951,BrueningLesch1992}. In particular, $\overline{\d}=\overline{\delta}^{\ast}=\delta^{\ast}$ and hence
		\begin{align*}
			\mathcal{D}(\overline{\d})=\{\omega\in\L^{2}_{k}(\Sigma)\mid \d\omega\in\L^{2}_{k+1}(\Sigma)\text{ weakly }\}\, ,
		\end{align*}
		where $\d\omega\in\L^{2}_{k+1}(\Sigma)$ \textit{weakly} means that $\exists\eta\in\L^{2}_{k+1}(\Sigma)$ such that $\forall\varphi\in\Omega^{k+1}_{\c}(\Sigma)$ it holds that $\langle\omega,\delta\varphi\rangle_{\L^{2}}=\langle\eta,\varphi\rangle_{\L^{2}}$. This implies for example that any $\omega\in\Omega^{k}(\Sigma;\CC)\cap\L^{2}_{k}(\Sigma)$ with $\d\omega\in\L^{2}_{k+1}(\Sigma)$ is contained in $\mathcal{D}(\overline{\d})$ and $\overline{\d}\omega=\d\omega$. A similar characterisation holds for $\mathcal{D}(\overline{\delta})$.
		\item[(ii)]As shown in \cite[Lemma 5.9.]{MorettiMurroVolpe2022}, the spectral theorem\footnote{More generally, if $\mathcal{H}$ is a Hilbert space, then for $\B\in\mathcal{B}(\mathcal{H})$ self-adjoint and $\A\:\mathcal{D}(\A)\to\mathcal{H}$ closed, $\B\A\subset\A\B$ implies $f(\B)\A=\A f(\B)$ on $\mathcal{D}(f(\B)\A)\cap\mathcal{D}(f(\B))\subset\mathcal{D}(\A f(\B))$ for every Borel measurable function $f\:\sigma(\B)\to\CC$.} implies for all $s\geq 0$
		\begin{align*}
		\E^{s}\overline{\d}&=\overline{\d}\E^{s},\qquad \text{on }\quad\mathcal{D}(\E^{s})\cap\mathcal{D}(\E^{s}\overline{\d})\subset\mathcal{D}(\overline{\d}\E^{s})\\
		\E^{s}\overline{\delta}&=\overline{\delta}\E^{s},\qquad \text{on }\quad\mathcal{D}(\E^{s})\cap\mathcal{D}(\E^{s}\overline{\delta})\subset\mathcal{D}(\overline{\delta}\E^{s})\,.
	\end{align*}
	\end{itemize}
	Now, let $\alpha,\beta\in\Omega^{\bullet}_{s}(\Sigma)$ with $\d\alpha\in\Omega^{\bullet}_{s}(\Sigma)$ and $\delta\beta\in\Omega^{\bullet}_{s}(\Sigma)$. With this assumption, by (i), $\alpha\in\H^{s}_{\bullet}(\Sigma)\cap\mathcal{D}(\E^{\frac{s}{2}}\overline{\d})$ and $\beta\in\H^{s}_{\bullet}(\Sigma)\cap\mathcal{D}(\E^{\frac{s}{2}}\overline{\delta})$. Using (ii), it follows that 
	\begin{align*}
		\langle \alpha,\delta\beta\rangle_{\H^{s}}&=\langle\E^{\frac{s}{2}}\alpha,\E^{\frac{s}{2}}\delta\beta\rangle_{\L^{2}}=\langle\E^{\frac{s}{2}}\alpha,\overline{\delta}\E^{\frac{s}{2}}\beta\rangle_{\L^{2}}=\\&=\langle\overline{\d}\E^{\frac{s}{2}}\alpha,\E^{\frac{s}{2}}\beta\rangle_{\L^{2}}=\langle\E^{\frac{s}{2}}\d\alpha,\E^{\frac{s}{2}}\beta\rangle_{\L^{2}}=\langle \d\alpha,\beta\rangle_{\H^{s}}
	\end{align*}
	which concludes the proof.
\end{proof}

As a second step, we show that $\d$ and $\delta$ are well defined as operators in $\H^{s}_{\bullet}(\Sigma)$ with domains $\mathcal{D}(\d)=\mathcal{D}(\delta)=\Omega^{k}_{\infty}(\Sigma)$:

\begin{lemma}\label{Def:OmegaS}
	The exterior derivative and codifferential are well-defined as operators 
	\begin{align*}
		\d\:\Omega^{k}_{s}(\Sigma)\to\Omega^{k+1}_{s-1}(\Sigma)\\
		\delta\:\Omega^{k}_{s}(\Sigma)\to\Omega^{k-1}_{s-1}(\Sigma)
	\end{align*}
	for any $s\in [0,\infty]$ with $s\geq 1$. In particular, $\d\:\Omega_{\infty}^{k}(\Sigma)\to\Omega_{\infty}^{k+1}(\Sigma)$ and $\delta\:\Omega_{\infty}^{k}(\Sigma)\to\Omega_{\infty}^{k-1}(\Sigma)$. 
\end{lemma}

\begin{proof}
	We show the claim only for $\d$, since the the proof for $\delta$ can be done analogously. First of all, we claim that the following inequality holds true for all $\varphi\in\Omega^{k}_{\c}(\Sigma;\CC)$:
	\begin{align}\label{eq:Ineq}
		\Vert\d\varphi\Vert_{\H^{s-1}}\leq\Vert\varphi\Vert_{\H^{s}}\, .
	\end{align}
	To see this, we first estimate
	\begin{align*}
		\Vert\varphi\Vert_{\H^{s}}^{2}&=\langle\E^{\frac{s}{2}}\varphi,\E^{\frac{s}{2}}\varphi\rangle_{\L^{2}}=\langle\E^{\frac{s-1}{2}}\varphi,\E^{\frac{s-1}{2}}\E\varphi\rangle_{\L^{2}}=\langle\varphi,\E\varphi\rangle_{\H^{s-1}}=\Vert\varphi\Vert_{\H^{s-1}}^{2}+\langle\varphi,\Delta\varphi\rangle_{\H^{s}}\\&=\Vert\varphi\Vert_{\H^{s-1}}^{2}+\Vert\d\varphi\Vert_{\H^{s-1}}^{2}+\Vert\delta\varphi\Vert_{\H^{s-1}}^{2}\geq\Vert\d\varphi\Vert_{\H^{s-1}}^{2}\,.
	\end{align*}
	where we used Lemma \ref{Lemma:FormalAdjointSobolev}. Now, let $\omega\in\Omega^{k}_{s}(\Sigma)$. Since $\Omega^{k}_{\c}(\Sigma;\CC)$ is dense in $\H^{s}_{k}(\Sigma)$, we can find a sequence $(\omega_{n})_{n}$ in $\Omega^{k}_{\c}(\Sigma;\CC)$ such that $\omega=\H^{s}\text{-}\lim_{n\to\infty}\omega_{n}$. By inequality~\eqref{eq:Ineq}, it follows that $(\d\omega_{n})_{n}$ converges in $\H^{s-1}_{k+1}(\Sigma)$, i.e.~there exists $\eta\in\H^{s-1}_{k+1}(\Sigma)$ such that $\eta=\H^{s-1}\text{-}\lim_{n\to\infty}\d\omega_{n}$. We claim that $\eta=\d\omega$. First of all, by Lemma \ref{Lemma:WeakCovergence}, we know that $\omega=\L^{2}\text{-}\lim_{n\to\infty}\omega_{n}$ and $\eta=\L^{2}\text{-}\lim_{n\to\infty}\d\omega_{n}$. Hence, for all $\varphi\in\Omega^{k+1}_{\c}(\Sigma;\CC)$ it holds that 
\begin{align*}
		&\langle\d\omega_{n},\varphi\rangle_{\L^{2}}\xrightarrow{n\to\infty}\langle\eta,\varphi\rangle_{\L^{2}}\\
		&\hspace*{0.8cm}\verteq\\
		&\langle\omega_{n},\delta\varphi\rangle_{\L^{2}}\xrightarrow{n\to\infty}\langle\omega,\delta\varphi\rangle_{\L^{2}}=\langle\d\omega,\varphi\rangle_{\L^{2}}
	\end{align*}
and hence $\langle\eta-\d\omega,\varphi\rangle_{\L^{2}}=0$. We conclude that $\eta=\d\omega$ by non-degeneracy.
\end{proof}

Next, we shall introduce a suitable class of mollifiers for differential form, similar to those discussed in \cite{BaerLectureNotes} in the setting of compact Riemannian manifolds:

\begin{lemma}\label{Lemma:Mollifiers}
	Consider the self-adjoint bounded operators $\J_{\varepsilon}\defeq e^{-\varepsilon\E}\:\L^{2}_{\bullet}(\Sigma)\to\L^{2}_{\bullet}(\Sigma)$ defined using spectral calculus for $\varepsilon>0$. Then:
	\begin{itemize}
		\item[(i)]$\J_{\varepsilon}$ commutes with $\E^{s}$ on $\H^{2s}_{\bullet}(\Sigma)$ for $s\in [0,\infty)$. In particular, $\J_{\varepsilon}$ is self-adjoint in $\H^{s}_{\bullet}(\Sigma)$. Furthermore, $\J_{\varepsilon}\:\H^{s}_{\bullet}(\Sigma)\to\H^{r}_{\bullet}(\Sigma)$ is well-defined and bounded for all $r,s\in [0,\infty)$.
		\item[(ii)]$\J_{\varepsilon}$ is smoothing, i.e.~$\mathrm{ran}(\J_{\varepsilon})\subset\Omega^{\bullet}(\Sigma;\CC)$.
		\item[(iii)]$\J_{\varepsilon}\:\H^{s}_{\bullet}(\Sigma)\to\H^{s}_{\bullet}(\Sigma)$ converges strongly to $\mathds{1}_{\H^{s}_{\bullet}(\Sigma)}$ as $\varepsilon\to 0^{+}$. 
		\item[(iv)]$\J_{\varepsilon}$ commutes with $\d$ on $\{\omega\in\L^{2}_{\bullet}(\Sigma)\cap\Omega^{k}(\Sigma;\CC)\mid\d\omega\in\L^{2}_{\bullet}(\Sigma)\}$ and similar for $\delta$.
	\end{itemize}
\end{lemma}

\begin{proof}
	For (i), we recall that for two measurable functions $f,g\:\sigma(\E)\to\CC$, it holds that $f(\E)g(\E)=g(\E)f(\E)$ on $\mathcal{D}(f(\E))\cap\mathcal{D}(g(\E))\cap\mathcal{D}((f\cdot g)(\E))$, where $f\cdot g$ denotes the pointwise product. Taking $f_{\varepsilon}(\lambda)\defeq e^{-\varepsilon\lambda}$ and $g(\lambda)\defeq \lambda^{\frac{s}{2}}$, it follows that $\J_{\varepsilon}\E^{\frac{s}{2}}=\E^{\frac{s}{2}}\J_{\varepsilon}$ on $\H^{s}_{\bullet}(\Sigma)$. Secondly, we compute 
	\begin{align}
		\Vert\J_{\varepsilon}\omega\Vert_{\H^{s}}=\Vert \E^{\frac{s}{2}}\J_{\varepsilon}\omega\Vert_{\L^{2}}=\Vert \E^{\frac{s-r}{2}}\J_{\varepsilon}\E^{\frac{r}{2}}\omega\Vert_{\L^{2}}\leq C\cdot\Vert\E^{\frac{r}{2}}\omega\Vert_{\L^{2}}=C\cdot\Vert\omega\Vert_{\H^{r}}
	\end{align}
	with $C>0$ for all $\omega\in\H^{s}_{\bullet}(\Sigma)$, where we used that $\E^{\frac{s-r}{2}}\J_{\varepsilon}$ is a bounded operator on $\L^{2}_{\bullet}(\Sigma)$. For (ii) we observe that $\mathcal{J}_{\varepsilon}\omega$ satisfies the hypoelliptic heat-type equation $(\partial_{\varepsilon}+\E)\mathcal{J}_{\varepsilon}\omega=0$ for any $\omega\in\L^{2}_{\bullet}(\Sigma)$. Statement (iii) is a consequence of the dominant convergence theorem and the fact that $f_{\varepsilon}(\lambda)=e^{-\varepsilon\lambda}$ converges to $\lambda\mapsto 1$ as $\varepsilon\to 0^{+}$. It is left to show (iv). First of all, we observe that $\J_{\varepsilon}$ can equivalently be defined by spectral calculus of the bounded operator $\E^{-1}$ as
	\begin{align*}
		\J_{\varepsilon}=g_{\varepsilon}(\E^{-1}),\hspace*{1cm} g_{\varepsilon}(\lambda)\defeq \begin{cases} e^{-\frac{\varepsilon}{\lambda}}, &\lambda> 0\\
		0, &\text{else}\end{cases}\, .
	\end{align*}
	Note that $g_{\varepsilon}$ is bounded and continuous on $\sigma(\E^{-1})\subset [0,1]$. Now, we recall that $\B\A=\A\B$ on $\mathcal{D}(\A)$ for $\B$ bounded and self-adjoint and $\A$ closed implies $f(\B)\A=\A f(\B)$ on $\mathcal{D}(\A)$ for any bounded measurable function $f$. In our case, $\E^{-1}\overline{\d}=\overline{\d}\E^{-1}$ on $\mathcal{D}(\overline{\d})$, where $\overline{\d}$ denotes the $\L^{2}$-closure of $\d$, as discussed in the proof of Lemma \ref{Lemma:FormalAdjointSobolev}. The claim follows from the fact that $\{\omega\in\L^{2}_{\bullet}(\Sigma)\cap\Omega^{k}(\Sigma;\CC)\mid\d\omega\in\L^{2}_{\bullet}(\Sigma)\}\subset\mathcal{D}(\overline{\d})$ and (ii).
\end{proof}

As a direct consequence of Lemma \ref{Lemma:FormalAdjointSobolev}, we see that $\d$ and $\delta$ are closable in $\H^{s}_{\bullet}(\Sigma)$, since their $\H^{s}$-adjoints are densely defined. We denote their $\H^{s}$-closure in the following by
\begin{align*}
	\overline{\d}&\:\mathcal{D}_{s}(\overline{\d})\to\H^{s}_{k+1}(\Sigma)\\
	\overline{\delta}&\:\mathcal{D}_{s}(\overline{\delta})\to\H^{s}_{k-1}(\Sigma)
\end{align*}
By definition, $\omega\in\H^{s}_{k}(\Sigma)$ is contained in $\mathcal{D}_{s}(\overline{\d})$, if there exists a sequence $(\omega_{n})_{n}$ in $\Omega_{\infty}^{k}(\Sigma)$ with $\omega=\H^{s}\text{-}\lim_{n\to\infty}\omega_{n}$ such that $(\d\omega_{n})_{n}$ is convergent in $\H^{s}_{k+1}(\Sigma)$. We then set $\overline{\d}\omega\defeq \H^{s}\text{-}\lim_{n\to\infty}\d\omega_{n}$.

\begin{lemma}\label{Lemma:Adjoints}
	It holds that $\overline{\delta}^{\ast}=\overline{\d}$, where the adjoint is taken in $\H^{s}_{\bullet}(\Sigma)$.
\end{lemma}

\begin{proof}
	First, note that $\overline{\delta}^{\ast}=\overline{\d}$ is equivalent to $\delta^{\ast}=\overline{\d}$, since $(\overline{\A})^{\ast}=\overline{\A^{\ast}}=\A^{\ast}$ for any closable operator $\A$ on a Hilbert space. To show $\overline{\d}\subset\delta^{\ast}$, let $\omega\in\mathcal{D}_{s}(\overline{\d})$. By assumption, there exists a sequence $\{\omega_{n}\}_{n}\in\Omega^{\bullet}_{\infty}(\Sigma)$ such that 
	\begin{align*}
		\omega=\H^{s}\text{-}\lim_{n\to\infty}\omega_{n}\hspace*{1cm}\text{and}\hspace*{1cm}\overline{\d}\omega=\H^{s}\text{-}\lim_{n\to\infty}\d\omega_{n}\, .
	\end{align*}
	As a consequence, for all $\varphi\in\Omega^{\bullet}_{\infty}(\Sigma)$ it holds that 
	\begin{align*}
		&\langle\d\omega_{n},\varphi\rangle_{\H^{s}}\xrightarrow{n\to\infty}\langle\overline{\d}\omega,\varphi\rangle_{\H^{s}}\\
		&\hspace*{0.8cm}\verteq\\
		&\langle\omega_{n},\delta\varphi\rangle_{\H^{s}}\xrightarrow{n\to\infty}\langle\omega,\delta\varphi\rangle_{\H^{s}}
	\end{align*}
	where we used Lemma \ref{Lemma:FormalAdjointSobolev}. We conclude that $\langle\omega,\delta\varphi\rangle_{\H^{s}}=\langle\overline{\d}\omega,\varphi\rangle_{\H^{s}}$ and hence $\omega\in\mathcal{D}(\delta^{\ast})$ as well as $\delta^{\ast}\omega=\overline{\d}\omega$. It is left to show $\mathcal{D}(\delta^{\ast})\subset\mathcal{D}_{s}(\overline{\d})$. Let $\omega\in\mathcal{D}(\delta^{\ast})$. We have to construct a sequence $\{\omega_{n}\}_{n}\in\Omega^{\bullet}_{\infty}(\Sigma)$ such that 
	\begin{align*}
		\omega=\H^{s}\text{-}\lim_{n\to\infty}\omega_{n}\hspace*{1cm}\text{and}\hspace*{1cm}(\d\omega_{n})_{n}\text{ is convergent in }\H^{s}_{\bullet}(\Sigma)\, ,
	\end{align*}
	because then it follows that $\omega\in\mathcal{D}_{s}(\overline{\d})$ as well as $\H^{s}\text{-}\lim_{n\to\infty}\d\omega_{n}=\overline{\d}\omega=\delta^{\ast}\omega$. We follows an analogues strategy as in \cite{Gaffney1951} adapted for the setting of Sobolev spaces: We consider the class of mollifiers introduces in Lemma \ref{Lemma:Mollifiers}. Now, we argue that for $\omega\in\mathcal{D}(\delta^{\ast})$, it holds that $\d\J_{\varepsilon}\omega=\J_{\varepsilon}\delta^{\ast}\omega$. First note that clearly $\J_{\varepsilon}\omega\in\Omega^{\bullet}_{\infty}(\Sigma)$. Let $\varphi\in\Omega^{\bullet}_{\infty}(\Sigma)$ be arbitrary. Then
	  \begin{align*}
	  	\langle\d\J_{\varepsilon}\omega,\varphi\rangle_{\H^{s}}\stackrel{\ref{Lemma:FormalAdjointSobolev}}{=}\langle\J_{\varepsilon}\omega,\delta\varphi\rangle_{\H^{s}}=\langle\omega,\J_{\varepsilon}\delta\varphi\rangle_{\H^{s}}=\langle\omega,\delta\J_{\varepsilon}\varphi\rangle_{\H^{s}}=\langle\delta^{\ast}\omega,\J_{\varepsilon}\varphi\rangle_{\H^{s}}=\langle\J_{\varepsilon}\delta^{\ast}\omega,\varphi\rangle_{\H^{s}}\, ,
	  \end{align*}
	  By non-degeneracy of $\langle\cdot,\cdot\rangle_{\H^{s}}$ on $\H^{s}_{\bullet}(\Sigma)\times\Omega^{s}_{\infty}(\Sigma)$, it follows that $\d\J_{\varepsilon}\omega=\J_{\varepsilon}\delta^{\ast}\omega$ as claimed and hence $\d\J_{\varepsilon}\omega\to\delta^{\ast}\omega$ as $\varepsilon\to 0^{+}$.
\end{proof}

The last ingredient needed in order to generalize the Hodge decomposition to Sobolev spaces is the following proposition.
\begin{proposition}
	The sequence
	\begin{align*}
		0\xrightarrow{}\mathcal{D}_{s}(\overline{\d_{0}})\xrightarrow{\overline{\d_{0}}}\mathcal{D}_{s}(\overline{\d_{1}})\xrightarrow{\overline{\d_{1}}}\dots\xrightarrow{\overline{\d_{d-1}}}\mathcal{D}_{s}(\overline{\d_{d}})\xrightarrow{\overline{\d_{d}}}0
	\end{align*}
	is a well-defined co-chain complex, i.e.~$\mathrm{ran}(\overline{\d_{k}})\subset\mathcal{D}_{s}(\overline{\d_{k+1}})$ as well as $\overline{\d_{k+1}}\circ\overline{\d_{k}}=0$.
\end{proposition}

\begin{proof}
	Consider $\omega\in\mathcal{D}_{s}(\overline{\d_{k}})$. By definition, this means that there is a sequence $(\omega_{n})_{n}$ in $\Omega^{k}_{\infty}(\Sigma)$ such that $\H^{s}\text{-}\lim_{n\to\infty}\omega_{n}=\omega$ and such that $(\d_{k}\omega_{n})_{n}$ is convergent in $\H^{s}_{k+1}(\Sigma)$. We then set $\overline{\d_{k}}\omega\defeq \H^{s}\text{-}\lim_{n\to\infty}\d_{k}\omega_{n}$. Now, consider the sequence $(\d_{k}\omega_{n})_{n}$. Clearly, also the sequence $(\d_{k+1}\d_{k}\omega_{n}=0)_{n}$ is converging. Hence, $\overline{\d_{k}}\omega\in\mathcal{D}_{s}(\overline{\d_{k+1}})$. Furthermore
	\begin{align*}
		\overline{\d_{k+1}}\overline{\d_{k}}\omega=\H^{s}\text{-}\lim_{n\to\infty}\d_{k+1}\d_{k}\omega_{n}=0\, ,
	\end{align*}
	which concludes the proof.
\end{proof}

In the terminology introduced by Brüning-Lesch \cite{BrueningLesch1992}, the complex $(\mathcal{D}_{s}(\overline{\d_{\bullet}}),\overline{\d_{\bullet}})$ is a \textit{Hilbert complex}, i.e.~a (finite) cochain complex whose cochain maps are closed operators between Hilbert spaces. As shown in \cite[Lemma 2.1.]{BrueningLesch1992}, in every such complex there exists a weak Hodge-type orthogonal decomposition. In the specific case of Sobolev spaces, this yields the first main result of this section, which for $s=0$ reduces to the well-known Hodge-Kodaira decomposition~\cite{Kodaira1949,deRhamBook}:

\begin{theorem}\label{Thm:HodgeDecom}
	The Sobolev space $\H^{s}_{k}(\Sigma)$ admits the following $\H^{s}$-orthogonal decompositions:
	\begin{align*}
		\H^{s}_{k}(\Sigma)\cong \mathrm{Har}^{s}_{k}(\Sigma)\oplus\overline{\d\Omega^{k-1}_{\infty}(\Sigma)}\oplus\overline{\delta\Omega^{k+1}_{\infty}(\Sigma)}\, ,
	\end{align*}
	where $\mathrm{Har}^{s}_{k}(\M)\defeq \mathrm{ker}(\overline{\d_{k}})\cap\mathrm{ker}(\overline{\delta_{k}})$ and the closures are taken w.r.t.~$\Vert\cdot\Vert_{\H^{s}}$. 
\end{theorem}

\begin{proof}
	First, we show that there is the $\H^{s}$-orthogonal decomposition
	\begin{align}\label{eq:Decomposition}
		\H^{s}_{k}(\Sigma)\cong \mathrm{Har}^{s}_{k}(\Sigma)\oplus\overline{\mathrm{ran}(\overline{\d})}\oplus\overline{\mathrm{ran}(\overline{\delta})}\,.
	\end{align}
	Since $\overline{\d}$ is a closed operator, its kernel is closed in $\H^{s}_{k}(\Sigma)$ and hence
	\begin{align*}
		\H^{s}_{k}(\Sigma)\cong\mathrm{ker}(\overline{\d})^{\perp}\oplus \mathrm{ker}(\overline{\d})=\mathrm{ker}(\overline{\d})^{\perp}\oplus \overline{\mathrm{ran}(\overline{\d})}\oplus (\mathrm{ker}(\overline{\d})\cap\mathrm{ran}(\overline{\d})^{\perp})\, .
	\end{align*}
	Decomposition~\eqref{eq:Decomposition} follows now from Lemma \ref{Lemma:Adjoints} and the general fact that for any densely defined closed operator $\A$ on a Hilbert space it holds that $\mathrm{ker}(\A^{\ast})^\perp=\overline{\mathrm{ran}(\A)}$. To prove the proposition, we use the simple observation that for any closable operator $\A$  it holds that $\overline{\mathrm{ran}(\overline{\A})}=\overline{\mathrm{ran}(\A)}$, which shows that $\overline{\mathrm{ran}(\overline{\d})}=\overline{\mathrm{ran}(\d)}=\overline{\d\Omega^{k-1}_{\infty}(\Sigma)}$ and similar for $\overline{\mathrm{ran}(\overline{\delta})}$.
\end{proof}

The following proposition gives an equivalent characterization of the space $\mathrm{Har}_{k}^{s}(\Sigma)$ in terms of harmonic forms:

\begin{proposition}\label{Prop:Harmonic}
	It holds that  $\mathrm{Har}_{k}^{s}(\Sigma)=\mathrm{ker}(\Delta_{k}\vert_{\Omega^{k}_{s}})$ for all $k\in\mathbb{N}$ and $s\in [0,\infty]$.
\end{proposition}

\begin{proof}
	For ``$\subset$'', let $\omega\in\mathrm{Har}_{k}^{s}(\Sigma)$. In particular, this means that $\omega\in\mathcal{D}_{s}(\overline{\d})\cap\mathcal{D}_{s}(\overline{\delta})$ and $\overline{\d}\omega=\overline{\delta}\omega=0$. We first show that $\omega$ is smooth: Since $\omega\in\mathcal{D}_{s}(\overline{\d})$, there exists a sequence $(\omega_{n})_{n}$ in $\Omega^{k}_{\infty}(\Sigma)$ such that $\omega=\H^{s}\text{-}\lim_{n\to\infty}\omega_{n}$ and $\overline{\d}\omega=\H^{s}\text{-}\lim_{n\to\infty}\d\omega_{n}=0$. By Lemma \ref{Lemma:WeakCovergence} the same convergence holds in the $\L^{2}$-topology. Therefore 
\begin{align*}
		&\langle\omega_{n},\delta\varphi\rangle_{\L^{2}}\xrightarrow{n\to\infty}\langle\omega,\delta\varphi\rangle_{\L^{2}}\\
		&\hspace*{0.8cm}\verteq\\
		&\langle\d\omega_{n},\varphi\rangle_{\L^{2}}\xrightarrow{n\to\infty}\langle\overline{\d}\omega,\varphi\rangle_{\L^{2}}=0
	\end{align*}
and hence $\langle\omega,\delta\varphi\rangle_{\L^{2}}=0$ for all $\varphi\in\Omega^{k+1}_{\c}(\Sigma;\CC)$. Similarly, we show that $\langle\omega,\d\varphi\rangle_{\L^{2}}=0$ for all $\varphi\in\Omega^{k-1}_{\c}(\Sigma;\CC)$. We conclude that $(\d+\delta)\omega=0$ in the distributional sense and hence $\omega\in\Omega^{k}(\Sigma;\CC)$ be elliptic regularity of the operator $\d+\delta$. It follows that $\omega\in\mathrm{ker}(\Delta_{k})\cap\Omega^{k}_{s}(\Sigma)$. For ``$\supset$'', let $\omega\in\Omega^{k}_{s}(\Sigma)$ be such that $\Delta\omega=0$. We recall that $\Delta\eta=0$ for $\eta\in\L^{2}_{k}(\Sigma)$ if and only if $\d\eta=\delta\eta=0$ on complete manifolds (see e.g.~\cite{Gaffney1954,Yau1976}), which implies $\d\omega=0$ and $\delta\omega=0$. We have shown that $\ker(\Delta\vert_{\Omega^{k}_{s}})\subset\ker(\d\vert_{\Omega^{k}_{s}})\cap\ker(\delta\vert_{\Omega^{k}_{s}})$. It is left to show that $\ker(\d\vert_{\Omega^{k}_{s}})\cap\ker(\delta\vert_{\Omega^{k}_{s}})\subset\ker(\overline{\d})\cap\ker(\overline{\delta})$. To show $\mathrm{ker}(\d)\cap\Omega^{k}_{s}(\Sigma)\subset\ker(\overline{\d})$, we take $\omega\in\Omega^{k}_{s}(\Sigma)$ such that $\d\omega=0$ and we have to construct a sequence $(\omega_{n})_{n}$ in $\Omega_{\infty}^{k}(\Sigma)$ such that $\omega=\H^{s}\text{-}\lim_{n\to\infty}\omega_{n}$ and $\H^{s}\text{-}\lim_{n\to\infty}\d\omega_{n}=0$. We consider the mollifiers $\J_{\varepsilon}$ introduced in Lemma \ref{Lemma:Mollifiers} and notice that $\d\J_{\varepsilon}\omega=\J_{\varepsilon}\d\omega=0$ and $\J_{\varepsilon}\omega\in\Omega^{k}_{\infty}(\Sigma)$. Hence $\J_{\varepsilon}\omega$ gives the  required sequence. Similarly, $\mathrm{ker}(\delta)\cap\Omega^{k}_{s}(\Sigma)\subset\ker(\overline{\delta})$.
\end{proof}

\begin{remark}
Let us give an overview of Hodge type decomposition theorems obtained for non-compact manifolds: A classical result due to Kodaira \cite{Kodaira1949} (see also \cite{BrueningLesch1992,deRhamBook}) gives a decomposition of $\L^{2}_{k}(\Sigma)$. Gromov proved in~\cite{Gromov1991} a strong Hodge decomposition for $\L^{2}_{k}(\Sigma)$ under the assumption of a mass gap of the Laplacian. Some results for (weighted) Sobolev spaces on asymptomatically Euclidean manifold can be found in \cite{Cantor} and for manifolds with compact boundary in \cite{Schwarz}. A decomposition of $\H^{1}$ for hyperbolic manifolds was obtained in \cite{ChanCzubakSuarez2018}.
\end{remark}

Next, let us discuss the case of the space $\H^{\infty}(\Sigma)$. This space is in general not a Hilbert space, but nevertheless one can decompose it in a similar fashion:

\begin{corollary}\label{InfintieHodge}
	The vector space $\H^{\infty}_{\bullet}(\Sigma)$ admits the following direct sum decomposition:
	\begin{align*}
		\H^{\infty}_{\bullet}(\Sigma)\cong \overline{\d\Omega^{k-1}_{\infty}(\Sigma)}^{\infty}\oplus\overline{\delta\Omega^{k+1}_{\infty}(\Sigma)}^{\infty}\oplus\mathrm{Har}^{\infty}_{\bullet}(\Sigma)\,,
	\end{align*}
	where we used the notation
	\begin{align*}
		\overline{\d\Omega^{k}_{\infty}(\Sigma)}^{\infty}\defeq \bigcap_{s\geq 0}\overline{\d\Omega^{k}_{\infty}(\Sigma)}^{\H^{s}},\quad \overline{\delta\Omega^{k}_{\infty}(\Sigma)}^{\infty}\defeq \bigcap_{s\geq 0}\overline{\delta\Omega^{k}_{\infty}(\Sigma)}^{\H^{s}},\quad \mathrm{Har}^{\infty}_{\bullet}(\Sigma)\defeq \bigcap_{s\geq 0}\mathrm{Har}^{s}_{\bullet}(\Sigma)\, .
	\end{align*}
	This decomposition is $\H^{s}$-orthogonal $\forall s\geq 0$ and hence in particular $\L^{2}$-orthogonal.
\end{corollary}

\begin{proof}
	Let $\omega\in\H^{\infty}_{\bullet}(\Sigma)$. By assumption, it belongs to $\H^{s}_{\bullet}(\Sigma)$ for any $s\geq 0$ and hence, using the Hodge decomposition (Theorem \ref{Thm:HodgeDecom}), there exists for any $s\geq 0$ a decomposition of the form
\begin{align*}
		\omega=\alpha_{s}+\beta_{s}+\gamma_{s},\quad\alpha_{s}\in\overline{\d\Omega^{k-1}_{\infty}(\Sigma)}^{\H^{s}},\,\beta_{s}\in\overline{\delta\Omega^{k+1}_{\infty}(\Sigma)}^{\H^{s}},\,\gamma\in\mathrm{Har}^{s}_{\bullet}(\Sigma)\, .
\end{align*}
	Now, by Lemma \ref{Lemma:WeakCovergence}, it clearly holds that 
\begin{align*}
	\overline{\d\Omega^{k-1}_{\infty}(\Sigma)}^{\H^{r}}\subset\overline{\d\Omega^{k-1}_{\infty}(\Sigma)}^{\H^{s}},\quad \overline{\delta\Omega^{k+1}_{\infty}(\Sigma)}^{\H^{r}}\subset\overline{\delta\Omega^{k+1}_{\infty}(\Sigma)}^{\H^{s}},\quad \mathrm{Har}^{r}_{\bullet}(\Sigma)\subset\mathrm{Har}^{s}_{\bullet}(\Sigma)
\end{align*}	
	 for all $r,s\geq 0$ with $r\geq s$. By uniqueness of the Hodge decomposition, we hence must have $\alpha_{s}=\alpha_{r}$, $\beta_{s}=\beta_{r}$ as well as $\gamma_{s}=\gamma_{r}$ for any pair $r,s\geq 0$, which proves the claimed decomposition.
\end{proof}

If $(\Sigma,h)$ is a Riemannian manifold of bounded geometry, then $\H^{\infty}_{k}(\Sigma)\subset\Omega^{k}(\Sigma;\CC)$ by the Sobolev embedding theorem, cf.~Remark \ref{Remark:SobolevEmbedding}, and hence in particular $\H^{\infty}_{k}(\Sigma)=\Omega^{k}_{\infty}(\Sigma)$. In the next section, we discuss a more general Hodge decomposition for smooth differential forms on non-compact manifolds.
\subsection{Smooth Hodge Decomposition on Non-Compact Manifolds}
We are finally in the position to generalize the (smooth) Hodge decomposition on non-compact manifolds. To this end, let us introduce the following notation.

\begin{definition}\label{Def:SubspacesHodge}
	Let $(\Sigma,h)$ be a complete Riemannian manifold. We define the following subspaces of $\Omega^{k}_{s}(\Sigma)$ for $s\in\RR$ and $\Omega^{k}_{\infty}(\Sigma)=\bigcap_{s\geq 0}\Omega^{k}_{s}(\Sigma)$: 
	\begin{align*}
		\Omega^{k}_{s,\d}(\Sigma) &\defeq \Omega^{k}(\Sigma;\CC)\cap\overline{\d_{\Sigma}\Omega^{k-1}_{\infty}(\Sigma)}^{\Vert\cdot\Vert_{\H^{s}}},\qquad\Omega_{\infty,\d}^{k}(\Sigma)\defeq\bigcap_{s\geq 0}\Omega_{s,\d}^{k}(\Sigma)\, ,\\
		\Omega^{k}_{s,\delta}(\Sigma) &\defeq \Omega^{k}(\Sigma;\CC)\cap\overline{\delta_{\Sigma}\Omega^{k+1}_{\infty}(\Sigma)}^{\Vert\cdot\Vert_{\H^{s}}},\qquad\Omega_{\infty,\delta}^{k}(\Sigma)\defeq\bigcap_{s\geq 0}\Omega_{s,\delta}^{k}(\Sigma)\, .
	\end{align*}
\end{definition}

We can finally state and prove the first main result of this paper, Theorem~\ref{Thm:HodgeDecomSmooth}.

\begin{proof}[Proof of Theorem~\ref{Thm:HodgeDecomSmooth}]
	Let $s\geq 0$ be fixed. First of all, we show that any form $\alpha\in\overline{\d\Omega^{k-1}_{\infty}(\Sigma)}$ satisfies $\langle\alpha,\delta\varphi\rangle_{\L^{2}}=0$ for all $\varphi\in\Omega^{k+1}_{\c}(\Sigma;\CC)$. Let us write $\alpha=\H^{s}\text{-}\lim_{n\to\infty}\d\alpha_{n}$ for a sequence $(\alpha_{n})_{n}$ in $\Omega^{k-1}_{\infty}(\Sigma)$. By Lemma \ref{Lemma:WeakCovergence}, it holds that $\alpha=\L^{2}\text{-}\lim_{n\to\infty}\d\alpha_{n}$. In particular, we have that
	\begin{align*}
		0=\langle\d\alpha_{n},\delta\varphi\rangle_{\L^{2}}\xrightarrow{n\to\infty}\langle\alpha,\delta\varphi\rangle_{\L^{2}}
	\end{align*}
	for any $\varphi\in\Omega^{k+1}_{\c}(\Sigma;\CC)$ and hence $\langle\alpha,\delta\varphi\rangle_{\L^{2}}=0$. Similarly, a form $\beta\in\overline{\delta\Omega^{k+1}_{\infty}(\Sigma)}$ satisfies $\langle\alpha,\d\varphi\rangle_{\L^{2}}=0$ for all $\varphi\in\Omega^{k-1}_{\c}(\Sigma;\CC)$. Now, let $\omega\in\Omega^{k}_{s}(\Sigma)$ and let us uniquely decompose it in accordance with Theorem \ref{Thm:HodgeDecom} as
	\begin{align*}
		\omega=\alpha+\beta+\gamma,\quad\alpha\in\overline{\d\Omega^{k-1}_{\infty}(\Sigma)},\,\beta\in\overline{\delta\Omega^{k+1}_{\infty}(\Sigma)},\,\gamma\in\mathrm{Har}^{s}_{k}(\Sigma)=\ker(\Delta\vert_{\Omega^{k}_{s}})\, .
	\end{align*}
	Then, for every $\varphi\in\Omega^{k+1}_{\c}(\Sigma;\CC)$, it holds that
	\begin{align*}
		\langle\delta\omega,\varphi\rangle_{\L^{2}}=\langle\omega,\d\varphi\rangle_{\L^{2}}=\langle\alpha,\d\varphi\rangle_{\L^{2}}=\langle\alpha,(\d+\delta)\varphi\rangle_{\L^{2}}\, .
	\end{align*}
	In other words, $\delta\omega=(\d+\delta)\alpha$, where $(\d+\delta)\alpha$ has to be understood in the distributional sense. By elliptic regularity of the operator $\d+\delta$ and smoothness of $\delta\omega$ we conclude that $\alpha\in\Omega^{k}(\Sigma;\CC)$. Similarly we argue for $\beta$. Smoothness of $\gamma$ has been shown in Proposition \ref{Prop:Harmonic}. For the exactness claims (i) and (ii), we use Poincaré duality: Let $\alpha\in\Omega^{k}_{s,\d}(\Sigma)$. Above we showed that $\langle\alpha,\delta\varphi\rangle_{\L^{2}}=0$ for all $\varphi\in\Omega^{k+1}_{\c}(\Sigma;\CC)$. Since $\alpha$ now is in addition assumed to be smooth, we conclude that $\d\alpha=0$ by non-degeneracy. Now, by Poincaré duality, a closed form $\alpha$ is exact if and only if $\langle\alpha,\psi\rangle_{\L^{2}}=0$ for all $\psi\in\Omega^{\bullet}_{\c}(\Sigma;\CC)\cap\ker(\delta)$. By assumption, there is a sequence $(\alpha_{n})_{n}$ in $\Omega^{k-1}_{\infty}(\Sigma)$ such that $\alpha=\H^{s}\text{-}\lim_{n\to\infty}\d\alpha_{n}$ and hence $\alpha=\L^{2}\text{-}\lim_{n\to\infty}\d\alpha_{n}$ by Lemma \ref{Lemma:WeakCovergence}. It follows that
	\begin{align*}
		0=\langle\alpha_{n},\delta\psi\rangle_{\L^{2}}=\langle\d\alpha_{n},\psi\rangle_{\L^{2}}\xrightarrow{n\to\infty}\langle\alpha,\psi\rangle_{\L^{2}}
	\end{align*}
	and hence $\langle\alpha,\psi\rangle_{\L^{2}}=0$. We conclude that $\alpha$ is exact. The claim for $\beta\in\Omega^{k}_{s,\delta}(\Sigma)$ follows by duality. For (iii), lets assume that $\d\omega=0$. It follows that $\d\beta=0$. Now, since also $\delta\beta=0$, we conclude that $\beta\in\Omega^{k}_{s,\delta}(\Sigma)\cap\mathrm{Har}_{k}^{s}(\Sigma)$ and hence $\beta=0$, since the Hodge decomposition is a direct sum decomposition. Similarly we show (iv). The decomposition of $\Omega_{\infty}^{k}(\Sigma)$ and the analogues properties (i)-(iv) follow immediately using similar arguments as in the proof of Corollary \ref{InfintieHodge}.
\end{proof}
\subsection{Poisson Equation on Non-Compact Manifolds}
As an immediate consequence of Theorem~\ref{Thm:HodgeDecomSmooth}, we get existence and uniqueness of solutions of the Poisson equation, which is closely related to the achievability of the Cauchy radiation gauge.

\begin{proposition}\label{prop:Poisson}
	Let $s\in [0,\infty]$ and $\omega\in\Omega_{s}^{1}(\Sigma)$. Then the Poisson equation
	\begin{align*}
		\Delta_{0}f=\delta\omega\, ,
	\end{align*}
	has a unique solution (up to constant) on the space $\{f\in C^{\infty}(\Sigma;\CC)\mid \d f\in\Omega_{s,\d}^{1}(\Sigma)\}$.
\end{proposition}

\begin{proof}
	By Theorem~\ref{Thm:HodgeDecomSmooth}, $\omega$ can be written as $\omega=\alpha+\beta$, where $\alpha\in\Omega_{s,\d}^{1}(\Sigma)$ and $\beta\in\ker(\delta\vert_{\Omega^{1}_{s}})$. Furthermore, $\alpha=\d f$ for some $f\in C^{\infty}(\Sigma;\CC)$. It follows that $\delta\omega=\delta\alpha=\Delta f$. For uniqueness, let $f\in C^{\infty}(\Sigma;\CC)$ be such that $\d f\in\Omega^{1}_{s,\d}(\Sigma)$ and $\Delta_{0}f=0$. Then, the $1$-form $\eta\defeq \d f$ satisfies $\d \eta=\delta\eta=0$ and hence in particular $\eta\in\ker(\Delta\vert_{\Omega^{1}_{s}})$. Since also $\eta\in\Omega^{1}_{s,\d}(\Sigma)$ by assumption, we conclude that $\eta=0$, by the fact that the decomposition in Theorem \ref{Thm:HodgeDecomSmooth} is a direct sum decomposition. It follows that $f=\mathrm{const}$ as claimed.
\end{proof}
\section{On the Cauchy Problem with Smooth Sobolev Data}\label{Sec:CauchyProblem}
In order to discuss the phase space of Maxwell theory in the Cauchy radiation gauge, as a preliminary result, we shall show that the Cauchy problem for the wave operator $\D_k$ is continuous in the Sobolev topology.  To this end, we assume the following setup.
\begin{setup}\label{setup}
$(\M=\RR\times\Sigma,g=-\beta^{2}\d t^{2}+h_{t})$ is a globally hyperbolic manifold such that
\begin{itemize}
\item  $(\Sigma,h_t)$ are complete Cauchy hypersurfaces of bounded geometry;
\item the lapse function $\beta$, $\beta^{-1}$ and their time derivatives are bounded in the sense that for any $m\in\RR$ and any $p\in\NN$ there exists a positive constant $C_{mp}\in \RR$ such that for any $u\in\H^s(\Sigma_t)$ it holds
$$\Vert \partial_t^p \beta^m u\Vert_{\H^{s}}\leq C_{mp} \Vert u \Vert_{\H^{s}}\,;$$
\item the second fundamental form $k$ is bounded in a similar sense as above.
\end{itemize}  
\end{setup}
Examples of globally hyperbolic manifolds satisfying Setup~\ref{setup} are  ultrastatic manifolds, de Sitter space and FLRW spacetimes.
\medskip

As in the previous section, we introduce a family of Sobolev spaces $\H^{s}_{\bullet}(\Sigma_{t})$ (see Definition~\ref{def:sob space}) and as before we denote its intersection with the smooth forms by
	$$	\Omega^{k}_{s}(\Sigma_{t})=\Omega^{k}(\Sigma;\CC)\cap\H^{s}_{k}(\Sigma_{t})\,.$$

Recall that on a globally hyperbolic manifold, any $k$-form $\omega\in\Omega^{k}(\M;\CC)$ can be decomposed as $\omega=\d t\wedge\omega_{0}+\omega_{1}$ with $\omega_{0}\defeq \partial_{t}\lrcorner\,\omega$, which yields a decomposition  $$\Gamma(\V_k)=\Omega^{k}(\M;\CC) \simeq C^\infty\big(\RR,\Omega^{k-1}(\Sigma;\CC)\big)\oplus C^\infty\big(\RR,\Omega^{k}(\Sigma;\CC)\big) \, .$$
We further introduce the subspace $$\Gamma_s(\V_k)\defeq C^\infty\big(\RR,\Omega_s^{k-1}(\Sigma_{\bullet})\big)\oplus C^\infty\big(\RR,\Omega_s^{k}(\Sigma_{\bullet})\big) $$ and set $\Gamma_{\tc,s}(\V_{k})\defeq \Gamma_{\tc}(\V_{k})\cap\Gamma_{s}(\V_{k})$ for the subspace of timelike compactly-supported fields.

\medskip

Let us now consider the normally-hyperbolic d'Alembertian $\D_{k}\defeq \delta\d+\d\delta\:\Gamma(\V_{k})\to\Gamma(\V_{k})$ and denote by $\rho_{k,t}\:\Gamma(\V_{k})\to\Gamma(\V_{\rho_{k,t}})$ the initial data map defined by
\begin{align*}
	\rho_{k,t}(\omega)\defeq (\omega\vert_{\Sigma_{t}},\i^{-1}\partial_{t}\omega\vert_{\Sigma_{t}})\, .
\end{align*}
where $\V_{\rho_{k,t}}\defeq \V_{k}\vert_{\Sigma_{t}}\oplus\V_{k}\vert_{\Sigma_{t}}\,$ and $\Gamma(\V_{k}\vert_{\Sigma_{t}})\cong\Omega^{k-1}(\Sigma;\CC)\oplus\Omega^{k}(\Sigma;\CC)$.
In the case $t=0$, we will simply write $\rho_{k}\defeq \rho_{k,0}$ and $\V_{\rho_{k}}\defeq \V_{\rho_{k,0}}$.
The space of smooth Sobolev initial data will be denoted by
\begin{align*}
	\Gamma_{s}(\V_{\rho_{k,t}})\defeq \mathcal{H}^{s}_{k}(\Sigma_{t})\oplus\mathcal{H}^{s-1}_{k}(\Sigma_{t})\qquad \text{with}\quad \mathcal{H}^{s}_{k}(\Sigma_{t})\defeq \Omega^{k-1}_{s}(\Sigma_{t})\oplus\Omega_{s}^{k}(\Sigma_{t}) \,,
\end{align*}
where the space $\mathcal{H}_{k}^{s}(\Sigma_{t})$ is naturally topologized by the norms
	\begin{align*}
		\Vert (\alpha_{0},\alpha_{1})\Vert_{\mathcal{H}^{s}}\defeq (\Vert\alpha_{0}\Vert_{\H^{s}}^{2}+\Vert\alpha_{1}\Vert_{\H^{s}}^{2})^{\frac{1}{2}}\, .
	\end{align*}
for all $(\alpha_{0},\alpha_{1})\in\mathcal{H}_{k}^{s}(\Sigma_{t})$. With the notation introduced above we can finally formulate the main result of this section.

\begin{theorem}\label{thm:sob Cauchy}
Let $(\M,g)$ be as in the Setup~\ref{setup}.  
Then the Cauchy problem for  $\D_k$  is well-posed, i.e.~for any $f\in\Gamma_{s-2}(\V_k)$ and $(h_1,h_2)\in\Gamma_s(\V_{\rho_{k,t_{0}}})$ with $t_{0}\in\RR$ there exists a unique solution $\omega\in\Gamma_{s}(\V_k)$ to the initial value problem
	\begin{equation*} 
	\begin{cases}{}
	{\D_k} \omega=f   \\
	\omega|_{\Sigma_{t_0}} = h_1   \\
	(\i^{-1}\partial_t \omega)|_{\Sigma_{t_0}}=h_2 \\
	\end{cases} 
	\end{equation*}
  which depends continuously on the data $(f,h_1,h_2)$ with respect to the natural Fréchet topologies.
\end{theorem}

\begin{remark}
	The Cauchy data map $\rho_{k}$ restricts to a well-defined map $\rho_{k}\:\Gamma_{s}(\V_{k})\to\Gamma_{s}(\V_{\rho_{k}})$. Theorem~\ref{thm:sob Cauchy} implies that its restriction to $\ker(\D_{k}\vert_{\Gamma_{s}})$ is bijective with inverse given by the \textit{Cauchy evolution operator} $\U_{k}\:\Gamma_{s}(\V_{\rho_{k}})\to\ker(\D_{k}\vert_{\Gamma_{s}})$. 
\end{remark}

Since compactly supported Cauchy data are dense\footnote{As mentioned earlier, using standard arguments it can be shown that $\Omega^{k}_\c(\Sigma,\CC)$ is dense in $\H_{k}^s(\Sigma)$ for any $s\geq 0$.} in $\Gamma_{s-2}(\V_k)\oplus \Gamma_s(\V_{\rho_k})$ , and the solution to Cauchy problem is spacelike compact, to prove our claim, it remains to show that given a sequence of compactly supported Cauchy data converging to smooth element in $\Gamma_{{s-2}}(\V_k)\oplus \Gamma_s(\V_{\rho_k})$, the solution will be an element in $\Gamma_{s}(\V_k)$ . To achieve our goal we need to introduce suitable energy estimates. We begin with a preliminary result.

\begin{lemma}\label{Lemma:L2Bound}
	Let $\omega\in C^{\infty}(\RR,\L^{2}_{k}(\Sigma_{\bullet}))$ such that $\partial_{t}\omega\in C^{\infty}(\RR,\L^{2}_{k}(\Sigma_{\bullet}))$. Then 
	\begin{align*}
		\frac{\d}{\d t}\Vert\omega_{t}\Vert_{\L^{2}}^{2}\leq 2\mathrm{Re}\{\langle\partial_{t}\omega_{t},\omega_{t}\rangle_{\L^{2}}\}+c(t)\cdot\Vert\omega_{t}\Vert_{\L^{2}}^{2}
	\end{align*}
	for some positive constant $c$ that is locally uniform in time.
\end{lemma}

\begin{proof}
	First, we note that there is a unitary operator $\mathsf{U}_{t}\:\L^{2}_{k}(\Sigma_{t})\to\L^{2}_{k}(\Sigma_{0})$ defined by $\mathsf{U}_{t}=\rho_{t}\cdot u_{t}$, where $u_{t}\:\Omega^{k}(\Sigma_{t};\CC)\to\Omega^{k}(\Sigma_{0};\CC)$ is constructed using the flow of $\partial_{t}$ and $\rho\in C^{\infty}(\M,(0,\infty))$ is the map defined by $\mathrm{vol}_{h_{t}}=\rho_{t}^{2}\mathrm{vol}_{h_{0}}$. By the dominate convergence, it follows that 
	\begin{align*}
		\frac{\d}{\d t}\Vert\omega_{t}\Vert_{\L^{2}}^{2}=&\frac{\d}{\d t}\Vert\mathsf{U}_{t}\omega_{t}\Vert_{\L^{2}}^{2}=\frac{\d}{\d t}\int_{\Sigma}\,(h_{0})^{-1}_{(k)}(u_{t}\omega_{t},u_{t}\omega_{t})\,\rho_{t}^{2}\mathrm{vol}_{h_{0}}\leq \\\leq& 2\mathrm{Re}\bigg\{\int_{\Sigma}(h_{0})^{-1}_{(k)}(\partial_{t}u_{t}\omega_{t},u_{t}\omega_{t})\,\rho^{2}_{t}\mathrm{vol}_{h_{0}}\bigg\}+c(t)\cdot\Vert\omega_{t}\Vert_{\L^{2}}^{2}=\\=&2\mathrm{Re}\{\langle\partial_{t}\omega_{t},\omega_{t}\rangle_{\L^{2}}\}+c(t)\cdot\Vert\omega_{t}\Vert_{\L^{2}}^{2}
	\end{align*}
	where we used the fact that $u_{t}$ commutes with $\partial_{t}$ up to bounded terms depending on the second fundamental form $k$ and lapse function $\beta$. Since the term involving $\partial_{t}\rho_{t}^{2}$, which yields a logarithmic change of the volume $\frac{\partial_{t}\rho_{t}^{2}}{\rho_{t}^{2}}\mathrm{vol}_{h_{t}}$, can be bounded by a time-dependent constant $c(t)$, we can conclude our proof.
\end{proof}	

\begin{proposition}[Energy Estimates]\label{Prop:EnergyEstimate}
	Assume the Setup~\ref{setup}. Then, for any  $\omega\in\Gamma_{s}(\V_{k})$ such that $\partial_t\omega\in \Gamma_{s}(\V_{k})$ and $\square_{k}\omega\in\Gamma_{s-2}(\V_{k})$ it holds
	\begin{align*}
		\mathcal{E}_{s}(\omega,t_{1})\leq \mathcal{E}_{s}(\omega,t_{0})\cdot e^{C (t_{1}-t_{0})}+\int_{t_{0}}^{t_{1}}e^{C (t_{1}-\tau)}\Vert\square_{k}\omega\vert_{\Sigma_{\tau}}\Vert_{\mathcal{H}^{s-2}}^{2}\,\d\tau
	\end{align*}
	for some constant $C>0$, where $\mathcal{E}_{s}(\omega,\cdot)\:\RR\to\RR$ is for all $t\in\RR$ defined by
	\begin{align*}
		\mathcal{E}_{s}(\omega,t):&=\Vert\omega\vert_{\Sigma_{t}}\Vert_{\mathcal{H}^{s}}^{2}+\Vert\partial_{t}\omega\vert_{\Sigma_{t}}\Vert_{\mathcal{H}^{s-1}}^{2}\,.
	\end{align*}
\end{proposition}

\begin{proof}
 We follow a similar strategy as in \cite{BaerWafo2015} (see also the discussion in \cite[Sec.~3.7]{BaerLectureNotes}). As a first step, we try to find suitable bounds for $\partial_{t}\Vert\omega\vert_{\Sigma_{t}}\Vert_{\mathcal{H}^{s}}^{2}$ and $\partial_{t}\Vert\partial_{t}\omega\vert_{\Sigma_{t}}\Vert_{\mathcal{H}^{s-1}}^{2}$. As usual, we decompose $\omega=\d t\wedge\omega_{0}+\omega_{1}$ and recall $\Vert\omega\vert_{\Sigma_{t}}\Vert_{\mathcal{H}^{s}}^{2}=\Vert\omega_{0}\vert_{\Sigma_{t}}\Vert_{\H^{s}}^{2}+\Vert\omega_{1}\vert_{\Sigma_{t}}\Vert_{\H^{s}}^{2}$. Using Lemma \ref{Lemma:L2Bound}, we find
	\begin{equation}\label{eq:FirstPiece}
	\begin{aligned}
		\frac{\d}{\d t}\Vert\omega_{i}\vert_{\Sigma_{t}}\Vert_{\H^{s}}^{2}&\leq  2\mathrm{Re}\{\langle\partial_{t}\E^{\frac{s}{2}}\omega_{i}\vert_{\Sigma_{t}},\E^{\frac{s}{2}}\omega_{i}\vert_{\Sigma_{t}}\rangle_{\L^{2}}\}+c_{1}(t)\cdot\Vert\E^{\frac{s}{2}}\omega_{i}\vert_{\Sigma_{t}}\Vert_{\L^{2}}^{2}\\
		&\leq 2\mathrm{Re}\{\langle\partial_{t}\omega_{i}\vert_{\Sigma_{t}},\omega_{i}\vert_{\Sigma_{t}}\rangle_{\H^{s}}\}+c_{2}(t)\cdot\Vert\omega_{i}\vert_{\Sigma_{t}}\Vert_{\H^{s}}^{2}\, ,
	\end{aligned}
	\end{equation}
	where, in order to obtain the second estimate, we used that the commutator $[\E^{\frac{s}{2}},\partial_{t}]$ is effectively a pseudodifferential operator of order $\leq s$, as one can easily see be calculating its principal symbol, and hence can be bounded in the $\Vert\cdot\Vert_{\H^{s}}$-norm. Following similar steps, we obtain the bound
	\begin{align}\label{eq:TimeDerivativeTerm}
		\frac{\d}{\d t}\Vert\partial_{t}\omega_{i}\vert_{\Sigma_{t}}\Vert_{\H^{s-1}}^{2}\leq 2\mathrm{Re}\{\langle\partial_{t}^{2}\omega_{i}\vert_{\Sigma_{t}},\partial_{t}\omega_{i}\vert_{\Sigma_{t}}\rangle_{\H^{s-1}}\}+c_{3}(t)\cdot\Vert\partial_{t}\omega_{i}\vert_{\Sigma_{t}}\Vert_{\H^{s-1}}^{2}\, .
	\end{align}
	Next, we want to estimate the inner product containing $\partial_{t}^{2}\omega_{i}$. For this, we recall the Weitzenböck formula of Riemannian geometry, which states $\D_{k}=-g^{\mu\nu}\nabla_{\mu}\nabla_{\nu}+\R$, where $\R\:\Gamma(\V_{k})\to\Gamma(\V_{k})$ is a linear differential operator of order zero constructed from the curvature tensor of $(\M,g)$. Now, similar to $\omega$, we decomposing the $k$-form $\D_{k}\omega$ as $\D_{k}\omega=\d t\wedge (\D_{k}\omega)_{0}+(\D_{k}\omega)_{1}$ with $(\D_{k}\omega)_{0}\defeq \partial_{t}\lrcorner (\D_{k}\omega)$. In general, both $(\D_{k}\omega)_{0}$ and $(\D_{k}\omega)_{1}$ will contain both the components $\omega_{0}\in C^{\infty}(\RR,\Omega^{k-1}_{s}(\Sigma_{\bullet}))$ and $\omega_{1}\in C^{\infty}(\RR,\Omega^{k}_{s}(\Sigma_{\bullet}))$ as well as their (first) derivatives, but it is not too hard to see that they have the following general structure:
	\begin{align}\label{eq:WaveOperatorDecomposition}
		(\D_{k}\omega)_{i}=\beta^{-2}\partial_{t}^{2}\omega_{i}+\E_{k-1+i}\omega_{i}+\A_{i}\omega_{0}+\B_{i}\omega_{1}+\C_{i}(\partial_{t}\omega_{0})+\F_{i}(\partial_{t}\omega_{1})
	\end{align}
	where $\E_{k}=\mathds{1}+\Delta_{k}$ with $\Delta_{k}=\delta_{\Sigma}\d_{\Sigma}+\d_{\Sigma}\delta_{\Sigma}$, as usual, and where  
	\begin{align*}
		&\A_{i}\:\Omega_{s}^{k-1}(\Sigma;\CC)\to\Omega_{s-1}^{k-1+i}(\Sigma;\CC)\\
		&\B_{i}\:\Omega_{s}^{k}(\Sigma;\CC)\to\Omega_{s-1}^{k-1+i}(\Sigma;\CC)\\
		&\C_{i}\:\Omega_{s}^{k-1}(\Sigma;\CC)\to\Omega_{s}^{k-1+i}(\Sigma;\CC)\\
		&\F_{i}\:\Omega_{s}^{k}(\Sigma;\CC)\to\Omega_{s}^{k-1+i}(\Sigma;\CC)
	\end{align*}
	are bounded linear differential operators, containing factors of $\beta,\beta^{-1}$ as well as contractions of the second fundamental form $k$ with $\omega_{i}$, $\A_{i},\B_{i}$ of order one and $\C_{i},\F_{i}$ of order zero, differentiating only in $\Sigma_{t}$-direction\footnote{If $(\M=\RR\times\Sigma,g=-\d t^{2}+h)$ is ultrastatic, then clearly $\A_{0}=\B_{1}=-\mathds{1}$ and $\A_{1}=\B_{0}=\C_{i}=\F_{i}=0$.}. Using the decomposition~\eqref{eq:WaveOperatorDecomposition} as well as the Cauchy-Schwarz inequality, we estimate the term involving $\partial_{t}^{2}\omega_{i}$ in~\eqref{eq:TimeDerivativeTerm} further as
	\begin{equation}\label{eq:SecondPiece}
	\begin{aligned}
		2\mathrm{Re}&\{\langle\partial_{t}^{2}\omega_{i}\vert_{\Sigma_{t}},\partial_{t}\omega_{i}\vert_{\Sigma_{t}}\rangle_{\H^{s-1}}\}\\
		=&2\mathrm{Re}\{\langle \beta^{2}((\D_{k}\omega)_{i}-\E\omega_{i}-\A_{i}\omega_{0}-\B_{i}\omega_{1}-\C_{i}(\partial_{t}\omega_{0})-\F_{i}(\partial_{t}\omega_{1}))\vert_{\Sigma_{t}},\partial_{t}\omega_{i}\vert_{\Sigma_{t}}\rangle_{\H^{s-1}}\}\\
		\leq& -2\mathrm{Re}\{\langle\E\omega_{i}\vert_{\Sigma_{t}},\partial_{t}\omega_{i}\vert_{\Sigma_{t}}\rangle_{\H^{s-1}}\}+c_{4}(t)\cdot\bigg( \Vert (\square_{k}\omega)_{i}\vert_{\Sigma_{t}}\Vert_{\H^{s}}+ \Vert \omega_{0}\vert_{\Sigma_{t}}\Vert_{\H^{s}}+\\
		&+\Vert \omega_{1}\vert_{\Sigma_{t}}\Vert_{\H^{s}}+\Vert \partial_{t}\omega_{0}\vert_{\Sigma_{t}}\Vert_{\H^{s-1}} +\Vert \partial_{t}\omega_{1}\vert_{\Sigma_{t}}\Vert_{\H^{s-1}}\bigg)\cdot\Vert \partial_{t}\omega_{i}\vert_{\Sigma_{t}}\Vert_{\sf{H}^{s-1}}
	\end{aligned}
	\end{equation}
	where also used that $\A_{i},\B_{i},\C_{i},\F_{i}$ are bounded in the relevant Sobolev norms. Now, note that 
	\begin{align*}
		\langle\E\omega_{i}\vert_{\Sigma_{t}},\partial_{t}\omega_{i}\vert_{\Sigma_{t}}\rangle_{\H^{s-1}}=\langle\E^\frac{1}{2}\omega_{i}\vert_{\Sigma_{t}},\E^\frac{1}{2}\partial_{t}\omega_{i}\vert_{\Sigma_{t}}\rangle_{\H^{s-1}}=\langle\omega_{i}\vert_{\Sigma_{t}},\partial_{t}\omega_{i}\vert_{\Sigma_{t}}\rangle_{\H^{s}}\, .
	\end{align*}
	Summing~\eqref{eq:FirstPiece} and~\eqref{eq:SecondPiece}, we see that the term $-2\mathrm{Re}\{\omega_{i}\vert_{\Sigma_{t}},\partial_{t}\omega_{i}\vert_{\Sigma_{t}}\rangle_{\H^{s}}\}$ cancels and we are left with the inequality
	\begin{align*}
	\frac{\d}{\d t}\mathcal{E}_{s}(\omega,t)\leq c_{4}(t)\cdot \mathcal{E}_{s}(\omega,t)+\Vert (\square_{k}\omega)\vert_{\Sigma_{t}}\Vert^{2}_{\mathcal{H}^{s-1}}\, .
\end{align*}
The claimed result follows by applying Grönwall's lemma \cite{Gronwall} (see also \cite[Lemma~1.5.1]{BaerLectureNotes}) to a compact time interval.
\end{proof}

Let us equip the space $\Gamma_{s}(\V_{k})$ with the topology induced by the family of seminorms
\begin{align*}	\Vert\omega\Vert_{\mathrm{I},s}\defeq \bigg(\int_{\mathrm{I}}\,\Vert\omega\vert_{\Sigma_{t}}\Vert_{\mathcal{H}^{s}}^{2}\,\mathrm{d}t\bigg)^{\frac{1}{2}}
\end{align*}
labelled by compact time intervals $\mathrm{I}\subset\RR$. Taking a compact exhaustion of $\mathbb{R}$, we conclude that the topology of $\Gamma_{s}(\V_{k})$ is induced by a countable family of seminorms and hence in particular (pseudo)metrizable. However, note that $\Gamma_{s}(\V_{k})$ is in general not a Fréchet space, since not every limit of a convergent sequence in $\Gamma_{s}(\V_{k})$ w.r.t.~$\Vert\cdot\Vert_{\mathrm{I},s}$ is necessarily smooth.

As an immediate consequence of the energy estimate, we obtain the following result.

\begin{corollary}\label{Corollary:Energy}
Assume the Setup~\ref{setup} and let $\omega\in\Gamma_{s}(\V_{k})$ be such that $\partial_{t}\omega\in\Gamma_{s}(\V_{k})$ and $\D_{k}\omega\in\Gamma_{s-2}(\V_{k})$. Then, for every compact time interval $\mathrm{I}=[t_{0},t_{1}]\subset\mathbb{R}$ it holds that
	\begin{align*}
		\Vert\omega\Vert_{\mathrm{I},s}^{2}\leq C\cdot\bigg(\Vert\omega\vert_{\Sigma_{t_{0}}}\Vert_{\mathcal{H}^{s}(\Sigma_{t_{0}})}^{2}+\Vert\partial_t\omega\vert_{\Sigma_{t_{0}}}\Vert_{\mathcal{H}^{s-1}(\Sigma_{t_{0}})}^{2}+\Vert\D_{k}\omega\Vert_{\mathrm{I},s}^{2}\bigg)
	\end{align*}
\end{corollary}

We omit the proof and refer to \cite[Corollary 12]{BaerWafo2015}. We can finally prove the main result of this section.

\begin{proof}[Proof of Theorem~\ref{thm:sob Cauchy}]
The existence and the uniqueness of the solution of the Cauchy problem for smooth Cauchy data is well-know, (see e.g.~\cite{BaerBook}), so we only have to show that we our choice of Cauchy data the solution belong to $\Gamma_s(\V_k)$. To this end, we recast that the solution to normally hyperbolic PDEs enjoy finite propagation of speed, namely if $(f,{h_{1}},{h_{2}})\in\Gamma_{\sc}(\V_{k})\times\Gamma_{\c}(\V_{\rho_{k}})$, then $u\in\Gamma_{\sc}(\V_k)$.
Therefore, consider a sequence of Cauchy data $(f_n,{h_{1}}_n,{h_{2}}_n)\in\Gamma_{\sc}(\V_{k})\times\Gamma_{\c}(\V_{\rho_{k}})$ which converge  to an element $(f^*,{h_1^*},{h_2^*})\in\Gamma_{s-2}(\V_{k})\times\Gamma_{s}(\V_{\rho_{k}})$ (in the intersection topology)  and denote the corresponding solution of the Cauchy problem with $u_n\in\Gamma_{\sc}(\V_k)$. Then, the limit $\lim u_n$ will also converge to a smooth solution $u$ to the Cauchy problem and, by Corollary~\ref{Corollary:Energy}, we can conclude that $u\in\Gamma_s(\V_k)$. 
\end{proof}

Moreover, we recall that the normally hyperbolic operators $\D_{k}$ are in particular \textit{Green hyperbolic}, which means that there exist (unique) Green operators $\G_{k}^{+}$ (resp.~$\G_{k}^{-}$), which are the inverses of $\D_{k}$ when restricted to past (resp.~future) compactly supported sections. In the setting of Sobolev spaces, Theorem~\ref{thm:sob Cauchy} gives rise to the following characterization:

\begin{proposition}\label{Prop:Green}
Let $(\M,g)$  as in the Setup~\ref{setup}. Then, there exist (unique) linear operators $\G_{k}^{\pm}\colon\Gamma_{\tc,s-2}(\V_{k})\to\Gamma_{s}(\V_{k})$,
    called \textit{advanced/retarded Green operators}, satisfying the following conditions:
    \begin{align*}
        \text{(i)}&\hspace*{1cm}\G_{k}^{\pm}\circ \D_{k}\vert_{\Gamma_{\tc,s}}=\D_{k}\circ\G_{k}^{\pm}=\Id_{\Gamma_{\tc,s}}\\
        \text{(ii)}&\hspace*{1cm}\mathrm{supp}(\G_{k}^{\pm}s)\subset \J^{\pm}(\mathrm{supp}(s))\hspace*{1cm}\forall s\in\Gamma_{\mathrm{tc}}(\V)
\end{align*}
Furthermore, we define the \textit{causal propagator} by $\G_{k}\defeq\G_{k}^{+}-\G_{k}^{-}:\Gamma_{\tc,s-2}(\V_{k})\to\Gamma_{s}(\V_{k})$. Then, the following sequence is exact and forms a complex:
\begin{align*}
	0\xrightarrow{\hspace*{1cm}}\Gamma_{\tc,s}(\V_{k})\xrightarrow{\hspace*{0.5cm}\D_{k}\hspace*{0.5cm}}\Gamma_{\tc,s-2}(\V_{k})\xrightarrow{\hspace*{0.5cm}\G_{k}\hspace*{0.5cm}}\Gamma_{s}(\V_{k})\xrightarrow{\hspace*{0.5cm}\D_{k}\hspace*{0.5cm}}\Gamma_{s-2}(\V_{k})\xrightarrow{\hspace*{1cm}}0\, .
\end{align*}
\end{proposition}

\begin{proof}
	Let $f\in\Gamma_{\tc,s-2}(\V_{k})$. By assumption, there exists a time $t_{0}\in\RR$ such that $\mathrm{supp}(f)\subset \J^{+}(\Sigma_{t_{0}})$. We set $\G_{k}^{+}f\defeq u$ where $u$ is the unique solution of the Cauchy problem
	\begin{align*}
		\begin{cases}
			\D_{k}u &=f\\
			u\vert_{\Sigma_{0}}&=0\\
			\partial_{t}u\vert_{\Sigma_{0}}&=0
		\end{cases}
	\end{align*}
	By Theorem~\ref{thm:sob Cauchy}, it follows that $u\in\Gamma_{s}(\V_{k})$ and hence $\G_{k}^{+}\:\Gamma_{\tc,s-2}(\V_{k})\to\Gamma_{s}(\V_{k})$. Similarly, we construct $\G_{k}^{-}\:\Gamma_{\tc,s-2}(\V_{k})\to\Gamma_{s}(\V_{k})$ by choosing $t_{1}$ such that $\mathrm{supp}(f)\subset \J^{-}(\Sigma_{t_{1}})$. Clearly, (i) and (ii) are satisfied. For uniqueness, we recall that the Cauchy problem of $\D_{k}$ is well-posed for arbitrary smooth Cauchy data. In particular, following the same construction as above, we obtain more general Green operators $\G^{+}\:\Gamma_{\pc}(\V_{k})\to\Gamma_{\pc}(\V_{k})$ and $\G^{-}\:\Gamma_{\fc}(\V_{k})\to\Gamma_{\fc}(\V_{k})$ with similar properties. By construction, these operators are the inverses of $\D_{k}\vert_{\Gamma_{\pc}}$ resp.~$\D_{k}\vert_{\Gamma_{\fc}}$ and hence in particular unique. Uniqueness of $\G_{k}^{\pm}$ then follows from the fact that they are the restrictions to $\Gamma_{\tc,s-2}(\V_{k})$. For more details, we refer to \cite{Baer2015}. 
	
	Let us now turn to the claimed exact complex: The sequence displayed above is clearly a complex, i.e.~$\G_{k}\circ\D_{k}\vert_{\Gamma_{\tc,s}}=0$ and $\D_{k}\circ\G_{k}=0$. For exactness, we follow the proof in \cite[Theorem 3.5.]{BaerGinoux2011} adapted to the setting of Sobolev spaces: For exactness at $\Gamma_{\tc,s}(\V_{k})$, we observe that any $\omega\in\ker(\square_{k})\cap\Gamma_{\tc,s}(\V_{k})$ satisfies $\omega=\G^{\pm}\D_{k}\omega=0$. For exactness at $\Gamma_{\tc,s-2}(\V_{k})$, we have to show $\ker(\G_{k}\vert_{\Gamma_{\tc,s-2}})=\ran(\D_{k}\vert_{\Gamma_{\tc,s}})$. The direction ``$\supset$'' is clear. For ``$\subset$'', let $\omega\in\Gamma_{\tc,s-2}(\V_{k})$ such that $\G_{k}\omega=0$. We set $\eta\defeq \G_{k}^{+}\omega=\G_{k}^{-}\omega$. It follows that $\eta\in\Gamma_{\tc,s}$ and $\square\eta=\omega$. For exactness at $\Gamma_{s}(\V_{k})$, we need to show $\ker(\D_{k}\vert_{\Gamma_{s}})=\ran(\G_{k}\vert_{\Gamma_{\tc,s-2}})$. The direction ``$\supset$'' is clear. For ``$\subset$'', let $\omega\in\Gamma_{s}(\V_{k})$ be such that $\D_{k}\omega=0$. We take cutoff functions $\chi^{\pm}$ with $\mathds{1}=\chi^{+}-\chi^{-}$ and $\mathrm{supp}(\chi^{\pm}\omega)\subset\J^{\pm}(\Sigma_{\pm})$ for suitable Cauchy hypersurfaces $\Sigma_{\pm}$ with the property that $\J^{+}(\Sigma_{+})\cap\J^{-}(\Sigma_{-})\neq\emptyset$. Set $\omega^{\pm}\defeq \chi^{\pm}\omega$ and $\eta\defeq \D_{k}\omega^{+}=\D_{k}\omega^{-}$. Then, clearly $\mathrm{supp}(\eta)\subset\J^{+}(\Sigma_{+})\cap\J^{-}(\Sigma_{-})$, hence $\eta\in\Gamma_{\tc}(\V_{k})$ and $\G^{\pm}\eta=\omega^{\pm}$. We conclude that $\G\eta=\omega$. Last but not least, exactness at $\Gamma_{s-2}(\V_{k})$ follows from Theorem~\ref{thm:sob Cauchy}.
\end{proof}

A straightforward computation shows that the Green operators $\G^{\pm}$ are formal adjoints of each other w.r.t.~$(\cdot,\cdot)_{\V_{k}}$. In particular, the causal propagator is formally anti self-adjoint. Furthermore, the exact sequence above establishes the isomorphism
\begin{align*}
		[\G_{k}]\:\frac{\Gamma_{\tc,s-2}(\V_{k})}{\mathrm{ran}(\D_{k}\vert_{\Gamma_{\tc,s}})}\xrightarrow{\cong}\mathrm{ker}(\D_{k}\vert_{\Gamma_{s}})\, .
\end{align*}

\begin{remark}\label{Remark:ExactSequenceExtended}
	We remark that the extended causal propagator $\G\:\Gamma_{\tc}(\V_{k})\to\Gamma(\V_{k})$ induces a similar exact sequence of the form 
\begin{align*}
	0\xrightarrow{\hspace*{1cm}}\Gamma_{\tc}(\V_{k})\xrightarrow{\hspace*{0.5cm}\D_{k}\hspace*{0.5cm}}\Gamma_{\tc}(\V_{k})\xrightarrow{\hspace*{0.5cm}\G\hspace*{0.5cm}}\Gamma(\V_{k})\xrightarrow{\hspace*{0.5cm}\D_{k}\hspace*{0.5cm}}\Gamma(\V_{k})\xrightarrow{\hspace*{1cm}}0\, .
\end{align*}
\end{remark}
\section{The Quantization of Maxwell Theory in the Cauchy Radiation Gauge}\label{Sec:PhaseSpaceQuant}
The aim of this section is to fix completely the gauge degrees of freedom of Maxwell theory and to discuss different representations of the phase space for Sobolev initial data. For this purpose, we will adapt the abstract formalism for linear gauge theories introduced by Hack-Schenkel in~\cite{HackSchenkel2013} to the setting of Sobolev spaces and discuss the notion of Hadamard states.
For the reader interested in the quantization of a scalar free field theory, we recommend~\cite{AQFT1,Gerard2019}.
\subsection{Phase Space and Complete Gauge Fixing}\label{Sec:PhaseSpace}
As before, we consider a Maxwell bundle $(\V_{k},(\cdot,\cdot)_{\V_{k}})$ introduced in Section~\ref{Sec:Maxwell}, which is defined over an $(n+1)$-dimensional globally hyperbolic manifold $(\M,g)$ as in the Setup~\ref{setup}. Since $(\Sigma,h_t)$ is complete for all $t\in\RR$ we denote the space of smooth spacelike Sobolev fields of degree $s=\infty$ and the corresponding space of initial data by
\begin{align*}
	\Gamma_{\infty}(\V_{k})&=C^\infty\big(\RR,\Omega_\infty^{k-1}(\Sigma_{\bullet})\big)\oplus C^\infty\big(\RR,\Omega_{\infty}^{k}(\Sigma_{\bullet})\big)\\
	\Gamma_{\infty}(\V_{\rho_{k}})&=\mathcal{H}_{k}^{\infty}(\Sigma_{0})\oplus\mathcal{H}_{k}^{\infty}(\Sigma_{0})
\end{align*}
with $\mathcal{H}_{k}^{\infty}(\Sigma)=\Omega_{\infty}^{k-1}(\Sigma)\oplus\Omega_{\infty}^{k}(\Sigma)$, as in Section \ref{Sec:CauchyProblem}. The corresponding Cauchy data map $\rho_{k}:\omega\mapsto (\omega\vert_{\Sigma_{0}},\partial_{t}\omega\vert_{\Sigma_{0}})$ is hence well-defined as a map $\rho_{k}\:\Gamma_{\infty}(\V_{k})\to\Gamma_{\infty}(\V_{\rho_{k}})$. 

\begin{definition}
	We consider the following subspaces of $\Gamma_{\infty}(\V_{1})$ and $\Gamma_{\infty}(\V_{\rho_{1}})$:
	\begin{align*}
		\Gamma_{\infty,\d}(\V_{1})&=\{\omega=\omega_{0}\d t+\omega_{1}\in\Gamma_{\infty}(\V_{1})\mid \omega_{1}\vert_{\Sigma_{0}},\partial_{t}\omega_{1}\vert_{\Sigma_{0}}\in\Omega^{1}_{\infty,\d}(\Sigma_{0})\}\\
	\Gamma_{\infty,\d}(\V_{\rho_{1}})&= \mathcal{H}^{\infty}_{1,\d}(\Sigma_{0})\oplus \mathcal{H}^{\infty}_{1,\d}(\Sigma_{0})
	\end{align*}
where $\mathcal{H}^{\infty}_{1,\d}(\Sigma_{0})\defeq \Omega^{0}_{\infty}(\Sigma_{0})\oplus\Omega^{1}_{\infty,\d}(\Sigma_{0})\subset\mathcal{H}^{\infty}_{1}(\Sigma_{0})$. Clearly, $\rho_{1}\:\Gamma_{\infty,\d}(\V_{1})\to\Gamma_{\infty,\d}(\V_{\rho_{1}})$.
\end{definition}

\begin{remark}
	Note that the definition $\Gamma_{\infty,\d}(\V_{1})$ depends on the choice of initial hypersurface $\Sigma_{0}$, since the condition $\omega_{1}\vert_{\Sigma_{0}}\in\Omega^{1}_{\infty,\d}(\Sigma_{0})$ does in general not propagate.
\end{remark}

Using the discussion of Section~\ref{sec:Hodge}, we obtain the following generalization of Proposition~\ref{Prop:CauchyRadGauge} for non-compact manifolds. We start by proving the following preliminary Lemma:

\begin{corollary}\label{Corollary:Gaugefixing}
Let $(\M,g)$ be a globally hyperbolic manifold as in the Setup~\ref{setup}.
 For any $A\in\Gamma_{\infty}(\V_{1})$, there exists a unique (up to a constant) $f\in\Gamma(\V_{0})$ with the property $\d f\in\Gamma_{\infty,\d}(\V_{1})$ s.t.~$A^{\prime}\defeq A-\d f$ satisfies the Cauchy radiation gauge condition. 
\end{corollary}

\begin{proof}
	The proof is analogues to the proof of Proposition~\ref{Prop:CauchyRadGauge}: Let us decompose $A=A_{0}\d t+A_{\Sigma}$ with $A_{0}=\partial_{t}\lrcorner A$. The condition for $A^{\prime}$ to satisfy the Cauchy radiation gauge amounts to solving the system 
    \begin{align*}
        \begin{cases}
            \D_{0}f&=\delta A\\
            \pi &=A_{0}\vert_{\Sigma_{0}}\\
            \DeltaZS  a&=\delta_{\Sigma}A_{\Sigma_{0}}\vert_{\Sigma}
        \end{cases}
    \end{align*}
    where $a\defeq f\vert_{\Sigma_{0}}$ and $\pi\defeq\partial_{t}f\vert_{\Sigma_{0}}$ denote the initial data of $f$. By assumption, $A_{\Sigma}\vert_{\Sigma_{0}}\in\Omega^{1}_{\infty,\d}(\Sigma_{0})$ and hence, by Proposition~\ref{prop:Poisson}, there exists a unique $a\in C^{\infty}(\Sigma;\CC)$ with the property $\d_{\Sigma}a\in\Omega^{1}_{\infty,\d}(\Sigma_{0})$ s.t.~$\DeltaZS a=\delta_{\Sigma}A_{\Sigma}\vert_{\Sigma_{0}}$. Therefore, the system above has a unique solution $f\in\Gamma(\V_{0})$ up to constant \cite[Chapter 3]{BaerBook}, which satisfies $\d_{\Sigma}f\vert_{\Sigma_{0}}\in\Omega^{1}_{\infty,\d}(\Sigma_{0})$. It remains to show that $\d f\in\Gamma_{\infty,\d}(\V_{1})$. For this, we observe that $\d f$ is the unique solution of the Cauchy problem 
    \begin{align*}
    		\begin{cases}
    			\square_{1}\d f&=\d\delta A\in\Gamma_{\infty}(\V_{1})\\
    			\d f\vert_{\Sigma_{0}}&=(A_{0}\d t + \d_{\Sigma}a)\vert_{\Sigma_{0}}\in \mathcal{H}_{\infty,\d}^{1}(\Sigma_{0})=\Omega^{0}_{\infty}(\Sigma_{0})\oplus\Omega^{1}_{\infty,\d}(\Sigma_{0})\\
    		\partial_{t}\d f\vert_{\Sigma_{0}}&=((\partial_{t}A_{0}+\beta h^{-1}(\d_{\Sigma}a-A_{\Sigma},\d_{\Sigma}\beta))\d t+\d_{\Sigma}A_{0})\vert_{\Sigma_{0}}\in \mathcal{H}_{\infty,\d}^{1}(\Sigma_{0})=\Omega^{0}_{\infty}(\Sigma_{0})\oplus\Omega^{1}_{\infty,\d}(\Sigma_{0})
    		\end{cases}
    \end{align*}
	where we used that $\d_{\Sigma}A_{0}\vert_{\Sigma_{0}}\in\d_{\Sigma}\Omega^{0}_{\infty}(\Sigma_{0})\subset\Omega^{1}_{\infty,\d}(\Sigma_{0})$ and where we rewrote the initial data as
	\begin{align*}
		\partial_{t}f\vert_{\Sigma_{0}}&=A_{0}\vert_{\Sigma_{0}}\\
		\d_{\Sigma}f\vert_{\Sigma_{0}}&=\d_{\Sigma}a\\
		\partial_{t}^{2}f\vert_{\Sigma_{0}}&=\partial_{t}A_{0}\vert_{\Sigma_{0}}+\beta h^{-1}(\d_{\Sigma}a-A_{\Sigma},\d_{\Sigma}\beta)\vert_{\Sigma_{0}}\\
		\partial_{t}\d_{\Sigma}f\vert_{\Sigma_{0}}&=\d_{\Sigma}\partial_{t}f\vert_{\Sigma_{0}}=\d_{\Sigma}A_{0}\vert_{\Sigma_{0}}
\end{align*}	    
	In the third line, we used $\nabla_{0}^{2}f\vert_{\Sigma_{0}}=\nabla_{0}A_{0}\vert_{\Sigma_{0}}$ and $\nabla_{0}A_{0}=\partial_{t}A_{0}-\beta^{-1}\mathrm{tr}_{h}(k)A_{0}\partial_{t}\beta-\beta h^{-1}(A_{\Sigma},\d_{\Sigma}\beta)$. We conclude that $\d f\in\Gamma_{\infty,\d}(\V_{1})$ by Theorem~\ref{thm:sob Cauchy}.
\end{proof}

After this preliminary discussion, we now move on to the phase space of Maxwell theory and its quantization. We will follow the general strategy for quantizing linear gauge theories developed by Hack-Schenkel \cite{HackSchenkel2013} and refined by Gérard-Wrochna~\cite{GerardWrochnaYM2015} (see also the presentation in \cite{GerardMurroWrochna2022})\footnote{For the relation of the Hack-Schenkel formalism to the BRST formalism, we refer the reader to \cite{WrochnaZahn2017}.}, which covers many important examples including linearized Yang-Mills theory as well as linearized gravity. The Hack-Schenkel formalism is usually discussed in the setting of fields with spacelike compact support. In the following, we shall adopt it, in the special case of Maxwell theory, to smooth fields, which are spatially in a Sobolev space, and to complete gauge fixing.\bigskip

Consider the Maxwell operator $\P\:\Gamma(\V_{1})\to\Gamma(\V_{1})$ given by $\P=\delta\d$. In order to incorporate the Cauchy radiation gauge, let us introduce some notation for the Cauchy data: We consider the decompositions
\begin{align}\label{NotationInitialData}
	\Gamma(\V_{\rho_{0}})\cong\Omega^{0}(\Sigma;\CC)^{2},\hspace*{2cm}\Gamma(\V_{\rho_{1}})\cong\Omega^{0}(\Sigma;\CC)^{2}\oplus\Omega^{1}(\Sigma;\CC)^{2}
\end{align}
and parametrize the Cauchy data maps $\rho_i\:\Gamma(\V_{i})\to\Gamma(\V_{\rho_{i}})$ w.r.t.~this decomposition as 
\begin{align*}
    \rho_{0}\: f \mapsto \begin{pmatrix}f\vert_{\Sigma_{0}}\\\i^{-1}\nabla_{0}f\vert_{\Sigma_{0}}\end{pmatrix}\hspace*{0.5cm}\text{and}\hspace*{0.5cm}
    \rho_{1}\: A \mapsto \begin{pmatrix}A_{0}\vert_{\Sigma_{0}}\\\i^{-1}\partial_{t}A_{0}\vert_{\Sigma_{0}}\\A_{\Sigma}\vert_{\Sigma_{0}}\\\i^{-1}\partial_{t}A_{\Sigma}\vert_{\Sigma_{0}}\end{pmatrix}\, .
\end{align*}
where we decompose $\Gamma(\V_{1})\ni A=A_{0}\d t+A_{\Sigma}$ as usual. Furthermore, we consider the \textit{Cauchy evolution operators} $\U_{i}\:\Gamma(\V_{\rho_{i}})\to\ker(\D_{i})$, which are the inverses of $\rho_{i}$ restricted to the right domain. With this notation, we introduce the operator
\begin{align*}
	\K\:\Gamma(\V_{0})\to\Gamma(\V_{1}),\qquad \K\defeq \d\, .
\end{align*}
Its formal adjoint w.r.t.~to $(\cdot,\cdot)_{\V_{i}}$ is the operator $\K^{\ast}\:\Gamma(\V_{1})\to\Gamma(\V_{0})$ given by $\K^{\ast}=\delta$, which parametrizes the Lorenz gauge condition. On the level of Cauchy data, we introduce the operators
\begin{align*}
	\K_{\Sigma}\:\Gamma(\V_{\rho_{0}})&\to\Gamma(\V_{\rho_{1}}),\qquad \K_{\Sigma}\defeq \rho_{1}\K\U_{0}\\
	\K_{\Sigma}^{\dagger}\:\Gamma(\V_{\rho_{1}})&\to\Gamma(\V_{\rho_{0}}),\qquad \K_{\Sigma}^{\dagger}\defeq \rho_{0}\K^{\ast}\U_{1}\, .
\end{align*} 
The notation $\K^{\dagger}_{\Sigma}$ will become clear as soon as a suitable Hermitian form will be introduced (see the end of Subsection~\ref{Sec:PhaseSpace}). Furthermore, note that $\K_{\Sigma}^{\dagger}\K_{\Sigma}=0$. Similarly, we encode the Cauchy temporal gauge both globally and on the level of initial data by introducing the operators
\begin{align*}
	\R_{\Sigma}\:\Gamma(\V_{\rho_{1}})&\to\Gamma(\V_{\rho_{1}}),\qquad \R_{\Sigma}\defeq\begin{pmatrix}\mathds{1} & 0&0&0\\ -\beta^{-1}\partial_{t}\beta\vert_{\Sigma_{0}} & \mathds{1} & -\beta h^{-1}(\d_{\Sigma}\beta,\cdot )\vert_{\Sigma_{0}} &0\\0&0&0&0\\0&0&0&0\end{pmatrix}\\
	\R\:\Gamma(\V_{1})&\to\Gamma(\V_{1}),\qquad\hspace*{0.2cm} \R\defeq \U_{1}\R_{\Sigma}\rho_{1}\, .
\end{align*}
In other words, $A\in\ker(\R)$ if and only if $A_{0}\vert_{\Sigma_{0}}=\nabla_{0}A_{0}\vert_{\Sigma_{0}}=0$, where we used the general formula
\begin{align*}
	\nabla_{0}A_{0}=\partial_{t}A_{0}-A_{0}\beta^{-1}\partial_{t}\beta-\beta h^{-1}(\d_{\Sigma}\beta,A_{\Sigma})\end{align*}
	 to relate $\nabla_{0}A_{0}\vert_{\Sigma_{0}}$ to the initial data $(A_{0}\vert_{\Sigma_{0}},\partial_{t}A_{0}\vert_{\Sigma_{0}},A_{\Sigma}\vert_{\Sigma_{0}},\partial_{t}A_{\Sigma}\vert_{\Sigma_{0}})$ of $A\in\Gamma(\V_{1})$.
	
\begin{proposition}\label{Prop:EquPhaseSpaces} Let $(\M,g)$ be a globally hyperbolic manifold as in the Setup~\ref{setup}. The following diagram is commutative and every map is an isomorphism:
\begin{equation*}
\begin{tikzcd}
   	 \PS\defeq\cfrac{\ker(\K^{\ast}\vert_{\Gamma_{\tc,\infty}})}{\ran(\P\vert_{\Gamma_{\tc}})\cap\Gamma_{\G}(\V_{1})} \arrow[r,"{[\G_{1}]}"]\arrow[dr,"{[\G_{1}]}"]\arrow[d,"{[\rho_1\G_{1}]}"] &\cfrac{\ker(\P\vert_{\Gamma_{\infty}})}{\ran(\K)\cap\Gamma_{\infty,\d}(\V_{1})} \\
    		\PSS\defeq\cfrac{\ker(\K_\Sigma^{\dagger}\vert_{\Gamma_{\infty}})}{\ran(\K_\Sigma)\cap\Gamma_{\infty,\d}(\V_{\rho_{1}})}\arrow[r,"{[\U_{1}]}"]\arrow[d,dashed,"\T_{\Sigma}"] & \cfrac{\ker(\D_{1}\vert_{\Gamma_{\infty}})\cap\ker(\K^{\ast}\vert_{\Gamma_{\infty}})}{\K(\ker(\D_{0}))\cap\Gamma_{\infty,\d}(\V_{1})} \arrow[u,hookrightarrow]\\
\PSR \defeq\ker(\K_{\Sigma}^{\dagger}\vert_{\Gamma_{\infty}})\cap\ker(\R_{\Sigma}\vert_{\Gamma_{\infty}})     \arrow[r,"\U_1"] & \ker(\D_{1}\vert_{\Gamma_{\infty}})\cap\ker(\K^{\ast}\vert_{\Gamma_{\infty}})\cap\ker(\R\vert_{\Gamma_{\infty}})\arrow[u,hookrightarrow]
    \end{tikzcd}
\end{equation*}
where the maps denoted by ``\,$\hookrightarrow$'' are induced by the inclusions and where 
\begin{align*}
\Gamma_{\G}(\V_{1}):= \{\omega\in\Gamma_{\tc,\infty}(\V_{1})\mid\G_{1}\omega\in\Gamma_{\infty,\d}(\V_{1})\}\, .
\end{align*}
\end{proposition}

\begin{proof}
	The two isomorphisms ``$\hookrightarrow$'' induced by the inclusion are clear and a direct consequence of Corollary \ref{Corollary:Gaugefixing}, i.e.~the fact that for any $A\in\ker(\P\vert_{\Gamma_{\infty}})$, we can find a unique $f\in\Gamma(\V_{1})$ (up to constant) such that $\K f\in\Gamma_{\infty,\d}(\V_{1})$ and such that $A+\K f$ satisfies the Cauchy radiation gauge. Furthermore, the isomorphism $\U_{1}$ on the bottom of the diagram is clear by definition, since $\U_{1}$ is bijective as a map $\U_{1}\:\Gamma_{\infty}(\V_{\rho_{1}})\to\ker(\D_{1}\vert_{\Gamma_{\infty}})$. 
	
In the following, we show that the map $[\G_{1}]$ on top is an isomorphism. The remaining isomorphisms are proven analogously and commutativity of the diagram is clear.  Bijectivity amounts to showing that
	\begin{align*}
		\text{injectivity:}\hspace*{0.5cm}& \G^{-1}_{1}(\ran(\K)\cap\Gamma_{\infty,\d}(\V_{1}))\cap\ker(\K^{\ast}\vert_{\Gamma_{\tc,\infty}})=\ran(\P\vert_{\Gamma_{\tc}})\cap\Gamma_{\G}(\V_{1})\\
		\text{surjectivity:}\hspace*{0.5cm}& \ker(\P\vert_{\Gamma_{\infty}})=\G_{1}(\ker(\K^{\ast}\vert_{\Gamma_{\tc,\infty}}))+(\ran(\K)\cap\Gamma_{\infty,\d}(\V_{1}))
	\end{align*}
	For ``$\supset$'' in injectivity, let $\omega=\P\eta$ with $\eta\in\Gamma_{\tc}(\V_{1})$ such that $\omega\in\Gamma_{\G}(\V_{1})$, i.e.~$\omega\in\Gamma_{\tc,\infty}(\V_{1})$ and $\G_{1}\omega\in\Gamma_{\infty,\d}(\V_{1})$.  Clearly $\K^{\ast}\omega=\K^{\ast}\P\eta=0$ and 
	\begin{align*}
		\G_{1}\omega=\G_{1}\P\eta=-\G_{1}\K\K^{\ast}\eta=\K(-\G_{0}\K^{\ast}\eta)\in\ran(\K)\cap\Gamma_{\infty,\d}(\V_{1})
	\end{align*}
	For ``$\subset$'', let $\omega\in\ker(\K^{\ast}\vert_{\Gamma_{\tc,\infty}})$ be such that $\G_{1}\omega\in\Gamma_{\infty,\d}(\V_{1})$ and $\G_{1}\omega=\K f$ for some $f\in\Gamma(\V_{0})$. Then $0=\K^{\ast}\omega=\D_{0}f$ implies that $f=\G_{0}g$ for some $g\in\Gamma_{\tc}(\V_{0})$. Now, $\G_{1}\omega=\K f=\G_{1}\K g$ implies that there exists $\eta\in\Gamma_{\tc}(\V_{1})$ such that $\K g-\omega=\D_{1}\eta$. Note that $\square_{0}\K^{\ast}\eta=\square_{0} g$ and hence $\D_{0}(g-\K^{\ast}\eta)=0$. Since $g-\K^{\ast}\eta\in\Gamma_{\tc}(\V_{0})$, we conclude by Remark \ref{Remark:ExactSequenceExtended}  that $g=\K^{\ast}\eta$. It follows that 
	\begin{align*}
		\omega=\K g-\D_{1}\eta=\K(g-\K^{\ast}\eta)-\P\eta=\P(-\eta)
	\end{align*}
	and hence $\omega\in\mathrm{ran}(\P\vert_{\Gamma_{\tc}})$. For surjectivity, ``$\supset$'' is clear. For ``$\subset$'', we pick $\omega\in\ker(\P\vert_{\Gamma_{\infty}})$. First, by Corollary~\ref{Corollary:Gaugefixing}, we can always find a gauge transformation $\omega=\omega+\K f$ with $f\in\Gamma(\V_{0})$ and $\K f\in\Gamma_{\infty,\d}(\V_{1})$ such that $\omega\in\ker(\K^{\ast}\vert_{\Gamma_{\infty}})$. Therefore, w.l.o.g.~we assume $\omega\in\ker(\P\vert_{\Gamma_{\infty}})\cap\ker(\K^{\ast}\vert_{\Gamma_{\infty}})$ and hence in particular $\D_{1}\omega=0$. It follows that $\omega=\G_{1}\eta$ for some $\eta\in\Gamma_{\tc,\infty}(\V_{1})$. Furthermore, $\K^{\ast}\omega=0$ implies $\G_{0}\K^{\ast}\eta=0$ and hence $\K^{\ast}\eta=\D_{0}g$ for some $g\in\Gamma_{\tc,\infty}(\V_{0})$ (since $\K^{\ast}\eta\in\Gamma_{\tc,\infty}(\V_{0}))$. Then, a straight-forward calculation shows that $f\defeq \G_{0}g$ is the required gauge transform, i.e.~$\omega-\K f=\G_{1}(\eta-\K g)\in\G_{1}(\ker(\K^{\ast}\vert_{\Gamma_{\tc,\infty}}))$. Observe that $f=\G_{0}g\in\Gamma_{\infty}(\V_{0})$ and hence $\K f\in\Gamma_{\infty,\d}(\V_{1})$ since clearly $\K (\Gamma_{\infty}(\V_{0}))\subset\Gamma_{\infty,\d}(\V_{1})$.
\end{proof}

We equip the bundles of initial data $\V_{\rho_{i}}$ with Hermitian bundle metrics $\langle\cdot,\cdot\rangle_{\V_{\rho_{i}}}$ and denote the corresponding Hermitian forms on sections by $(\cdot,\cdot)_{\V_{\rho_{i}}}\:\Gamma_\infty(\V_{\rho_{i}})\times\Gamma_\infty(\V_{\rho_{i}})\to\CC$. Next, we define linear operators\footnote{Using Green's formula it can be shown that $\G_{i,\Sigma}$ are in fact linear differential operators.} $\G_{i,\Sigma}\:\Gamma_\infty(\V_{\rho_{i}})\to\Gamma_\infty(\V_{\rho_{i}})$, which  are uniquely determined by the relation
\begin{align}\label{GSigma}
	\G_{i}=(\rho_{i}\G_{i})^{\ast}\G_{i,\Sigma}(\rho_{i}\G_{i})\, ,
\end{align}
where $(\rho_{i}\G_{i})^{\ast}$ denotes the adjoint of $\rho_{i}\G_{i}$ w.r.t.~$(\cdot,\cdot)_{\V_{i}}$ and $(\cdot,\cdot)_{\V_{\rho_i}}$. With this notation, we define Hermitian forms by
\begin{align*}
	\q_{i,\Sigma}\:\Gamma_{\infty}(\V_{\rho_{i}})\times\Gamma_{\infty}(\V_{\rho_{i}})\to\mathbb{C},\quad \q_{i,\Sigma}(\cdot,\cdot)\defeq \i (\cdot,\G_{i,\Sigma}\cdot)_{\V_{\rho_{i}}} \, .
\end{align*}
Note that the choice of fibre metrics $\langle\cdot,\cdot\rangle_{\V_{\rho_{i}}}$ on $\V_{\rho_{i}}$ was arbitrary, however, the Hermitian forms $q_{i,\Sigma}$ are uniquely fixed by $\G_{i}$ and $\langle\cdot,\cdot\rangle_{\V_{i}}$. Furthermore, note that the operator $\K_{\Sigma}^{\dagger}$ is the formal adjoint of $\K_{\Sigma}$ w.r.t.~the Hermitian forms $\q_{i,\Sigma}$.

\begin{definition}[Phase Space]
	We equip the vector spaces $\PS$ and $\PSS$ with the Hermitian forms
	\begin{align*}
		\q_{1}\:\PS\times\PS&\to\CC,\qquad \q_{1}([\cdot],[\cdot])\defeq \i(\cdot,\G_{1}\cdot)_{\V_{1}}\\
		\q_{1,\Sigma}\:\PSS\times\PSS&\to\CC,\qquad \q_{1,\Sigma}([\cdot],[\cdot])\defeq \i(\cdot,\G_{1,\Sigma}\cdot)_{\V_{\rho_{1}}}\, .
	\end{align*}
\end{definition}

By Proposition~\ref{Prop:EquPhaseSpaces}, the map $[\rho_{1}\G_{1}]$ provides a unitary isomorphism
\begin{align*}
	[\rho_{1}\G_{1}]\:(\PS,\q_{1})\xrightarrow{\cong}(\PSS,\q_{1,\Sigma})\, .
\end{align*}
\subsection{Cauchy Surface Covariance and Hadamard States}\label{HadamardDefinition}
The quantization of a linear, free field gauge
theory is realizes by a two-step procedure: First, one assigns to the classical phase space $(\PS,\q_1)$  a unital $*$-algebra $\CCR$ abstractly generated by symbols $\Id,\Phi(v),\Phi^{\ast}(w)$ for all $v,w\in\PS$, such that 
\begin{itemize}
\item[$\bullet$]the assignment $v\mapsto\Phi(v)$  and $w\mapsto\Phi^{\ast}(w)$ are respectively $\mathbb{C}$-anti-linear and $\mathbb{C}$-linear;
\item[$\bullet$]$\Phi(v)^{\ast}=\Phi^{\ast}(v)$;
\item[$\bullet$]the \textit{canonical commutation relations} are fulfilled:
\begin{align*}
	&[\Phi(v),\Phi(w)]=[\Phi^{\ast}(v),\Phi^{\ast}(w)]=0\, ,\\
	&[\Phi(v),\Phi^{\ast}(w)]=\q_1(v,w)\Id \,.
\end{align*}
\end{itemize}
Then, one
determines the admissible \textit{physical states} of the system by identifying a suitable subclass of
the linear, positive, i.e.~$\omega(a^{\ast}a)\geq 0$, and normalized, i.e.~$\omega(\Id)=1$, functionals $\omega\: \CCR \to \CC$.  Once that a state is specified, the Gelfand-Naimark-Segal (GNS) construction
guarantees the existence of a representation of the quantum field algebra as operators define on a common dense subspace of some Hilbert space.  In this paper, we will not  worry about the explicit construction of such representation and we will restrict our attention to the subclass of so-called {\it quasifree states}, which are fully
determined by a pair of {\it (spacetime) covariances}, i.e.~Hermitian forms $\Lambda^{\pm}\:\PS\times\PS\to\CC$ defined by 
\begin{align*}
	\Lambda^{+}(v,w)\defeq\omega(\Phi(v)\Phi^{\ast}(w))\hspace*{1cm}\text{and}\hspace*{1cm}\Lambda^{-}(v,w)\defeq\omega(\Phi^{\ast}(w)\Phi(v))\,.
\end{align*}
for all $v,w\in\PS$. Indeed, any quasifree state $\omega\:\CCR\to\CC$ can be written as
\begin{align*}
	\omega\bigg(\prod_{i=1}^{n}\Phi^{\ast}(v_{i})\prod_{j=1}^{m}\Phi(w_{j})\bigg)=\delta_{mn}\sum_{\sigma\in\mathfrak{S}^{n}}\prod_{i=1}^{n}\omega(\Phi^{\ast}(v_{\sigma(i)})\Phi(w_{\sigma(i)}))
\end{align*}
for $m,n\in\NN$. In fact, every pair of Hermitian forms $\Lambda^{\pm}\:\PS\times\PS\to\CC$ on $\PS$, which are non-negative in the sense that $\Lambda^{\pm}(v,v)\geq 0$ for all $v\in\PS$, and with the property $\Lambda^{+}-\Lambda^{-}=\q_{1}$ define a unique quasi-free state in this way. For further details we refer to~\cite{Gerard2019}. Adapting \cite{Gerard2019,GerardWrochnaYM2015} to the setting of Sobolev spaces, there is the following characterization:

\begin{proposition} 
Consider  a pair of \textup{pseudo-covariances} $\lambda^\pm$, namely continuous linear maps $\lambda^{\pm}\:\Gamma_{\tc,\infty}(\V_{1})\to\Gamma(\V_{1})$ satisfying 
\begin{itemize}
\item[(i)] \quad $(\lambda^{\pm})^{\ast}=\lambda^{\pm}$ \, w.r.t.\, $(\cdot,\cdot)_{\V_{1}}$\, and \, $\lambda^{\pm}\:\ran(\P\vert_{\Gamma_{\tc}})\cap\Gamma_{\G}(\V_{1})\to\ran(\P\vert_{\Gamma_{\tc}})\cap\Gamma_{\G}(\V_{1})$;
\item[(ii)]\quad $\D_{1}\lambda^{\pm}=\lambda^{\pm}\D_{1}=0$   and  $(\lambda^{+}-\lambda^{-})s=\i\G_{1}s$ mod $\sf{K}(\mathrm{ker}(\D_{0}))\cap\Gamma_{\infty,\d}(\sf{V}_{1})$ \, $\forall s\in \ker(\K^{\ast}\vert_{\Gamma_{\tc,\infty}})$;
\item[(iii)]\quad $( s,\lambda^{\pm}s)_{\V_{1}}\geq 0$\,  for any  $ s\in\ker(\K^{\ast}\vert_{\Gamma_{\tc,\infty}})$.
\end{itemize}
Then, the Hermitian forms $\Lambda^{\pm}\:\PS\times\PS\to\CC$ defined by
\begin{align*}
	\Lambda^{\pm}([s],[t])\defeq( s,\lambda^{\pm}t)_{\V_{1}} \qquad \forall s,t\in\ker(\K^{\ast}\vert_{\Gamma_{\tc,\infty}})
\end{align*}
are the spacetime covariances of a quasifree state on $\CCR$.
 \end{proposition}
 The name ``pseudo-''covariance comes from the fact that
$\lambda^\pm$ are not required to be positive for $(\cdot\,,\cdot)_{\V_1}$ on $\Gamma_\infty (\V_1)$, but only on the subspace
$\ker(\K^{\ast}\vert_{\Gamma_{\tc,\infty}})$. \bigskip
 
It is widely accepted that among all possible (quasifree) states, the physical ones are required to satisfy the so-called Hadamard condition. The reasons for this choice are manifold: For example, it implies the finiteness of the quantum fluctuations of the expectation value of every observable, and it allows us to construct Wick polynomials following a covariant scheme. This requirement is conveniently translated to the language of microlocal analysis:

\begin{definition}\label{Def:Hadamard}
	 A quasifree state $\omega$ on $\CCR$ induced by a pair of pseudo-covariances $\lambda^\pm$ is called \textit{Hadamard state} if
	\begin{align*}
		\WF^{\prime}(\lambda^{\pm})\subset\mathcal{N}^{\pm}\times\mathcal{N}^{\pm}\, ,
	\end{align*}
	where $\WF^{\prime}(\lambda^{\pm})$ is defined as the primed wavefront set of the Schwartz kernel of $\lambda^\pm$ and 
$$\mathcal{N}^{\pm}\defeq \{(p,\xi)\in\T^{\ast}\M\mid  g^{-1}(\xi,\xi)=0 \text{ and } \pm \xi(v)>0\,\,\forall v\in\T_{p}\M\text{ future-directed timelike}\}\, .$$
\end{definition}
This definition of Hadamard states is the same as in \cite{GerardWrochnaYM2015,GerardMurroWrochna2022}. See also \cite{WrochnaPhD,Gerard2019} for reviews of other (equivalent) definitions. The construction of Hadamard states is in general a hard task and to be accomplished, it is often preferable  to work on the level of initial data. 

\begin{definition}\label{Def:PseudoCovariances}
Let $\omega$ be a quasifree state on $\CCR$ induced by pseudo-covariances $\lambda^{\pm}$. The continuous linear maps $\lambda^\pm_\Sigma\:\Gamma_{\infty}(\V_{\rho_{1}})\to\Gamma(\V_{\rho_{1}})$ defined by
\begin{align*}
	\lambda^{\pm}_{\Sigma}\defeq(\rho_{1}^{\ast}\G_{1,\Sigma})^{\ast}\lambda^{\pm}(\rho^{\ast}_{1}G_{1,\Sigma})\,,
\end{align*}
where the adjoints are taken w.r.t.~$(\cdot,\cdot)_{\V_{i}}$ and $(\cdot,\cdot)_{\V_{\rho_{i}}}$, are called \textit{Cauchy pseudo-covariances}.
\end{definition}
Since the bundle metrics $\langle\cdot,\cdot\rangle_{\V_{\rho_{i}}}$ can be chosen arbitrarily, it is more natural to work with the physical Hermitian form $\q_{1,\Sigma}$. Following \cite{Gerard2019,GerardWrochnaYM2015} we get the following characterization:
 
\begin{proposition}\label{prop:hadafromc}
Suppose $c^\pm:\Gamma_{\infty}(\V_{1})\to\Gamma(\V_{1})$ are continuous linear operators satisfying
\begin{itemize}
\item[(i)] \quad $( c^\pm)^{\dagger}=c^{\pm}$\, w.r.t. \,$\q_{1,\Sigma}$\,\,  and \,  $ c^\pm(\ran(\K_{\Sigma})\cap \Gamma_{\infty,\d}(\V_{\rho_{1}}))\subset\ran(\K_{\Sigma})\cap \Gamma_{\infty,\d}(\V_{\rho_{1}})$;
\item[(ii)] \quad $(c^{+}+ c^{-})\f=\f$ \, modulo \, $\ran(\K_\Sigma)\cap \Gamma_{\infty,\d}(\V_{\rho_{1}})$ \, for any $ \f\in\ker(\K_{\Sigma}^{\dagger}\vert_{\Gamma_{\infty}})$;
\item[(iii)]\quad $\pm\q_{1,\Sigma}(\f,c^{\pm}\f)\geq 0$ \, for any $\f\in\ker(\K_{\Sigma}^{\dagger}\vert_{\Gamma_{\infty}})$.
\end{itemize}
Then, $\lambda^{\pm}_{\Sigma}\defeq\pm \i\G_{1,\Sigma}c^{\pm}$ are the Cauchy pseudo-covariances of a quasifree state on $\CCR$, or in other words, 
\begin{align*}
	\lambda^{\pm}\defeq (\rho_{1}\G_{1})^{\ast}\lambda^{\pm}_{\Sigma}(\rho_{1}\G_{1})=\pm i\U_{1} c^{\pm}(\rho_{1}\G_{1}),
\end{align*}
are pseudo-covariances of a quasifree state on $\CCR$. Furthermore, suppose that $\D_1$ is normally hyperbolic. If for some neighbourhood $U$ of $\Sigma$ in $\M$ we have
 $$ \text{(iv)} \hspace*{0.5cm} \WF'(\mathcal{U}_{1} c_{1}^{\pm})\subset (\mathcal{N}^{\pm}\cup F)\times\T^{*}\Sigma \; \text{ over $U\times \Sigma$
where $F\subset \T^{*}\M$ is a conic set s.t.~$F\cap \mathcal{N}= \emptyset$}$$ 
then the associated state  is Hadamard. 
 \end{proposition}
 
\begin{remark}
	Note that $(\rho_{1}\G_{1})$ can continuously be extended to a surjective map $(\rho_{1}\G_{1})\:\Gamma^{\prime}(\V_{1})\to\Gamma^{\prime}(\V_{\rho_{1}})$, which is needed in the definition of $\lambda^{\pm}$ above (see \cite[Lemma 3.7.]{GerardWrochnaYM2015} for details).
\end{remark}
\subsection{Construction of Hadamard states in Ultrastatic Spacetimes}\label{Sec:ConstructionHadamardFINAL}
The aim of this section is to construct Hadamard states for Maxwell fields in the Cauchy radiation gauge on ultrastatic globally hyperbolic manifolds with a Cauchy hypersurface of \textit{bounded geometry}. The general case can be achieved by performing the usual deformation argument (see also Section~\ref{Sec:Finale}).
\subsubsection{Phase Space and Gauge Fixing}\label{Subsec:PhaseSpaceUltrastatic}
As a first step, we provide an explicit realization of the abstract isomorphism $\T_{\Sigma}$ between $\PSS$ and $\PSR$ introduced in Proposition~\ref{Prop:EquPhaseSpaces} on ultrastatic spacetimes. This will be fundamental for discussing the Hadamard properties of the pseudodifferential operators $c^\pm$ we are going to construct. Throughout this section, let $(\M,g)$ be globally hyperbolic and ultrastatic manifold 
\begin{align*}
	\M\cong\mathbb{R}\times\Sigma\,,
	\qquad
	 g=-\d t^2+ h\,,
\end{align*}
where $\Sigma$ denotes a smooth spacelike Cauchy hypersurface and $h$ a Riemannian metric on $\Sigma$. Recall that this implies in particular that $(\Sigma,h)$ is complete, see e.g.~\cite{Kay1978,Sanchez2004}. Therefore, the assumptions of the Setup~\ref{setup} are satisfied. \bigskip

In order to construct the operator $\T_\Sigma$, we need first an explicit expression of $\K_\Sigma$ and $\K_\Sigma^\dagger$. 

\begin{lemma}\label{LemmaKSigma}
Let $(\M,g=-\d t^{2}+h)$ be an ultrastatic spacetime. Using the notation of Section~\ref{Sec:PhaseSpace} and the decompositions of $\Gamma(\V_{\rho_{k}})$ as in Equation \eqref{NotationInitialData}, the operators $\K_{\Sigma}$ and $\K_{\Sigma}^{\dagger}$ take the form
	\begin{align*}
		\K_{\Sigma}=\begin{pmatrix}
            0 & \i\Id\\\i\DeltaZS   &0 \\ \d_{\Sigma} & 0 \\ 0 &\d_{\Sigma}
        \end{pmatrix}\hspace*{1cm}&\text{and}\hspace*{1cm}\K_{\Sigma}^{\dagger}=\begin{pmatrix}
            0 & \i\Id & \delta_{\Sigma} & 0\\ \i\DeltaZS  &0&0&\delta_{\Sigma}
        \end{pmatrix}\, .
	\end{align*}
\end{lemma}
    
\begin{proof}
Let $(a,\pi)\in\Gamma(\V_{0})$. Then, $\K_{\Sigma}=\rho_{1}\K\U_{0}$ acts on $(a,\pi)$ as 
\begin{align}\label{eq:KSigma}
	\K_{\Sigma}\begin{pmatrix}a\\\pi\end{pmatrix}=\begin{pmatrix}\partial_{t}f\vert_{\Sigma}\\ \i^{-1}\partial_{t}^{2}f\vert_{\Sigma}\\(\d f)_{\Sigma}\vert_{\Sigma}\\\i^{-1}\partial_{t}(\d f)_{\Sigma}\vert_{\Sigma}\end{pmatrix}=\begin{pmatrix}\i\pi\\ \i\DeltaZS a\\\d_{\Sigma}a\\\d_{\Sigma}\pi\end{pmatrix}\hspace*{0.5cm}\text{with}\hspace*{0.5cm}f\defeq \U_{0}\begin{pmatrix}a\\\pi\end{pmatrix}\in\ker(\D_{0})\, ,
\end{align}
where we used that $\D_{0}f=\partial_{t}^{2}f+\DeltaZS f=0$. Similarly, for $(a_{0},\pi_{0},a_{\Sigma},\pi_{\Sigma})\in\Gamma(\V_{1})$, we calculate
    \begin{align*}
    		\K_{\Sigma}^{\dagger}\begin{pmatrix}a_{0}\\\pi_{0}\\ a_{\Sigma}\\\pi_{\Sigma}\end{pmatrix}=\begin{pmatrix}\delta A\vert_{\Sigma}\\\i^{-1}\partial_{t}\delta A\vert_{\Sigma}\end{pmatrix}=\begin{pmatrix}\delta_{\Sigma}a_{\Sigma}+\i \pi_{0}\\i\Delta_{0}a_{0}+\delta_{\Sigma}\pi_{\Sigma}\end{pmatrix}\hspace*{0.5cm}\text{with}\hspace*{0.5cm}A\defeq \U_{1}\begin{pmatrix}a_{0}\\\pi_{0}\\ a_{\Sigma}\\\pi_{\Sigma}\end{pmatrix}\in\ker(\D_{1})\, ,
    \end{align*}
    where we used $\delta A=\delta_{\Sigma}A_{\Sigma}+\partial_{t}A_{0}$ as well as $(\square_{1}A)_{0}=\partial_{t}^{2}A_{0}+\Delta_{0}A_{0}=0$.
\end{proof} 

\begin{corollary}\label{Cor:VSR}
The space of initial data in the Cauchy radiation gauge $\PSR$ (see Proposition~\ref{Prop:EquPhaseSpaces}) on an ultrastatic globally hyperbolic manifold $(\M,g=-\d t^2+h)$ is given by
\begin{align*}
		\PSR=\{(0,0,a_{\Sigma},\pi_{\Sigma})\in\Gamma_{\infty}(\V_{\rho_{1}})\mid \delta_{\Sigma}a_{\Sigma}=\delta_{\Sigma}\pi_{\Sigma}=0\}.
\end{align*}
\end{corollary}

\begin{lemma}\label{Lemma:ProjectionsHodgeDecomp}
	Let $(\Sigma,h)$ be a complete Riemannian manifold. The orthogonal projector $$\Pi\:\Omega^{1}_{\infty}(\Sigma)\to\ker(\delta_{\Sigma}\vert_{\Omega^{1}_{\infty}}),$$ where $\Omega^{1}_{\infty}(\Sigma)\cong\Omega^{1}_{\infty,\d}(\Sigma)\oplus\ker(\delta_{\Sigma}\vert_{\Omega^{1}_{\infty}})$ by Theorem~\ref{Thm:HodgeDecomSmooth}, has the following properties:
\begin{itemize}
	\item[(i)]$\Pi^{2}=\Pi$, $\mathrm{ran}(\Pi)=\ker(\delta_{\Sigma}\vert_{\Omega^{1}_{\infty}})$ and $\ker(\Pi)=\Omega^{1}_{\infty,\d}(\Sigma)$.
	\item[(ii)]$\Pi$ is orthogonal w.r.t.~$\langle\cdot,\cdot\rangle_{\L^{2}}$.
	\item[(iii)]$\Pi\Delta_{1}=\Delta_{1}\Pi$ on $\Omega^{1}_{\infty}(\Sigma)$.
	\item[(iv)]$\Pi=\mathds{1}-\d_{\Sigma}\Delta_{1}^{-1}\delta_{\Sigma}$ where $\Delta_{1}$ denotes the Laplacian as an operator
	\begin{align*}
		\Delta_{1}\:\{f\in C^{\infty}(\Sigma;\CC)\mid\d_{\Sigma}f\in\Omega^{1}_{\infty,\d}(\Sigma)\}/_{\sim_{\c}}\to\mathrm{ran}(\delta_{\Sigma}\vert_{\Omega^{1}_{\infty}})\, ,
	\end{align*}
	where $\sim_{\c}$ denotes the equivalence relation on $C^{\infty}(\Sigma;\CC)$ identifying functions that differ by a constant. Recall that $\Delta_{1}$ with this domain is bijective by Proposition~\ref{prop:Poisson}.
\end{itemize}	 
\end{lemma}

\begin{proof}
	(i) is clear and follow from the fact that $\Pi$ is an orthogonal projector. For (ii) we note that $\Omega^{1}_{\infty,\d}(\Sigma)$ is orthogonal to $\ker(\delta_{\Sigma}\vert_{\Omega^{1}_{\infty}})$ also w.r.t.~$\langle\cdot,\cdot\rangle_{\L^{2}}$: Let $\alpha\in\Omega^{1}_{\infty,\d}(\Sigma)$ and $\beta\in\ker(\delta_{\Sigma}\vert_{\Omega^{1}_{\infty}})$. By assumption, $\alpha=\H^{\infty}\text{-}\lim_{n\to\infty}\d_{\Sigma} f_{n}$ for some sequence $(f_{n})_{n}$ in $\Omega^{0}_{\infty}(\Sigma)$. By Lemma~\ref{Lemma:WeakCovergence}, the same limit holds in the $\L^{2}$-sense and hence
	\begin{align*}
		0=\langle f_{n},\delta_{\Sigma}\beta\rangle_{\L^{2}}=\langle\d_{\Sigma} f_{n},\beta\rangle_{\L^{2}}\xrightarrow{n\to\infty}\langle\alpha,\beta\rangle_{\L^{2}}
	\end{align*}
	which shows that $\langle\alpha,\beta\rangle_{\L^{2}}=0$.
	
	For (iii), we take $\omega\in\Omega^{1}_{\infty}(\Sigma)$ and decompose it uniquely as $\omega=\alpha+\beta$ with $\alpha\in\Omega^{1}_{\infty,\d}(\Sigma)$ and $\beta\in\mathrm{ker}(\delta_{\Sigma}\vert_{\Omega^{1}_{\infty}})$ so that $\Pi\omega=\beta$. It follows that $\Delta_{1}\Pi\omega=\Delta_{1}\beta$. On the other hand, $\Delta_{1}\omega=\Delta_{1}\alpha+\Delta_{1}\beta$. Clearly, $\Delta_{1}\beta\in\ker(\delta_{\Sigma}\vert_{\Omega^{1}_{\infty}})$. It remains to show that $\Delta_{1}\alpha\in\Omega^{1}_{\infty,\d}(\Sigma)$, since this implies $\Pi\Delta_{1}\omega=\Delta_{1}\beta$ and hence $\Delta_{1}\Pi\omega=\Pi\Delta_{1}\omega$. By assumption, $\alpha=\H^{\infty}\text{-}\lim_{n\to\infty}\d\varphi_{n}$ for a sequence $(\varphi_{n})_{n}$ in $\Omega^{0}_{\infty}(\Sigma)$. Clearly $\Delta_{1}\d\varphi_{n}=\d\Delta_{0}\varphi_{n}$ and $\Delta_{0}\varphi_{n}\in\Omega^{0}_{\infty}(\Sigma)$. It follows that $\Delta_{1}\alpha=\H^{\infty}\text{-}\lim_{n\to\infty}\d\Delta_{0}\varphi_{n}$, since $\Delta_{1}$ is bounded as an operator $\Delta_{1}\:\Omega^{1}_{s+2}(\Sigma)\to\Omega^{1}_{s}(\Sigma)$ for any $s\geq 0$, cf.~Lemma~\ref{Def:OmegaS}. We conclude that $\Delta_{1}\alpha\in\Omega^{1}_{\infty,\d}(\Sigma)$, as claimed.
	
	For (iv), we take $\omega\in\Omega^{1}_{\infty}(\Sigma)$ and write $\omega=\alpha+\beta$ for $\alpha\in\Omega^{1}_{\infty,\d}(\Sigma)$ and $\beta\in\ker(\delta_{\Sigma}\vert_{\Omega^{1}_{\infty}})$. Furthermore, we write $\alpha=\d_{\Sigma}f$ for some $f\in C^{\infty}(\Sigma;\CC)$, which is possible by Theorem~\ref{Thm:HodgeDecomSmooth}(i). Then $\Pi\omega=\beta$ and 
	\begin{align*}
		(\mathds{1}-\d_{\Sigma}\Delta_{1}^{-1}\delta_{\Sigma})\omega=\omega-\d_{\Sigma}\Delta_{1}^{-1}\delta_{\Sigma}\omega=\omega-\d_{\Sigma}f=\omega-\alpha=\beta.
	\end{align*}
	where we used that $f$ is the unique smooth solution $\Delta_{0}f=\delta_{\Sigma}\omega$ satisfying $\d_{\Sigma}f\in\Omega^{1}_{\infty,\d}(\Sigma)$, see Proposition~\ref{prop:Poisson}.
\end{proof}

\begin{proposition}\label{Prop:ProjOp}
Let $(\M,g)$ be a ultrastatic globally hyperbolic manifold such that $(\Sigma,h)$ is of bounded geometry. Then , the operator $\T_{\Sigma}\:\Gamma_{\infty}(\V_{\rho_{1}})\to \Gamma_{\infty}(\V_{\rho_{1}})$ given by
\begin{align*}
		\T_{\Sigma}=\begin{pmatrix}
	0&0&0&0\\0&0&0&0\\0 & 0 & \Pi& 0\\ 0&0 &0&\Pi
	\end{pmatrix}\,.
	\end{align*}
	has the following properties:
	\begin{itemize}
	\item[(i)] $\T_{\Sigma}^{2}=\T_{\Sigma}$ \, and\,  $\T_{\Sigma}\vert_{\PSR}=\Id$
	\item[(ii)] $\T_{\Sigma}=\Id-\K_{\Sigma}(\R_{\Sigma}\K_{\Sigma})^{-1}\R_{\Sigma}$ \, on\, $\ker(\K_{\Sigma}^{\dagger}\vert_{\Gamma_{\infty}})$ and it has the following properties:
	\begin{itemize}
	\item[(iia)] $\ker(\T_{\Sigma}\vert_{\ker(\K_{\Sigma}^{\dagger}\vert_{\Gamma_{\infty}})})=\ran(\K_{\Sigma})\cap\Gamma_{\infty,\d}(\V_{\rho_{1}})$ 
	\item[(iib)] $\ran(\T_{\Sigma}\vert_{\ker(\K_{\Sigma}^{\dagger}\vert_{\Gamma_{\infty}})})=\PSR$
	\end{itemize}
	\end{itemize}
	In particular, (ii) implies that $\T_{\Sigma}$ is well-defined and bijective as a map $\T_{\Sigma}:\PSS\to\PSR$.
\end{proposition}

\begin{proof}
Claim (i) follows directly from Lemma~\ref{Lemma:ProjectionsHodgeDecomp} and the characterization of $\PSR$ in Corollary~\ref{Cor:VSR}. Let us now turn to (ii): By Lemma~\ref{LemmaKSigma}, the operator $\R_{\Sigma}\K_{\Sigma}\:\Gamma(\V_{\rho_{0}})\to\Gamma(\V_{\rho_{1}})$ is given by 
\begin{align}\label{eq:ProofRK}
        \R_{\Sigma}\K_{\Sigma}=\begin{pmatrix}0& \i\Id\\ \i\DeltaZS & 0\\ 0&0\\0&0\end{pmatrix}
\end{align}
Now, following Proposition~\ref{prop:Poisson}, this operator is well-defined and bijective as an operator
\begin{align*}
	\R_{\Sigma}\K_{\Sigma}\:\{(a,\pi)\in\Gamma(\V_{\rho_{0}})\mid \d_{\Sigma}a\in\Omega^{1}_{\infty,\d}(\Sigma)\}/_{\sim_{\c}}\to\{(a,\pi,0,0)\in\Gamma(\V_{\rho_{1}})\mid \pi\in\ran(\delta_{\Sigma}\vert_{\Omega^{1}_{\infty}})\}\,,
\end{align*}
where $\sim_{\c}$ denotes the equivalence relation on $\Gamma(\V_{0})$ identifying $(a_{0},\pi_{0})$ and $(a_{1},\pi_{1})$ if and only if $a_{0}$ and $a_{1}$ differ by a constant. It follows that $(\R_{\Sigma}\K_{\Sigma})^{-1}\R_{\Sigma}$ is well-defined on $\ker(\K_{\Sigma}^{\dagger}\vert_{\Gamma_{\infty}})$. Furthermore, it is straight-forward to verify that 
\begin{align*}
		\K_{\Sigma}(\R_{\Sigma}\K_{\Sigma})^{-1}\R_{\Sigma}\begin{pmatrix}a_{0}\\\pi_{0}\\a_{\Sigma}\\\pi_{\Sigma}\end{pmatrix}=\begin{pmatrix}
		a_{0}\\\pi_{0}\\-\d_{\Sigma}\DeltaZS ^{-1}(\i\pi_{0})\\-\i\d_{\Sigma}a_{0}
		\end{pmatrix}=\begin{pmatrix}
		a_{0}\\\pi_{0}\\\d_{\Sigma}\DeltaZS ^{-1}(\delta_{\Sigma}a_{\Sigma})\\\d_{\Sigma}\DeltaZS ^{-1}(\delta_{\Sigma}\pi_{\Sigma})
		\end{pmatrix}\, ,
	\end{align*}
for all $(a_{0},\pi_{0},a_{\Sigma},\pi_{\Sigma})\in\ker(\K_{\Sigma}\vert_{\Gamma_{\infty}})$, where $\Delta_{0}^{-1}$ is as defined in Lemma~\ref{Lemma:ProjectionsHodgeDecomp}(iii) and where we used Lemma~\ref{LemmaKSigma} to conclude that $\i\pi_{0}=-\delta_{\Sigma}a_{\Sigma}$ as well as $\i\DeltaZS a_{0}=-\delta_{\Sigma}\pi_{\Sigma}$, which implies $\Delta_{0}^{-1}\delta_{\Sigma}\pi_{\Sigma}=-ia_{0}$ where $a_{0}$ is such that $\d_{\Sigma}a_{0}\in\Omega^{1}_{\infty,\d}(\Sigma)$. We conclude that $\T_{\Sigma}=\Id-\K_{\Sigma}(\R_{\Sigma}\K_{\Sigma})^{-1}\R_{\Sigma}$ on $\ker(\K_{\Sigma}^{\dagger}\vert_{\Gamma_{\infty}})$ by Lemma~\ref{LemmaKSigma}. It remains to check (iia) and (iib). For $\mathrm{ran}(\T_{\Sigma}\vert_{\ker(\K_{\Sigma}^{\dagger}\vert_{\Gamma_{\infty}})})=\PSR$, the direction ``$\subset$'' is clear since $\T_{\Sigma}\:\Gamma_{\infty}(\V_{\rho_{1}})\to\PSR$. The other direction is clear since clearly $\T_{\Sigma}=\mathds{1}$ on $\PSR$, which follows from $\Pi=\mathds{1}$ on $\ker(\delta_{\Sigma}\vert_{\Omega^{1}_{\infty}})$. Similarly, $\ker(\T_{\Sigma}\vert_{\ker(\K_{\Sigma}^{\dagger}\vert_{\Gamma_{\infty}})})=\ran(\K_{\Sigma})\cap\Gamma_{\infty,\d}(\V_{\rho_{1}})$ follows from $\ker(\Pi)=\Omega^{1}_{\infty,\d}(\Sigma)$.
\end{proof}

We conclude this section, by showing that the projectors $\T_{\Sigma}$ induces a Hermitian form $\q_{\Sigma,\R}$ on the space $\PSR$ such that
\begin{align*}
	\T_{\Sigma}\:(\PSS,\q_{1,\Sigma})\to(\PSR,\q_{\Sigma,\R})
\end{align*}
is unitary, i.e.~$\q_{\Sigma,\R}(\T_{\Sigma}[\cdot],\T_{\Sigma}[\cdot])=\q_{1,\Sigma}([\cdot],[\cdot])=\i(\cdot,\G_{1,\Sigma}\cdot)_{\V_{\rho_{1}}}$.
Before entering into the details, let us give an explicit characterization of $\G_{i,\Sigma}$. To this end, let us endow $\Gamma(\V_{\rho_i})$ with the Hermitian forms defined by
    \begin{align*}
        \bigg(\begin{pmatrix}
            a\\\pi
        \end{pmatrix},\begin{pmatrix}
            b\\\eta
        \end{pmatrix}\bigg)_{\V_{\rho_{0}}}&\defeq\int_{\Sigma}\big(\overline{a}b+\overline{\pi}\eta\big)\,\volS,\\\bigg(\begin{pmatrix}
            a_{0}\\\pi_{0}\\a_{\Sigma}\\\pi_{\Sigma}
        \end{pmatrix},\begin{pmatrix}
            b_{0}\\\eta_{0}\\b_{\Sigma}\\\eta_{\Sigma}
        \end{pmatrix}\bigg)_{\V_{\rho_{1}}}&\defeq\int_{\Sigma}\big(\overline{a_{0}}b_{0}+\overline{\pi_{0}}\eta_{0}+h^{\sharp}(\overline{a_{\Sigma}},b_{\Sigma})+h^{\sharp}(\overline{\pi_{\Sigma}},\eta_{\Sigma})\big)\,\volS\, .
    \end{align*}
Note that $\T_{\Sigma}\:\Gamma_{\infty}(\V_{\rho_{1}})\to\Gamma_{\infty}(\V_{\rho_{1}})$ is formally self-adjoint w.r.t.~$(\cdot,\cdot)_{\V_{\rho_{1}}}$ as can be seen from its definition and from Lemma~\ref{Lemma:ProjectionsHodgeDecomp}(ii). Using the definition of $\G_{i,\Sigma}$ in Equation~\eqref{GSigma} as well as Green's identity (see e.g.~\cite{Gerard2019}), a straight-forward computation shows that they can be written in the following matrix form:\footnote{While the operators $\G_{i,\Sigma}$ are completely specified by the data of a linear gauge theory, recall that the explicit matrix representation depends on the choice of $\langle\cdot,\cdot\rangle_{\V_{\rho_{i}}}$, whose choice is arbitrary. }

\begin{align*}
		\G_{0,\Sigma}=\frac{1}{\i}\begin{pmatrix}
			0 & \Id\\ \Id & 0
		\end{pmatrix}\hspace*{1cm}&\text{and}\hspace*{1cm}\G_{1,\Sigma}=\frac{1}{\i}\begin{pmatrix}
			0 & -\Id & 0 &0\\-\Id &0&0&0\\ 0&0&0&\Id\\0&0&\Id &0
		\end{pmatrix}
	\end{align*}
	
\begin{proposition}
	The Hermitian form $\q_{\Sigma,\R}$ on $\PSR$ is given by
	\begin{align*}
		\q_{\Sigma,\R}(\f,\g)=\i(\f,\G_{\Sigma,\R}\g)_{\V_{\rho_{1}}}
	\end{align*}
	for all $\f,\g\in\PSR$, where the linear operator $\G_{\Sigma,\R}\:\Gamma_\infty(\V_{\rho_{1}})\to\Gamma_\infty(\V_{\rho_{1}})$ is given by
	\begin{align*}
		\G_{\Sigma,\R}\defeq\G_{1,\Sigma}\vert_{0\,\oplus\,(\Omega^{1}(\Sigma;\CC))^{2}}=\frac{1}{\i}\begin{pmatrix}0&0&0&0\\0&0&0&0\\0&0&0&\Id\\0&0&\Id&0\end{pmatrix}\, .
	\end{align*}
\end{proposition}

\begin{proof}
Let $\f,\g\in\PSR$. By definition, in particular $[\f],[\g]\in\PSS$ as well as $T_{\Sigma}[\f]=\f$ and $T_{\Sigma}[\g]=\g$. The claim then follows from the definition of $\q_{\Sigma,\R}$, i.e.~$\q_{\Sigma,\R}(\f,\g)=\q_{\Sigma,\R}(T_{\Sigma}[\f],T_{\Sigma}[\g])=\q_{1,\Sigma}([\f],[\g])=\i\langle\f,\G_{1,\Sigma}\g\rangle_{\V_{\rho_{1}}}=\i\langle\f,\G_{\Sigma,\R}\g\rangle_{\V_{\rho_{1}}}$.
\end{proof}
\subsubsection{Pseudodifferential Calculus on Manifolds of Bounded Geometry}
In this section, we briefly recall Shubin's calculus of pseudodifferential operators on manifolds of bounded geometry \cite{Kordyukov1991,Shubin1992}, to fix the notation and terminology. We follow the presentations in \cite{GerardStoskopf2021,GerardMurroWrochna2022}. For more details, see also \cite{GerardWrochnaOulghazi2017,Gerard2019}.\bigskip

Let $(\Sigma,h)$ be a Riemannian manifold of dimension $d\in\mathbb{N}$. A rank $(q,p)$-tensor field $T\in\Gamma(\T^{\ast}\Sigma^{\otimes q}\otimes \T\Sigma^{\otimes p})$ is called \emph{bounded}, if $\Vert T\Vert\defeq\sup_{x\in\Sigma}\Vert T_{x}\Vert_{x}<\infty$, where $\Vert\cdot\Vert_{x}$ denotes the canonical norm on $\T^{\ast}_{x}\Sigma^{\otimes q}\otimes \T_{x}\Sigma^{\otimes p}$ defined by the metric $h_{x}$ and its inverse. The Riemannian manifold $(\Sigma,h)$ is said to be of \emph{bounded geometry} \cite{Gromov1985}, if its injectivity radius is non-zero and if the Riemann curvature tensor and all its covariant derivatives are bounded. Note that bounded geometry is a purely geometrical concept, since on every manifold there exists a Riemannian metric of bounded geometry, see e.g.~\cite{Greene1978}. Let $U\subset\mathbb{R}^{d}$ be open and $\delta$ denote the Euclidean metric. Then, $\mathrm{BT}_{q}^{p}(U,\delta)$ is the space of all rank $(q,p)$-tensor fields on $U$, which are bounded with all their derivatives, equipped with its natural Fréchet space topology. For the case $p=q=0$ we also write $C^{\infty}_{\b}(U)$. A Riemannian manifold $(\Sigma,h)$ is of bounded geometry if and only if there exists a \emph{family of bounded charts}, i.e.~around every $x\in\Sigma$ there exists an open chart $(U_{x},\varphi_{x}\:U_{x}\to B_{1}(0))$, where $B_{1}(0)\subset\mathbb{R}^{d}$ denotes the ball around $0$ with radius $1$, such that the following holds:
\begin{align*}
	\text{(i)}&\hspace*{1cm}\text{The family }(h_{x}\defeq(\varphi_{x})_{\ast}h)_{x\in\Sigma}\text{ is bounded in }\mathrm{BT}_{2}^{0}(B_{1}(0),\delta)\\
	\text{(ii)}&\hspace*{1cm}\exists C>0\:\hspace*{0.5cm}C^{-1}\delta\leq h_{x}\leq C\delta\hspace*{0.5cm}\forall x\in\Sigma
\end{align*}
If $(U_{n},\varphi_{n})_{n\in\mathbb{N}}$ is a countable subcover of a bounded family of charts $(U_{x},\varphi_{x})_{x\in\Sigma}$, we call it \emph{bounded atlas}, if it is in addition \textit{uniformly finite}, i.e.~there exists $N\in\mathbb{N}$ such that $\bigcap_{j\in J}U_{j}=\emptyset$ for any index set $J\subset\mathbb{N}$ with $\vert J\vert>N$. A partition of unity $(\chi_{n})_{n\in\mathbb{N}}$ subordinate to a bounded atlas $(U_{n},\varphi_{n})_{n\in\mathbb{N}}$ is called \emph{bounded partition of unity}, if in addition $((\varphi_{n})_{\ast}\chi_{n})_{n\in\mathbb{N}}$ is a bounded family in $C^{\infty}_{\b}(B_{1}(0))$.\bigskip

Let now $(\Sigma,h)$ be an $d$-dimensional Riemannian manifold of bounded geometry and $\E\xrightarrow{\pi}\Sigma$ be a \emph{vector bundle of bounded geometry} of rank $k\in\mathbb{N}$, i.e.~there exists a bundle atlas of $\E$, which is also a bounded atlas of $(\Sigma,h)$, such that the transition functions form a bounded family of matrix-valued functions. Let us fix a bounded atlas $(U_{n},\varphi_{n})_{n\in\mathbb{N}}$ with corresponding bounded local trivializations $\psi_{n}\:\pi^{-1}(U_{n})\to U_{n}\times\mathbb{C}^{k}$ and a corresponding bounded partition of unity $(\chi_{n})_{n\in\mathbb{N}}$ subordinate to it. For any $U\subset\mathbb{R}^{d}$, let $\mathcal{S}_{\mathrm{cl}}^{m}(\T^{\ast}U,\mathbb{C}^{k\times k})\subset C^{\infty}(U\times\mathbb{R}^{d},\mathbb{C}^{k\times k})$ denote the set of matrix-valued classical (uniform) symbols of order $m$, see e.g.~\cite{Shubin2001,HormanderBookIII}.

\begin{definition}
	Let $\BS^{m}(\T^{\ast}\Sigma,L(\E))$ denote the set of $a\in C^{\infty}(\T^{\ast}\Sigma,\mathrm{End}(\E))$, such that for every $p\in\Sigma$, it holds that $(\varphi_{p})_{\ast}a\in\mathcal{S}_{\mathrm{cl}}^{m}(\T^{\ast}B_{1}(0),\mathbb{C}^{k\times k})$, and such that the family $((\varphi_{p})_{\ast}a)_{p\in\Sigma}$ is bounded in $\mathcal{S}_{\mathrm{cl}}^{m}(\T^{\ast}B_{1}(0),\mathbb{C}^{k\times k})$.
\end{definition}

An element $a\in\BS^{m}(\T^{\ast}\Sigma,L(\E))$ is called \emph{(classical) bounded symbol of order $m$}. Let us denote by $E\:\mathcal{S}_{\mathrm{cl}}^{m}(\T^{\ast}B_{1}(0),\mathbb{C}^{k\times k})\to\mathcal{S}_{\mathrm{cl}}^{m}(\T^{\ast}\mathbb{R}^{n},\mathbb{C}^{k\times k})$ a continuous extension. Furthermore, let us denote by $T_{i}$ and $\widetilde{T}_{i}$ the push-forwards of $C^{\infty}(U_{n},\E)$ and $C^{\infty}(\T^{\ast}U_{n},L(\E))$, respectively, under $\varphi_{n}$ and the corresponding local trivialization $\psi_{n}$ of $\E$. If $a\in \BS^{m}(\T^{\ast}\Sigma,L(\E))$, then we define its \emph{quantization} by
\begin{align*}
	\mathrm{Op}(a)\defeq\sum_{n\in\mathbb{N}}(\chi_{n}T_{n}^{-1})\circ\mathrm{Op}(E\widetilde{T}_{n}a)\circ (\chi_{n}T_{n})\:\Gamma_{\c}(\Sigma,\E)\to\Gamma(\Sigma,\E)\, ,
\end{align*}
where $\mathrm{Op}(E\widetilde{T}_{n}a)$ denotes the usual Kohn-Nirenberg quantization of $E\widetilde{T}_{n}a\in\mathcal{S}_{\mathrm{cl}}^{m}(\T^{\ast}\mathbb{R}^{d},\mathbb{C}^{k\times k})$, see e.g.~\cite{Shubin2001,HormanderBookIII}. In general, the quantization map depends on the various choices of $E,\chi_{n},\varphi_{n},\psi_{n}$, however, one can show that for any other quantization map $\mathrm{Op}^{\prime}$ one has that
\begin{align*}
	\mathrm{Op}(a)-\mathrm{Op}^{\prime}(a)\in \mathcal{W}^{-\infty}(\Sigma,\E)\, ,
\end{align*}
where $\mathcal{W}^{-\infty}(\Sigma,\E)$ denotes an ideal of smoothing operators defined by
\begin{align*}
	\mathcal{W}^{-\infty}(\Sigma,E)\defeq\bigcap_{n\in\mathbb{N}}B(H^{-m}(\Sigma,\E),H^{m}(\Sigma,\E))\subset\Psi^{-\infty}(\Sigma,\E)\,,
\end{align*}
where $H^{m}(\Sigma,\E)$ denotes the Sobolev spaces of order $m\in\mathbb{N}$ and $B(\cdot,\cdot)$ the space of bounded linear operators between them. This gives rise to the following definition:

\begin{definition}
	Let $(\Sigma,h)$ be a Riemmanian manifold of bounded geometry and $\mathrm{Op}$ a quantization map as above. The space of \emph{bounded pseudodifferential operators} of order $m$ is defined by
	\begin{align*}
		\Psi^{m}_{\b}(\Sigma,\E)\defeq\mathrm{Op}(\BS^{m}(\T^{\ast}\Sigma,L(\E)))+\mathcal{W}^{-\infty}(\Sigma,\E).
	\end{align*}
	Furthermore, we set $\Psi^{\infty}_{\b}(\Sigma,\E)\defeq\bigcup_{m\in\mathbb{N}}\Psi^{m}_{\b}(\Sigma,\E)$.
\end{definition}
Note that every properly-supported classical pseudodifferential operator on $\Sigma$ is also bounded. Furthermore, the space of bounded pseudodifferential operators is closed under compositions and taking adjoints.
\subsubsection{Microlocal Factorization and Hadamard States on Ultrastatic Manifolds}
Throughout this section, let $(\M,g)$ be a globally hyperbolic and ultrastatic manifold with a Cauchy surface of bounded geometry, namely
\begin{align*}
	\M=\mathbb{R}\times\Sigma\,,
	\qquad
	 g=-\d t^2+ h\,,
\end{align*}
where $\Sigma$ denotes a smooth spacelike Cauchy hypersurface and $h$ a Riemannian metric on $\Sigma$, such that $(\Sigma,h)$ is of bounded geometry. On this class of manifolds, the normally hyperbolic operators $\D_i$ admits a microlocal factorization. 
To this end, let us introduce the following notation
$$\E_i\defeq\begin{cases}
 \underline{\CC}\otimes\T^*\Sigma & i=1 \,, \\
\underline{\CC} & i=0\,,
\end{cases} $$
where $\underline{\CC}\defeq\Sigma\times\CC$ denotes the trivial line bundle.

\begin{lemma}\label{lem:commuting}
	Let $(\Sigma,h)$ be a Riemannian manifold of bounded geometry, $\Pi$ be the projector defined in Lemma~\ref{Lemma:ProjectionsHodgeDecomp} and let $\DeltakS$ be the Hodge-Laplacian acting on $k$-forms, for $k=0,1$.  There exists $\varepsilon_{k}\in C_{\b}^{\infty}(\Sigma,\Psi^{1}_{\b}(\Sigma, \E_i))$ and $r_{k,-\infty}\in\Psi^{-\infty}(\Sigma,\E_i))$ such that the following is fulfilled:
	\begin{align*}
		\text{(i)}&\hspace*{1cm}\varepsilon_{k}^{\ast}=\varepsilon_{k}\hspace*{1cm}\text{w.r.t. Hodge inner product}\hspace*{1cm}\text{and}\hspace*{1cm}\varepsilon_{k}^{2}=\Delta_{k}+r_{k,-\infty}\\
		\text{(ii)}&\hspace*{1cm}\sigma_{\varepsilon_{k}}(\xi)=\sqrt{h^{-1}(\xi,\xi)}\Id\\
		\text{(iii)}&\hspace*{1cm}\varepsilon_{1} \Pi=\Pi\varepsilon_{1}  \hspace*{0.5cm}\text{ up to }\mathcal W^{-\infty}(\E_i) \,.
	\end{align*}
\end{lemma}

\begin{proof}
	By~\cite[Lemma 5.1]{GerardMurroWrochna2022}, we know that there exists $\epsilon_k$ with domain $\H^1(\Sigma)$ that are $m$-accretive and satisfy (i) and (ii).
For	 point (iii) we notice that
$$ ( \varepsilon^2_{1} - r_{1,-\infty})\Pi = \Delta_1\Pi =\Pi\Delta_1  =\Pi(\epsilon^2_1 - 	 r_{1,-\infty})\,, $$
by Lemma~\ref{Lemma:ProjectionsHodgeDecomp}(iii). Note that $r_{1,-\infty}:\H^{s}(\Sigma)\to\H^{\infty}(\Sigma)$ for any $s\geq 0$, which shows that the composition of $\Pi$ with $r_{1,-\infty}$ are well-defined. This implies that
$$\varepsilon^2_{1} \Pi  =\Pi\epsilon^2_1  +  \bar r_{1,-\infty} $$  
for $\bar r_{1,-\infty}\defeq   r_{1,-\infty}\Pi -\Pi \tilde{r}_{1,-\infty}$. Furthermore, $\bar r_{1,-\infty}$ is clearly a smoothing operator. Then, $\varepsilon_1^2\Pi= \Pi \varepsilon_1^2$ up to a smooth kernel, and, iterating this argument, the same holds true for polynomials of $\varepsilon_1^2$, and in particular, $\varepsilon_1^{-2}\Pi= \Pi \varepsilon_1^{-2}$ again up to smoothing. Using Stone-Weierstrass theorem, we can generalize to continuous function $f(\varepsilon_1^{-2})$ and the generalization to unbounded function can be obtained choosing a sequence of bounded measurable
functions $f_n$ that converge pointwise to $f$.  For more details we refer to~\cite{Valter}.
\end{proof}

  We next perform a microlocal factorization of the Cauchy evolution operator $\U_{1}$ of $\square_{1}$. On ultrastatic spacetimes, the equation $\square_{1}A=0$ for $A=A_{0}\d t+A_{\Sigma}\in\Omega^{1}(\M;\CC)$ decouples into the hyperbolic equations 
  \begin{align*}
  \square_{0}A_{0}=(\partial_{t}^{2}+\Delta_{0})A_{0}=0\quad \text{and} \quad (\partial_{t}^{2}+\Delta_{1})A_{\Sigma}=0\, .
  \end{align*}
   Hence, we consider the Cauchy evolution operator $\U_{\partial_{t}^{2}+\Delta_{k}}$ of the hyperbolic operator $\partial_{t}^{2}+\Delta_{k}$ acting on $k$-forms for $k\in\{0,1\}$. Let now $A\in C^{\infty}(\RR,\Omega^{k}_{\infty}(\Sigma_{\bullet}))$  be a solution of the Cauchy problem for $\partial_{t}^{2}+\Delta_{k}$. By setting
  $$\Psi(t)= \col{A(t)}{\i^{-1}\p_{t}A(t)}$$ 
  the Cauchy problem for $\partial_{t}^{2}+\Delta_{k}$ can be rewritten as the first order system
  \begin{equation}\label{cauchy}
  \p_{t}\Psi(t)= \i A(t)\Psi(t), \quad A(t)= \mat{0}{\Id}{\Delta_k}{0}, 
  \end{equation}
and any solution at time $t$ can we written as the action of the Cauchy evolution operator, i.e.~$\Psi(t)= \U(t,0)\Psi(0)$.  Since the operator $\partial_{t}^{2}+\Delta_{k}$ admits a microlocal factorization (cf.~Lemma~\ref{lem:commuting}), we can diagonalize (up to a smoothing operator) also the the Cauchy evolution operator $\U_{\partial_{t}^{2}+\Delta_{k}}$. To this end, define the operator $S$ and $S^{-1}$  respectively by  
  \[
  S= \i^{-1}\mat{\Id}{-\Id}{\epsilon_k}{\epsilon_k}(2 \epsilon_k )^{-1}, \quad S^{-1}= \i \mat{\epsilon_k}{\Id}{-\epsilon_k}{\Id}\,.
  \]
  The Cauchy problem~\eqref{cauchy} for $\Psi$ can be rewritten as the Cauchy problem for $\tilde\Psi(t)\defeq S^{-1}\Psi(t)$
\begin{equation}\label{cauchyapp}
\p_{t}\tilde\Psi(t)= \i B(t) \tilde\Psi(t), 
\qquad \text{ for } \,\,   B= \mat{\epsilon_k}{0}{0}{- \epsilon_k}+B_{-\infty},
\end{equation}
where $B_{-\infty}$ is a smoothing operator.
This implies in particular that the Cauchy evolution operator $\U_{\partial_{t}^{2}+\Delta_{k}}$ for the Cauchy problem~\eqref{cauchy} can be factorized as $$\U_{\partial_{t}^{2}+\Delta_{k}}(t,s)=S \U_{\mathrm{ap}}(t,s) S^{-1}\,,$$ where $\U_{\mathrm{ap}}$ is the Cauchy evolution operator of the Cauchy problem~\eqref{cauchyapp}.
Since the Hermitian form $q_{k,\Sigma}$ is not preserved, i.e.~
\begin{align*}
	S^{\ast}\begin{pmatrix}
	0& \mathds{1} \\ \mathds{1} & 0
\end{pmatrix}	 S=(2\epsilon_k)^{-1}\begin{pmatrix} \mathds{1} &0\\0 &-\mathds{1}\end{pmatrix}
\end{align*}
it is convenient to consider the operator $
	S_\epsilon=S(2\epsilon_k)^{1/2}$ instead of $S$, which again provides a microlocal factorization of $\U_{\partial_{t}^{2}+\Delta_{k}}$. 
 
Since $\epsilon_k$ are self-adjoint, we can use the spectral calculus to conclude the following.
 \begin{proposition}\label{prop:cauchy ev diag}
 Let $(\M,g)$ be a globally hyperbolic, ultrastatic manifold with a Cauchy surface $(\Sigma,h)$ of bounded geometry. Denote by $\U_{\partial_{t}^{2}+\Delta_{k}}$ the Cauchy evolution operator for $\partial_{t}^{2}+\Delta_{k}$ and let $S_\epsilon$ be the operator defined by
  $$ S_\epsilon= \i^{-1}\mat{\Id}{-\Id}{\epsilon_k}{\epsilon_k}(2 \epsilon_k )^{-\frac{1}{2}}\,.$$
 Then we have
 \[
\begin{aligned}
\U_{\partial_{t}^{2}+\Delta_{k}}&= S_\epsilon\mat{\exp(\i\epsilon_k t)}{0}{0}{\exp(-\i \epsilon_k t)}S_\epsilon^{-1}+ R_{-\infty}\,,
\end{aligned}
  \]
  where $R_{-\infty}\in\mathcal{W}^{-\infty} (\E_k \oplus \E_k)$ is a smoothing operator.
  \end{proposition}

We are finally in the position to
prove Theorem~\ref{thm:hada}, i.e.~to construct suitable pseudodifferential operator $c^{\pm}$ on $\PSS$ on a  globally hyperbolic, ultrastatic manifold $(\M,g)$ with a Cauchy surface $(\Sigma,h)$ of bounded geometry.
Following Section~\ref{Sec:PhaseSpaceQuant} (cf.~Proposition~\ref{prop:hadafromc}) the operators $c^\pm$ will give rise to a unique quasifree Hadamard state $\omega$ on $\CCR$. 

\begin{proof}[Proof of Theorem~\ref{thm:hada}]
\begin{itemize}
\item[(i)] Since $\varepsilon_{k}$ are formally self-adjoint with respect to the Hodge $\L^{2}$-inner product on $\Sigma$, it follows that
	\begin{align*}
		(\pi^{\pm})^{\dagger}=\G_{1,\Sigma}^{-1}(\pi^{\pm})^{\ast}\G_{1,\Sigma}=\pi^{\pm}\hspace*{1cm}\text{with}\hspace*{1cm}(\pi^{\pm})^{\ast}=\frac{1}{2}\begin{pmatrix}\mathds{1} & \pm\varepsilon_{0} &0&0\\\pm\varepsilon_{0}^{-1} &\mathds{1}&0&0\\ 0&0&\mathds{1} & \pm\varepsilon_{1}\\0&0&\pm\varepsilon_{1}^{-1} &\mathds{1}\end{pmatrix}\, ,
	\end{align*}
	where $(\pi^{\pm})^{\ast}$ is the adjoint with respect to $\langle\cdot,\cdot\rangle_{\V_{\rho_{0}}}$. In other words, $\pi^{\pm}$ are formally self-adjoint w.r.t.~$\sigma_{1,\Sigma}$. A direct computations shows that $\T_{\Sigma}$ is formally self-adjoint w.r.t.~$\sigma_{1,\Sigma}$, which implies that $(c^{\pm})^{\dagger}=c^{\pm}$. Next, observe that $c^{\pm}$ clearly preserves $\ran(\K_{\Sigma})\cap\Gamma_{\infty,\d}(\V_{\rho_{1}})$, since $\T_{\Sigma}\circ\K_{\Sigma}=0$. 
\item[(ii)] We note that $\pi^{+}+\pi^{-}=\Id$ on the whole space of initial data $\Gamma_\infty(\V_{\rho_{1}})$ and hence
	\begin{align*}
		(c^{+}+c^{-})\f=T_{\Sigma}^{2}\f=T_{\Sigma}\f=\f\quad\text{mod}\quad\ran(\K_{\Sigma})\cap \Gamma_{\infty,\d}(\V_{\rho_{1}})
	\end{align*}
	for all $\f\in\ker(\K_{\Sigma}^{\dagger}\vert_{\Gamma_\infty})$, where in the last step we used that for every $\f\in\ker(\K_{\Sigma}^{\dagger}\vert_{\Gamma_\infty})$ there exists a (unique) $\g\in\ran(\K_{\Sigma})\cap \Gamma_{\infty,\d}(\V_{\rho_{1}})$, such that $\T_{\Sigma}(\f+\g)=\f+\g$ and hence $\T_{\Sigma}\f=\f+\g$ using that $\T_{\Sigma}\circ\K_{\Sigma}=0$. 
	\item[(iii)] A direct computations shows that
	\begin{align*}
		\pm\q_{1,\Sigma}(\f,c^{\pm}f)=\pm\q_{1,\Sigma}(\f,\T_{\Sigma}\pi^{\pm}\T_{\Sigma}f)=\pm\q_{\Sigma,\R}(\T_{\Sigma}\f,\pi^{\pm}\T_{\Sigma}f)\geq 0
\end{align*}	 
	where we used that $\q_{1,\Sigma}(\f,\T_{\Sigma}\g)=\q_{\Sigma,\R}(\T_{\Sigma}\f,\g)$ for $\f\in\Gamma_{\infty}(\V_{\rho_{1}})$ and $\g\in\L^{2}(\V_{\rho_{1}})$. 

\item[(iv)] By Lemma~\ref{lem:commuting} (iii), $\pi^\pm$ commutes with $\T_\Sigma$ modulo a smooth kernel, so we only need to check that  $\pi^\pm$ satisfies 
$$\WF'(\mathcal{U}_{1} \pi^{\pm})\subset (\mathcal{N}^{\pm}\cup F)\times\T^{*}\Sigma \; \text{ for $F=\{k=0\}\subset \T^*\M$}\,,$$
where the integral kernel of $\mathcal{U}_{1} \pi^{\pm}$ is understood by extending $\pi^{\pm}$ to any compactly supported initial data via the Hahn-Banach continuous extension theorem for locally convex Fr\'echet spaces.
To this end, we begin by observing that $\pi^\pm$ can be actually written as
$\pi^\pm= S_\epsilon \Pi^\pm S^{-1}_\epsilon$ with
$$\Pi^+ = \begin{pmatrix}
1 & 0 & 0 & 0 \\
0 & 0 & 0 & 0\\
0 & 0 & 1 & 0\\
0 & 0 & 0 & 0
\end{pmatrix} \qquad \Pi^- = \begin{pmatrix}
0 & 0 & 0 & 0 \\
0 & 1 & 0 & 0\\
0 & 0 & 0 & 0\\
0 & 0 & 0 & 1
\end{pmatrix} \,.
$$
So let define 
$Q_{\pm}= (\p_{t} \pm \i \epsilon_0)\oplus (\p_{t} \pm \i \epsilon_1)$, considered as an operator acting on $\M\times \Sigma$ on the first group of variables and let $\U(t, x, x')$ the distributional kernel of $\U_{1} \pi^{\pm}$. 
From Proposition~\ref{prop:cauchy ev diag} it follows that $Q_{\pm} \U_1\in \Gamma(\M\times \Sigma, L(\V_{1}\otimes\CC^{2}, \V_{1}))$. If $Q_{\pm}$ were classical $\Psi$DOs on $\M\times \Sigma$, this would imply that $\WF'(\U_1 c^\pm)\subset \mathcal N^{\pm}\times \T^{*}\Sigma$ by elliptic regularity. We reduce ourselves to this situation by an argument from \cite[Lemma~6.5.5]{DH}, see for example \cite[Proposition~11.3.2]{Gerard2019} for details.
\end{itemize}

\noindent Verified points (i)-(iv), our claim follows by Proposition~\ref{prop:hadafromc}.
\end{proof}
\section{Existence of Hadamard States in Globally Hyperbolic Spacetimes }\label{Sec:Finale}
We are finally in the position to prove the existence of Hadamard states on globally hyperbolic manifolds for the CCR-algebra associated to the (gauge-invariant) {\em compactly supported observables}. We begin by recasting the quantization scheme for this class of observables. To this end, we will benefit from~\cite{HackSchenkel2013,GerardWrochnaYM2015}.
\subsection{The Classical Theory and its Phase Space}\label{sec:classical phase space}
Adopting the notation of the Section~\ref{Sec:PhaseSpaceQuant}, the phase space can be characterized as follow.

\begin{proposition}\label{pro:PhaseSpace}
	Let $(\M,g)$ be a globally hyperbolic manifold. Then, the following diagram is commutative and every map is an isomorphism:
    \begin{equation*}
        \begin{tikzcd}
            \PS\defeq\cfrac{\ker(\K^{\ast}\vert_{\Gamma_{\c}})}{\ran(\P\vert_{\Gamma_{\c}})} \arrow[r,"{[\G_{1}]}"]\arrow[dr,"{[\G_{1}]}"]\arrow[d,swap,"{[\rho_1\G_{1}]}"] &\cfrac{\ker(\P\vert_{\Gamma_{\sc}})}{\ran(\K\vert_{\Gamma_{\sc}})}\\             
  \PSS\defeq\cfrac{\ker(\K_\Sigma^{\dagger}\vert_{\Gamma_{\c}})}{\ran(\K_\Sigma\vert_{\Gamma_{\c}})} \arrow[r,"{[\U_{1}]}"]           & \cfrac{\ker(\D_{1}\vert_{\Gamma_{\sc}})\cap\ker(\K^{\ast}\vert_{\Gamma_{\sc}})}{\K(\ker(\D_{0}\vert_{\Gamma_{\sc}}))}\arrow[u,hookrightarrow,swap]
        \end{tikzcd}
    \end{equation*}
	 Furthermore, $(\PSS,\q_{1,\Sigma})$ is a well-defined Hermitian vector space and the following map is unitary:
	\begin{align*}
		[\rho_{1}\G_{1}]\:\Big(\PS,\q_1\Big) \to \Big(\PSS,\q_{1,\Sigma}\Big)  \, ,
	\end{align*}
	where $\q_{1}([\cdot],[\cdot])=\langle\cdot,\i\G_{1}\cdot\rangle_{\V_{1}}$ and $\q_{1,\Sigma}([\cdot],[\cdot])=\langle\cdot,\i\G_{1,\Sigma}\cdot\rangle_{\V_{\rho_{1}}}$.
\end{proposition}
\subsection{Quantization and Hadamard States}\label{sec:quantiz}
The quantization of a free field gauge
theory is realized as a two-step procedure. First, one assigns to the classical phase space $(\PS,\q_1)$  a unital $*$-algebra $\CCR$ abstractly generated by symbols $\Id,\Phi(v),\Phi^{\ast}(w)$ for all $v,w\in\PS$, such that the assignment $v\mapsto\Phi(v)$  and $w\mapsto\Phi^{\ast}(v)$ are respectively $\mathbb{C}$-anti-linear and $\mathbb{C}$-linear , such that $\Phi(v)^{\ast}=\Phi^{\ast}(v)$ and such that the \textit{canonical commutation relations} are fulfilled:
\begin{align*}
	&[\Phi(v),\Phi(w)]=[\Phi^{\ast}(v),\Phi^{\ast}(w)]=0\, ,\\
	&[\Phi(v),\Phi^{\ast}(w)]=\q_1(v,w)\Id \,.
\end{align*}
To conclude the quantization, one has to constructs quasifree Hadamard states. 
As before, any quasifree state is fully
determined by a pair of {\it (spacetime) covariances}, i.e.~Hermitian forms on $\PS$ defined by 
\begin{align*}
	\Lambda^{+}(v,w)\defeq\omega(\Phi(v)\Phi^{\ast}(w))\hspace*{1cm}\text{and}\hspace*{1cm}\Lambda^{-}(v,w)\defeq\omega(\Phi^{\ast}(w)\Phi(v))\,.
\end{align*}
for all $v,w\in\PS$. Following \cite{Gerard2019,GerardWrochnaYM2015} we get the following characterization:
 
\begin{proposition}
Consider  a pair of \textup{pseudo-covariances} $\lambda^\pm$, namely continuous linear maps $\lambda^{\pm}\:\Gamma_{\c}(\V_{1})\to\Gamma(\V_{1})$ satisfying
\begin{itemize}
\item[(i)] \quad $(\lambda^{\pm})^{\ast}=\lambda^{\pm}$ \, w.r.t.\, $(\cdot,\cdot)_{\V_{1}}$\, and \, $\lambda^{\pm}\:\ran(\K\vert_{\Gamma_{\c}})\to \ran(\K)$;
\item[(ii)]\quad $\D_{1}\lambda^{\pm}=\lambda^{\pm}\D_{1}=0$   and  $(\lambda^{+}-\lambda^{-})s=\i\G_{1}s$\, modulo \, $\ran(\K)$ \, for any $ s\in \ker(\K^{\ast}\vert_{\Gamma_{\c}})$;
\item[(iii)]\quad $( s,\lambda^{\pm}s)_{\V_{1}}\geq 0$\,  for any  $ s\in\ker(\K^{\ast}\vert_{\Gamma_{\c}})$;
\item[(iv)] \quad $\WF^{\prime}(\lambda^{\pm})\subset\mathcal{N}^{\pm}\times\mathcal{N}^{\pm}\, $.
\end{itemize}
Then, the Hermitian forms $\Lambda^{\pm}\:\PS\times\PS\to\CC$ defined by
\begin{align*}
	\Lambda^{\pm}([s],[t])\defeq( s,\lambda^{\pm}t)_{\V_{1}} \qquad \forall s,t\in\ker(\K^{\ast}\vert_{\Gamma_{\c}})
\end{align*}
are the spacetime covariances of a quasifree Hadamard state on $\CCR$.
 \end{proposition}

 Working at the level of initial data, we get the following characterization:
 
\begin{proposition}\label{prop:hadafromC}
Suppose $c^\pm:\Gamma_{\c}(\V_{1})\to\Gamma(\V_{1})$ are continuous linear operators satisfying
\begin{itemize}
\item[(i)] \quad $( c^\pm)^{\dagger}=c^{\pm}$\, w.r.t. \,$\q_{1,\Sigma}$\,\,  and \,  $ c^\pm(\ran(\K_\Sigma\vert_{\Gamma_{\c}}))\subset\ran(\K_\Sigma)$;
\item[(ii)] \quad $(c^{+}+ c^{-})\f=\f$ \, modulo \, $\ran(\K_\Sigma)$ \, for any $ \f\in\ker(\K_{\Sigma}^{\dagger}\vert_{\Gamma_{\c}})$;
\item[(iii)]\quad $\pm\q_{1,\Sigma}(\f,c^{\pm}\f)\geq 0$ \, for any $\f\in\ker(\K_{\Sigma}^{\dagger}\vert_{\Gamma_{\c}})$;
\item[(iv)] \quad $ \WF'(\mathcal{U}_{1} c_{1}^{\pm})\subset (\mathcal{N}^{\pm}\cup F)\times\T^{*}\Sigma \; \text{ over $U\times \Sigma$
where $F\subset \T^{*}\M$ is a conic set s.t.~$F\cap \mathcal{N}= \emptyset$}\,.$ 
\end{itemize}
Then $\lambda^{\pm}\defeq \pm i\U_{1} c^{\pm}(\rho_{1}\G_{1})$
are pseudo-covariances of a quasifree Hadamard state on $\CCR$. 
 \end{proposition}
\subsection{Proof of the Main Theorem}
We are finally to prove the existence of Hadamard states on any globally hyperbolic spacetimes for Maxwell fields. To this end, we follow~\cite{MorettiMurroVolpe2023,MollerMIT} but we will also benefit from~\cite{GerardWrochnaYM2015,
MorettiMurroVolpe2022,defSHS,FPmoller}.
 
\begin{proof}[Proof of Theorem~\ref{thm:main}.]
We split the proof in three separated steps: 
\begin{itemize}
\item[1.] Let $(\M,g)$ be an ultrastatic globally hyperbolic manifolds such that $\Sigma$ is of bounded geometry. Consider the operators $\overline{c}^{\pm}\:\Gamma_\infty(\V_{\rho_{1}})\to\Gamma_\infty(\V_{\rho_{1}})$ defined in Theorem~\ref{thm:hada}.
On account of formula~\eqref{eq:KSigma}, then it follows immediately that for any $h\in\Gamma_\c(\V_{\rho_1})$ it holds $\K_\Sigma h \in \Gamma_{\infty,\d}(\V_{\rho_1})$. Furthermore, since $\ker(\K_\Sigma^\dagger|_{\Gamma_\c})\subset \ker(\K_\Sigma^\dagger|_{\Gamma_\infty})$, we can conclude that $c^\pm\defeq \overline{c}|_{\Gamma_\c}$ satisfies the hypothesis (i)-(iv) of Proposition~\ref{prop:hadafromC}.
\item[2.] Consider two globally hyperbolic metrics $g_0$ and $g_1$ on $\M$ such that $g_0\preceq g_1$, i.e. the open lightcone $V_p^{g_0}$ is contained in the open lightcone $V_p^{g_1}$ for any $p\in\M$.
Let be $\chi\in C^\infty(\M; [0,1])$  and let $g_\chi$ be the Lorentzian metric on $\M$ given by 
\begin{equation}g_\chi \defeq (1-\chi) g_0 + \chi  g_1\, .\label{gCHI}
\end{equation}
As shown in ~\cite[Section 2]{MorettiMurroVolpe2023}, $g_\chi$ is globally hyperbolic and satisfies $g_0\preceq g_\chi \preceq g_1$.
Consider now two Cauchy hypersurfaces {$\Sigma_0,\Sigma_1\subset (\M,g_\chi)$} such that
$\Sigma_1 \subset J_{g_\chi}^+(\Sigma_0)$  and
$$\chi_{|_{J_{g_\chi}^+(\Sigma_1)}}=1 \,,\qquad \text{and } \qquad 
\chi_{|_{J_{g_\chi}^-(\Sigma_0)}}=0 \, ,$$ 
and consider the spaces ${\PSS}_i$ defined by
$${\PSS}_i \defeq 	\frac{\ker(\K^\dagger_{\Sigma_i}|_{\Gamma_\c})}{ \ran(\K_{\Sigma_i}|_{\Gamma_\c})} \qquad \text{ for i=0,1}.$$
The operator $R\defeq\rho_1\U_\chi$ obtained from composing the Cauchy evolution operator $\D_\chi$ with the Cauchy data map $\rho_1$ on $\Sigma_1$ implements an unitary isomorphism between $({\PSS}_0, \q_{1,\Sigma_0})$ and $({\PSR}_1, \q_{1,\Sigma_1})$.
Therefore, given two operators $\c_0^\pm$ with domain given by $({\PSR}_0, \q_{1,\Sigma_0})$, we can define the operators $c_1^\pm\defeq(R^{-1})^\dagger c^\pm_0 R^{-1}$. A routine computations shows that if $c^\pm_0$ satisfy property (i)-(iv) of Theorem~\ref{thm:hada}, then also $c^\pm_1$ do.
Similarly, given two operators $\c_1^\pm$ with domain given by $({\PSR}_1, \q_{1,\Sigma_1})$, we can define the operators $c_0^\pm\defeq R^\dagger c^\pm_1 R$. A routine computations shows that if $c^\pm_1$ satisfy property (i)-(iv) of Theorem~\ref{thm:hada}, then also $c^\pm_0$ do.
\item[3.] We now consider a family of metric $\{g_i\}_{i\in[0,1]}$ with the following properties:
\begin{align*}
	\text{(I) }&\quad g_{i}\text{ are globally hyperbolic}\\
	\text{(II) }&\quad g_i\preceq g_{i+1}\text{ or }g_{i+1}\preceq g_{i}\\
	\text{(III) }&\quad g_{i=1}=g_1\text{ and }g_{i=0}=g_0\text{ with $g_{0}$ ultrastatic s.t.~$(\Sigma,h_0)$ is of bounded geometry}
\end{align*}
The existence of such a family of metrics was proven in~\cite[Section 5]{MorettiMurroVolpe2022} and the two metric were called \textit{paracausally related}. Since $g_0$ is ultrastatic and $h_0$ is of bounded geometry, by part 1.~of this proof, we now that there exists $c_0^\pm$ satisfying property (i)-(iv) of Proposition~\ref{prop:hadafromC}. Therefore, it is enough to iterate step 2 to define two operators $c_1^\pm$ with domain ${\PSS}_1$ enjoying the same properties.  
\end{itemize}
This concludes the proof.
	\end{proof}
\bibliographystyle{abbrv}

\end{document}